\documentclass[11pt, leqno]{amsart}
\usepackage {amssymb}
\usepackage {amsmath}
\usepackage {bbm}
\usepackage{amsthm}
\usepackage{mathtools}
\usepackage{graphicx}
\usepackage {amscd}
\usepackage {epic}
\usepackage[mathscr]{euscript}
\usepackage{comment}

\usepackage{geometry}
\usepackage{bm}

\usepackage[dvipsnames]{xcolor}
\usepackage[colorlinks,citecolor=OliveGreen,linkcolor=Mahogany,urlcolor=Plum,pagebackref]{hyperref}
\usepackage[alphabetic]{amsrefs}

\usepackage{todonotes}

\theoremstyle{plain}
\newtheorem{theorem}{Theorem}[section]
\newtheorem{lemma}[theorem]{Lemma}
\newtheorem{proposition}[theorem]{Proposition}
\newtheorem{corollary}[theorem]{Corollary}

\newtheorem{question}[theorem]{Question}

\theoremstyle{definition}
\newtheorem{example}[theorem]{Example}
\newtheorem{definition}[theorem]{Definition}

\newtheorem{deflem}[theorem]{Definition-Lemma}
\newtheorem{defn}[theorem]{Definition}

\theoremstyle{remark}
\newtheorem{remark}[theorem]{Remark}

\newcommand{\cF}{\mathcal{F}}

\newcommand{\cE}{\mathcal{E}}

\newcommand{\cO}{\mathcal{O}}
\newcommand{\cN}{\mathcal{N}}
\newcommand{\cM}{\mathcal{M}}

\newcommand{\cI}{\mathcal{I}}
\newcommand{\cJ}{\mathcal{J}}

\newcommand{\cP}{\mathcal{P}}

\newcommand{\cS}{\mathcal{S}}
\newcommand{\cH}{\mathcal{H}}

\newcommand{\bC}{\mathbb{C}}

\newcommand{\bQ}{\mathbb{Q}}

\newcommand{\bR}{\mathbb{R}}

\newcommand{\bZ}{\mathbb{Z}}

\newcommand{\fa}{\mathfrak{a}}
\newcommand{\fb}{\mathfrak{b}}
\newcommand{\fc}{\mathfrak{c}}

\newcommand{\fm}{\mathfrak{m}}
\newcommand{\fn}{\mathfrak{n}}

\newcommand{\ev}{\mathrm{ev}}
\newcommand{\ic}{\mathrm{ic}}

\newcommand{\an}{^\mathrm{an}}

\newcommand{\supp}{\mathrm{supp}}

\newcommand{\vol}{\mathrm{vol}}
\newcommand{\ord}{\mathrm{ord}}

\newcommand{\lct}{\mathrm{lct}}
\newcommand{\Int}{\mathrm{Int}}
\newcommand{\Val}{\mathrm{Val}}
\newcommand{\DivVal}{\mathrm{DivVal}}
\newcommand{\nvol}{\widehat{\mathrm{vol}}}

\newcommand{\Spec}{\mathrm{Spec}}

\newcommand{\gr}{\mathrm{gr}}
\newcommand{\supj}{^{(j)}}
\newcommand{\supl}{^{(l)}}

\newcommand{\loc}{\mathrm{loc}}
\newcommand{\QM}{\mathrm{QM}}
\newcommand{\bk}{\mathbbm{k}}

\newcommand{\e}{\mathrm{e}}

\newcommand{\Bl}{\mathrm{Bl}}

\newcommand{\R}{\mathbb{R}}
\newcommand{\la}{\lambda}

\newcommand{\QMVal}{\mathrm{QMVal}}

\newcommand{\RV}{\mathrm{RV}}

\newcommand{\Fil}{\mathrm{Fil}}

\newcommand{\covol}{\mathrm{covol}}

\numberwithin{equation}{section}

\begin{document}
	
    \title{On the geometry of spaces of filtrations on local rings}
	

    \author{Lu Qi}\footnote{This work was partially supported by the NSF FRG grant DMS-2139613, the NSF grant DMS-2201349 of Chenyang Xu and a Shanghai Sailing program 24YF2709800.}
    \address{School of Mathematical Sciences, East China Normal University, Shanghai 200241, China}
    
    \address{Department of Mathematics, Princeton University, Princeton, NJ 08544, USA.}
    
    \email{lqi@math.ecnu.edu.cn}
	
    \begin{abstract}
        We study the geometry of spaces of filtrations on a Noetherian local domain. We introduce a metric $d_1$ on the space of saturated filtrations, inspired by the Darvas metric in complex geometry, such that it is a geodesic metric space. In the toric case, using Newton-Okounkov bodies, we identify the space of saturated monomial filtrations with a subspace of $L^1_\loc$.
        We also consider several other topologies on such spaces and study the semi-continuity of the log canonical threshold function in the spirit of Demailly-Koll\'ar.
        Moreover, there is a natural lattice structure on the space of saturated filtrations, which is a generalization of the classical result that the ideals of a ring form a lattice.
    \end{abstract}

    \maketitle
	
    \setcounter{tocdepth}{1}
    \tableofcontents

\section{Introduction}

    Throughout this paper, we work with a Noetherian local domain $(R,\fm,\kappa)$ which is analytically irreducible, that is, the $\fm$-adic completion $\hat{R}$ is a domain. 

    As a generalization of ideals, graded sequences of ideals have been extensively studied in the past few years, and have played an important role in many areas of algebraic geometry; see for example \cites{ELS03,JM12,Cut13} and \cites{BFJ08,BdFF12,BFJ14,Li18,Blu18,Liu18,Xu20,XZ20b,LXZ22}.
    
    Filtrations on a local ring $(R,\fm)$ are the continuously indexed version of graded sequences of ideals. To be more precise, an $\fm$-filtration $\fa_\bullet$ on $R$ is a collection $\{\fa_\lambda\}_{\lambda\in\bR_{>0}}$ of $\fm$-primary ideals, which is decreasing, multiplicative, and left continuous. The three conditions mean $\fa_\lambda\subset\fa_\mu$ for $\lambda>\mu$, $\fa_\lambda\cdot\fa_\mu\subset\fa_{\lambda+\mu}$ and $\fa_{\lambda-\epsilon}=\fa_\lambda$ for $0<\epsilon\ll 1$, respectively. 
    
    In this paper, we study the geometry of certain spaces of such filtrations and prove some structural results. The first main result is that there is a pseudometric $d_1$ on the set of $\fm$-filtrations with positive multiplicity, which is an analogue of the Darvas metric introduced in \cite{Dar15} in complex geometry (see \cite{BJ21} for the non-archimedean setting), such that under a suitable retraction, the set of saturated filtrations is a geodesic metric space. In the toric setting, the subspace of saturated monomial filtrations can be identified with the space of cobounded convex sets of the dual cone with the symmetric difference metric, which is naturally a subspace of $L^1_\loc$. 
    We also introduce another metric, denoted by $d_\infty$, and several weak topologies, and study the continuity properties of the log canonical threshold function, following Demailly-Koll\'ar \cite{DK01}. 
    Several spaces of filtrations have a natural structure of a lattice, generalizing a classical result that the ideals of a ring form a modular lattice.

\subsection{The metric geometry of space of filtrations}\label{ssec:metric geometry}
	
\subsubsection{The Darvas metric, multiplicities and saturations}\label{sssec:d_1 intro}
	
Given a compact K\"ahler manifold $(X,\omega)$, Darvas \cite{Dar15} introduced various Finsler metrics on the space $\cH$ of smooth K\"ahler potentials, generalizing the metric studied in \cite{Mab87}. One of the metrics, known as the \emph{Darvas metric} $d_1$, has found numerous applications in complex geometry. In particular, the geodesic metric space $(\cH,d_1)$ and its completion, $(\cE^1,d_1)$, the space of K\"ahler potentials of finite energy, are closely related to K\"ahler-Einstein metrics. See, for example, \cites{DR17,BBEGZ19,DL20}. A comprehensive survey on this metric space and its applications, particularly in relation to Tian’s Properness Conjecture, can be found in \cite{Dar19}.
	
More recently, Boucksom and Jonsson \cites{BJ21} developed a similar construction in the non-archimedean setting, where they introduced a pseudometric $d_1$, also referred to as the Darvas metric, on the space $\cN$ of 
norms on the section ring $R(X,L)$ of a polarized variety $(X,L)$. Geodesics on this space have been independently studied by Reboulet \cite{Reb22} and by Blum, Liu, Xu and Zhuang \cites{BLXZ21}, and these studies have found successful applications in the algebraic theory of $K$-stability and  K\"ahler-Ricci soliton degenerations. For some related constructions, see also \cites{Reb20,Wu22,Fin23}.
	
Our primary motivation is to introduce an analogue of the aforementioned constructions on the space of 
filtrations on a local ring $(R,\fm)$. 
	
	
Given an $\fm$-filtration, following Ein, Lazarsfeld, and Smith \cite{ELS03}, the \emph{multiplicity} of $\fa_\bullet$ is defined to be
\[
    \e(\fa_\bullet)\coloneqq 
    \lim_{ \bZ_{>0} \ni m\to\infty}
    \frac{\ell(R/\fa_m)}{m^n/n!} 
    =\lim_{ \bZ_{>0} \ni m\to\infty} 
    \frac{\e(\fa_m)}{m^n} ,
\]
where the existence of the above limit and the equality were proven in increasing generality by \cites{ELS03,Mus02,LM09,Cut13,Cut14}. 
This invariant is the local counterpart of the volume of a graded linear series of a line bundle.
Denote the space of all $\fm$-filtrations on $R$ with positive multiplicity by $\Fil_{R,\fm}$, and define a function $d_1:\Fil_{R,\fm}\times\Fil_{R,\fm}\to\bR_{\ge 0}$ by    
\begin{equation}\label{eqn:def of d_1}
    d_1(\fa_\bullet,\fb_\bullet)\coloneqq
    2\e(\fa_\bullet\cap\fb_\bullet)-\e(\fa_\bullet)-\e(\fb_\bullet),
\end{equation}
where $\fa_\bullet\cap\fb_\bullet$ is defined termwise. \eqref{eqn:def of d_1} is a similar expression to the Darvas metrics in the global case \cites{Dar15,BJ21}.

An issue is that we can have $\e(\fa_\bullet)=\e(\fb_\bullet)$ for $\fa_\bullet\subset \fb_\bullet$, which means that $d_1$ is not a metric on $\Fil_{R,\fm}$. An observation of \cite{BLQ22} is that this issue can be resolved by considering \emph{saturated filtrations}, that is, $\fa_\bullet=\widetilde{\fa_\bullet}$, where $\widetilde{\fa_\bullet}$ is the \emph{saturation} of $\fa_\bullet$ defined by
\[
    \widetilde\fa_\lambda\coloneqq \{f\in \fm \mid v(f)\ge \lambda\cdot v(\fa_\bullet) \text{ for any divisorial valuation $v$ on $R$ centered at }\fm\}
\]
for any $\lambda\in\bR_{>0}$. Denote by $\Fil^s_{R,\fm}\subset \Fil_{R,\fm}$ the subset consisting of saturated filtrations. Note that there is a canonical element $\widetilde{\fm^\bullet}\in\Fil^s_{R,\fm}$, where $\fm^\bullet\coloneqq\{\fm^{\lfloor \lambda \rfloor}\}_{\lambda\in\bR_{>0}}$ is the $\fm$-adic filtration.
	
Given $\fa_{\bullet,0},\fa_{\bullet,1}\in\Fil_{R,\fm}$, the \emph{geodesic} $\fa_{\bullet,t}$  between them, introduced in \cites{XZ20b,BLQ22}, is a segment of $\fm$-filtrations for $t\in[0,1]$, defined by
\begin{equation}\label{eqn:geodesic}
    \fa_{\lambda,t}\coloneqq  
    \sum_{\mu,\nu\ge 0,\ (1-t)\mu+t\nu=\lambda} \fa_{\mu,0}\cap\fa_{\nu,1}
\end{equation}
for any $\lambda>0$. The definition is a local analogue of the geodesics studied in \cite{BLXZ21}.
	
Our first main result is that the Hausdorff quotient of $(\Fil_{R,\fm},d_1)$ is naturally $(\Fil^s_{R,\fm})$, which is endowed with the structure of a geodesic metric space by the above constructions.

\begin{theorem}\label{thm:d_1 metric}
    Let $(R,\fm)$ be a Noetherian local domain that is analytically irreducible. Denote the set of all linearly bounded $\fm$-filtrations on $R$ by $\Fil_{R,\fm}$ and let $\Fil^s_{R,\fm}\subset\Fil_{R,\fm}$ be the subset consisting of saturated ones. Then
    \begin{enumerate}
        \item the function $d_1$ given by \eqref{eqn:def of d_1} is a pseudometric on $\Fil_{R,\fm}$, 
			
        \item the Hausdorff quotient of $(\Fil_{R,\fm},d_1)$ is $(\Fil^s_{R,\fm},d_1)$, and
			
        \item  $(\Fil^s_{R,\fm},d_1)$ is a convex metric space. Moreover, if $R$ contains a field, then it is a geodesic metric space, and a geodesic between $\fa_{\bullet,0}$ and $\fa_{\bullet,1}$ can be given by $\fa_{\bullet,t}$ for $t\in[0,1]$, where $\fa_{\bullet,t}$ is defined by \eqref{eqn:geodesic}.
    \end{enumerate}
\end{theorem}
	
For more details about saturated filtrations, see section \ref{ssec:saturation}. For definitions related to metric spaces, see section \ref{ssec:metric spaces}. 
	
\subsubsection{The toric case}

In the case where $(R,\fm)$ is the local ring of a toric singularity, we can concretely describe the subspace $\Fil^{s,\mathrm{mon}}_{R,\fm}\subset \Fil^s_{R,\fm}$ of saturated monomial filtrations.

Let $\bk$ be an algebraically closed field of characteristic $0$. Let $N$ be a free abelian group of rank $n\ge 1$ and $M=N^*$ its dual. Let $\sigma\subset N_\bR=N\otimes_\bZ \bR$ be a strongly convex rational polyhedral cone of maximal dimension. We have an affine normal toric variety $X_\sigma=\Spec R_\sigma=\Spec \bk[\sigma^\vee\cap M]$ with a unique torus invariant point $x$, where $\sigma^\vee\subset M_\bR=M\otimes_\bZ \bR$ is the dual cone of $\sigma$. Let $R$ be the local ring of $X$ at $x$ and $\fm$ its maximal ideal.

Let $\Fil^{s,\mathrm{mon}}_{R,\fm}\subset\Fil^s_{R,\fm}$  be the subset of all monomial filtrations (i.e. filtrations $\fa_\bullet$ such that every $\fa_\lambda$ is a monomial ideal). Via the Okounkov body construction, we identify the metric space $(\Fil^{s,\mathrm{mon}}_{R,\fm},d_1)$ with a subspace of the space of cobounded subsets of $\sigma^\vee$ with the symmetric difference metric (sometimes also called the Fr\'echet-Nykodym-Aronszajn metric).

\begin{theorem}\label{thm:toric metric}
    If $(R,\fm)$ is the local ring of a toric singularity $x\in X=\Spec R_\sigma$. Then
    \begin{enumerate}
        \item taking Newton-Okounkov body gives an isometry $P:(\Fil^{s,\mathrm{mon}}_{R,\fm},d_1)\to (\cP(\sigma^\vee),d)$,
        and

        \item the metric space $(\Fil^{s,\mathrm{mon}}_{R,\fm},d_1)$ is complete.
    \end{enumerate}
\end{theorem}

In the above theorem, $\cP(\sigma^\vee)$ denotes the set of all closed convex subsets of $\sigma^\vee$ with bounded complement, and the metric $d$ is defined to be the Euclidean volume of the symmetric difference of two subsets. For more details, we refer to Section \ref{ssec:toric}. Note that $P\mapsto \chi_P$ gives an embedding $\cP(\sigma^\vee)\to L^1_\loc(\sigma^\vee)$, where $\chi_P$ is the characteristic function of $P\subset \sigma$.

\subsubsection{The supnorm metric and homogeneous norms}\label{sssec:d_infty intro}
	
As in the global case \cite{BJ21}, there is also a correspondence between filtrations and norms in the local setting. Indeed, given $\fa\in\Fil_{R,\fm}$, we can define a function $\ord_{\fa_\bullet}$ on $R$ by $\ord_{\fa_\bullet}(f)\coloneqq \sup\{\lambda\in\bR\mid f\in \fa_\lambda\}$. In particular, the canonical saturated filtration $\widetilde{\fm^\bullet}$ defines a canonical norm $\chi_0\coloneqq \ord_{\widetilde{\fm^\bullet}}$. 
See Definition \ref{defn:norm} for the definition of norms and Definition-Lemma \ref{deflem:fm-equivalence}, Lemma \ref{lem:fm-equivalence lin bounded} and Lemma \ref{lem:properties of homogeous} for the detailed dictionary. 

Despite the equivalence, it is sometimes more convenient to use one of the languages. 
As an example, we now consider another pseudometric $d_\infty$ on the space $\cN_{R,\fm}$ of $\fm$-norms on $R$, which behaves like the uniform metric on the space of continuous functions on a compact space. This can also be viewed as the local analogue of the $d_\infty$ metric introduced in \cite{BJ21}. 
	
Fix a homogeneous norm $\rho\in\cN^h_{R,\fm}$, define 
\begin{equation}\label{eqn:d_infty metric}
    d_{\infty,\rho}(\chi,\chi')\coloneqq \limsup_{\lambda\to\infty} \sup_{\rho(f)\ge\lambda}\frac{|\chi(f)-\chi'(f)|}{\rho(f)}.
\end{equation}
	
We will denote the function $d_{\infty,\chi_0}$ by $d_\infty$.
\begin{theorem}\label{thm:d_infty metric}
    Let $(R,\fm)$ be a Noetherian local domain that is analytically irreducible. Then
		
    \begin{enumerate}
        \item the function $d_\infty$ is a pseudometric on $\cN_{R,\fm}$, and
			
        \item the restriction of $d_\infty$ to $\cN^h_{R,\fm}$ is a metric, and the metric space $(\cN^h_{R,\fm},d_\infty)$ is complete.
    \end{enumerate}
\end{theorem}

Here $\cN^h_{R,\fm}$ is the subset of $\cN_{R,\fm}$ of homogeneous norms. See Section \ref{ssec:homogeneous} for the definition and more details.

\subsection{Weak topologies and semi-continuity of log canonical thresholds}\label{ssec:lct}

Besides the metrics introduced, we also consider the filtrations, or equivalently, norms, as functions on various spaces, and thus define some weak topologies on the spaces of filtrations. We sketch the ideas here. For more details, see Section \ref{sssec:weak topologies}. 

The first perspective is to view a norm $\chi$ as a function on the ring $R$, and thus define the \emph{weak topology} to be the product topology, that is, the weakest topology such that for any $f\in R$, the function $f\mapsto\chi(f)$ is continuous. Since a real valuation (centered at $\fm$) is a ($\fm$-)norm, this topological space contains the valuation space with the weak topology introduced in \cite{JM12} as a subspace. 

Another point of view is to consider a filtration $\fa_\bullet$ as a function on certain valuation spaces by $v\mapsto v(\fa_\bullet)$, where $v(\fa_\bullet)\coloneqq\lim_{\lambda\to\infty} \frac{v(\fa_\lambda)}{\lambda}$. We call the topology defined this way using the space $\Val^+_{R,\fm}$ of valuations with positive volume the $+$-\emph{topology}.
	
We prove that the log canonical threshold satisfies certain semi-continuity on the space of filtrations, and that it is locally Lipschitz continuous with respect to the $d_\infty$-topology, which is similar to \cite[Theorem 3.3]{DK01}.
	
\begin{theorem}\label{thm:lct}
    Let $(X,x)=(\Spec R,\fm)$ be a klt singularity over a field $\bk$ of characteristic $0$. Then
    \begin{enumerate}
        \item if $\fa_{\bullet,k}\in\Fil$ converges weakly to $\fa_\bullet\in\Fil$, then 
        \[
            \lct(\fa_\bullet)\le \liminf_{k\to\infty}\lct(\fa_{\bullet,k}),
        \]
			
        \item if $\fa_{\bullet,k}\in\Fil^s$ converges to $\fa_\bullet\in\Fil^s$ in the $+$-topology, then
        \[
            \lct(\fa_\bullet)\ge
            \limsup_{k\to\infty}\lct(\fa_{\bullet,k}),
        \]
        and
			
        \item given $\fa_\bullet\in\Fil^s$, for any $\epsilon>0$, there exists $\delta:=\delta(\fa_\bullet,\epsilon)>0$ such that for any $\fb_\bullet$ with $d_\infty(\fa_\bullet,\fb_\bullet)<\delta$, we have 
        \[
            |\lct(\fb_\bullet)-\lct(\fa_\bullet)|\le\epsilon.
        \]
    \end{enumerate}
\end{theorem}

\subsection{Miscellaneous results about filtrations}

A classical result in ring theory is that the semiring of all ideals of a ring $R$ with the partial order by inclusion forms a modular lattice, where meet is given by intersection and join is given by sum. We denote the lattice of all ideals (all $\fm$-primary ideals when $(R,\fm)$ is a local ring, resp.) by $\cI_R$ ($\cI_{R,\fm}$, resp.).

The spaces of filtrations are equipped with a natural partial order by inclusion, that is, $\fa_\bullet\subset\fb_\bullet$ if and only if $\fa_\lambda\subset\fb_\lambda$ for any $\lambda\in\bR_{>0}$. One can similarly define the intersection of two filtrations termwise, and in Definition-Lemma \ref{deflem:join of filtrations} (Definition-Lemma \ref{deflem:saturated join}, resp.), we define the \emph{join} $\vee$ (\emph{saturated join} $\vee_s$, resp.), which is the analogue of the sum of two ideals. The next theorem asserts that these operations define natural lattice structures on the spaces of  $\fm$-filtrations. 

\begin{theorem}\label{thm:lattice}
    Let $(R,\fm)$ be a Noetherian local domain that is analytically irreducible. Then
    \begin{enumerate}
        \item the set of all $\fm$-filtrations with the partial order by inclusion is a lattice, where the meet is given by $\cap$ and the join is given by $\vee$, and there is an injective join morphism from $\cI_{R,\fm}$ to its sub-lattice $(\Fil_{R,\fm},\subset,\cap,\vee)$, 
		
		\item $(\Fil^s_{R,\fm},\subset,\cap,\vee_s)$ is a distributive lattice. Moreover, $\fa\mapsto\widetilde{\fa^\bullet}$ is an injective join morphism from $\cI^{\ic}_{R,\fm}$, the set of integrally closed $\fm$-primary ideals, and the saturation $\Fil_{R,\fm}\to\Fil^s_{R,\fm}$ is a surjective join morphism. 
	\end{enumerate}
\end{theorem}

In the above theorem, $\widetilde{\fa^\bullet}$ denotes the saturation of the \emph{$\fa$-adic filtration} $\fa^\bullet=\{\fa^{\lceil \lambda \rceil}\}$. The above theorem, together with Theorem \ref{thm:d_1 metric}, can be viewed as a higher-dimensional analogue of the results of \cite{FJ04}.

\medskip

We will frequently use the following characterization for saturated filtrations, which generalizes results of \cites{Ree56a,Ree56b}. 

\begin{proposition}[= Proposition \ref{prop:alt def for saturation}]
    A filtration $\fa_\bullet\in\Fil_{R,\fm}$ is saturated if and only if there exists a (non-empty) subset $\Sigma\subset \DivVal_{R,\fm}$ such that 
    \[
        \fa_\bullet=\cap_{v\in \Sigma} \fa_\bullet(v).
    \]
\end{proposition}

We remark that the above proposition can be viewed as an analogue of the Berkovich Maximum Modulus Principle. See also Lemma \ref{lem:alt def for homogeneous}. 

As an application, when $\fa$ is an $\fm$-primary ideal, we prove that the saturation of the $\fa$-adic filtration is determined by the set of Rees valuations of $\fa$, as expected.

\begin{proposition}[= Proposition \ref{prop:saturation and Rees}]
    Let $\fa\in\cI_{R,\fm}$ be an $\fm$-primary ideal. Then 
    \[
        \widetilde{\fa^\bullet}=
        \cap_{v\in \RV(\fa)}\fa_\bullet(\frac{v}{v(\fa)}),
    \]
    where $\RV(\fa)$ is the set of Rees valuations of $\fa$.
\end{proposition}

We refer to \cite[Chapter 10]{HS06} for more details about Rees valuations.

\medskip

\noindent \textbf{Organization of the paper.}
In Section \ref{sec:prelim}, we recall some definitions and basic properties that will be needed. Along the way, we also prove some basic facts about filtrations. 
In Section \ref{sec:d_1}, we introduce the Darvas metric $d_1$ and consider its properties.
In Section \ref{sec:d_infty}, we define the metric(s) $d_\infty$, and compare the different topologies on the space of filtrations. 
In Section \ref{sec:proof} we prove the main results Theorem \ref{thm:d_1 metric} through Theorem \ref{thm:lattice}. 
In Section \ref{sec:discussion}, we discuss the relation of the results in this paper with results in the global case, give several further examples, and propose some related questions.

\bigskip

\noindent \textbf{Notation}.
Throughout the paper, $(R,\fm,\kappa)$ denotes an $n$-dimensional Noetherian local domain which is \emph{analytically irreducible}, that is, the $\fm$-adic completion $\hat{R}$ is again a domain. $\kappa:=R/\fm$ denotes the residue field of $R$. We will always denote the Krull dimension of $R$ by $n$. When $R$ contains a field $\bk$, we sometimes write $X=\Spec R$ and $x\in X$ the closed point corresponding to $\fm$. In this case, we call $x\in X$ a \emph{singularity} over the field $\bk$.

We write $c\coloneqq c(a,\alpha,\ldots)$ to mean that $c\in\bR_{>0}$ is a constant depending only on $a,\alpha\ldots$.

\bigskip

\noindent\textbf{Acknowledgments.} 
    A large portion of this paper is derived from the author's doctoral thesis. The author would like to thank his advisor, Professor Chenyang Xu, for his constant support, encouragement and numerous inspiring conversations. Part of this work is inspired by an earlier collaboration with Harold Blum and Yuchen Liu. The author owes them special thanks for their insights on the topic. 

    The author would like to thank Zhiyuan Chen, Colin Fan, Yujie Luo, Tomasso de Fernex, R\'emi Reboulet, Xiaowei Wang and Ziquan Zhuang for fruitful discussions. The author would also like to thank Tam\'as Darvas, Jingjun Han, Jihao Liu, Minghao Miao, Junyao Peng, Linsheng Wang, Lingyao Xie, Qingyuan Xue, Tong Zhang and Junyan Zhao for kind comments. Special thanks are also due to R\'emi Reboulet and Mattias Jonsson for reviewing a preliminary version of this paper and offering valuable suggestions. Part of this work is done while the author visited Northwestern University, the University of Utah, Fudan University and BICMR at Peking University. The author would like to thank them for their hospitality and the amazing environment they provided.

    We are also grateful to the referees for the corrections and many valuable comments, which improved the quality of this paper.

\section{Preliminaries}\label{sec:prelim}


In this section, we recall definitions and basic properties that will be needed in the paper. Several results in Section \ref{ssec:saturation}, including a characterization of saturated filtrations Proposition \ref{prop:alt def for saturation}, the lattice structure Definition-Lemma \ref{deflem:saturated join} and the distributivity Proposition \ref{prop:distributivity}, seem to be new.

\subsection{Filtrations and norms}\label{ssec:definition}

\subsubsection{Filtrations}\label{sssec:filtrations}

\begin{definition}\label{defn:filtration}
    A \emph{filtration} on $R$ is a collection $\fa_\bullet =(\fa_\la )_{\la \in \R_{>0}}$  of ideals of $R$ such that
    \begin{enumerate}
        \item (decreasing) $\fa_\la \subset \fa_{\mu}$ for any $\la >\mu$,
		
        \item (left-continuous) $\fa_{\la} = \fa_{\la-\epsilon}$ for any $\lambda\in\bR_{>0}$ and $0<\epsilon\ll1$, and
		
        \item (multiplicative) $\fa_{\la} \cdot \fa_{\mu} \subset \fa_{\la+ \mu}$ for any $\lambda,\mu\in\bR_{>0}$.
    \end{enumerate}
    A filtration $\fa_\bullet$ is called an $\fm$-\emph{filtration} if $\fa_\lambda$ is $\fm$-primary for any $\lambda\in\bR_{>0}$.
    By convention, we always set $\fa_{0}\coloneqq R$. 
\end{definition}

The definition is a local analogue of a filtration on the section ring of a polarized variety in \cite{BHJ17}. We will exclusively work with $\fm$-filtrations in the sequel, though some of the results also hold for general filtrations on $R$.

For $\lambda\in\bR_{\ge 0}$, set $\fa_{>\lambda}:=\cup_{\mu>\lambda}\fa_\mu$. If $\lambda\in\R_{>0}$ satisfies $\fa_{>\lambda}\subsetneq \fa_\lambda$, then we call $\lambda$ a \emph{jumping number} of $\fa_\bullet$. The \emph{scaling} of an $\fm$-filtration $\fa_\bullet$ by $c\in \bR_{>0}$ is $\fa_{c\bullet}:=(\fa_{c \la } )_{\la\in \bR_{>0}}$, which is an $\fm$-filtration. For two $\fm$-filtrations $\fa_\bullet$ and $\fb_\bullet$, we say that $\fa_\bullet\subset \fb_\bullet$ if $\fa_\lambda\subset \fb_\lambda$ for any $\lambda\in\bR_{>0}$. This defines a partial order on the set of all $\fm$-filtrations. There is a maximal element among all $\fm$-filtrations, $\{\fm\}_{\lambda\in\bR_{>0}}$. We will sometimes abuse the notation and denote this $\fm$-filtration by $\fm$.

Let $\fa_\bullet$ and $\fb_\bullet$ be two $\fm$-filtrations. Their \emph{intersection} is $\fa_\bullet\cap \fb_\bullet:=(\fa_\lambda\cap \fb_\lambda)_{\lambda\in\bR_{>0}}$. 
It is not hard to check that this is again an $\fm$-filtration. 

\begin{deflem}\label{deflem:arbitrary intersection}
    Let $\{\fa_{\bullet,i}\}_{i\in I}$ be an arbitrary non-empty set of $\fm$-filtrations. Assume that there exists an $\fm$-filtration $\fb_\bullet$ such that $\fb_\bullet\subset\fa_{\bullet,i}$ for any $i\in I$. Then the collection of $\fm$-primary ideals, $(\cap_{i\in I} \fa_{\lambda,i})_{\lambda\in\bR_{>0}}$, is an $\fm$-filtration, called the \emph{intersection} of the $\fa_{\bullet,i}$, and is denoted by $\cap_{i\in I}\fa_{\bullet,i}$. Moreover, if an $\fm$-filtration $\fc_\bullet$ satisfies $\fc_\bullet\subset\fa_{\bullet,i}$ for any $i\in I$, then we have $\fc_\bullet\subset\cap_{i\in I}\fa_{\bullet,i}$.
\end{deflem}

\begin{proof}
    By assumption, for any $\lambda>0$ we have $\fb_\lambda\subset\fa_{\lambda,i}$ and hence $\cap_{i\in I} \fa_{\lambda,i}$ is $\fm$-primary. It is straightforward to verify conditions (1) and (3) in the definition.
	
    To prove (2), note that $\fb_\lambda\subset\cap_i\fa_{\lambda,i}\subset\cap_i\fa_{\mu,i}$ for any $\mu\le \lambda$. In $R/\fb_\lambda$ we have
    \begin{align*}
        \cap_{i\in I}\fa_{\lambda,i}=&\cap_{i\in I}\cap_{\epsilon>0}\fa_{\lambda-\epsilon,i}=\cap_{\epsilon>0}\cap_{i\in I} \fa_{\lambda-\epsilon,i}\\
        =&\cap_{l\in\bZ_{>1/\lambda}} \cap_{i\in I} \fa_{\lambda-\frac{1}{l},i}=\cap_{i\in I} \fa_{\lambda-\frac{1}{l},i}
    \end{align*}
    for some $l\in\bZ_{>1/\lambda}$, where the first equality follows from the left-continuity of $\fa_{\bullet,i}$, and the last equality uses the fact that $R/\fb_\lambda$ is Artinian. This proves that condition (2) holds for $\{\cap_{i\in I}\fa_{\lambda,i}\}$, and thus it is an $\fm$-filtration.
	
    The last assertion follows by definition, and the proof is finished.
\end{proof}

\begin{deflem}\label{deflem:join of filtrations}
    Let $\{\fa_{\bullet,i}\}_{i\in I}$ be an arbitrary set of $\fm$-filtrations. Then there is a filtration, called the \emph{join} of the $\fa_{\bullet,i}$, denoted by $\vee_{i\in I}\fa_{\bullet,i}$, such that  $\fa_{\bullet,i}\subset\vee_{i\in I}\fa_{\bullet,i}$ for any $i\in I$. Moreover, if an $\fm$-filtration $\fc_\bullet$ satisfies $\fa_{\bullet,i}\subset\fc_\bullet$ for any $i\in I$, then we have $\vee_{i\in I}\fa_{\bullet,i}\subset\fc_\bullet$.
\end{deflem}

\begin{proof}
    If $I=\emptyset$, then $\vee_{i\in I}\fa_{\bullet,i}$ is the maximal $\fm$-filtration $\fm$. 
	
    Assume now $I\ne\emptyset$ and fix $0\in I$. The set
    \[
        J:=\{\fb_\bullet\mid \fa_{\bullet,i}\subset \fb_\bullet \text{ for any } i\in I\}
    \]
    is non-empty, since the maximal $\fm$-filtration $\fm\in J$. Moreover, by definition $\fa_{\bullet,0}\subset\fb_\bullet$ for any $\fb_\bullet\in J$. Hence we can apply Definition-Lemma \ref{deflem:arbitrary intersection} to get an $\fm$-filtration
    \[
        \vee_{i\in I}\fa_{\bullet,i}\coloneqq \cap_{\fb_\bullet\in J}\fb_\bullet.
    \]
    
    The last assertion follows by definition, and the proof is finished.
\end{proof}

\begin{example}\label{eg:ideal join}
    Given an $\fm$-primary ideal $\fa\in\cI_\fm$, we have an $\fm$-filtration $\fa^\bullet$, the $\fa$-\emph{adic filtration}, defined by $\{\fa^{\lceil\lambda\rceil}\}_{\lambda\in\bR_{>0}}$. Thus we get an injection from the set $\cI_\fm$ to the set of all $\fm$-filtrations defined by $\fa\mapsto\fa^\bullet$. In particular, there is a canonical $\fm$-filtration $\fm^\bullet$. 
	
    For $\fa,\fb\in\cI_\fm$, we claim that $\fa^\bullet\vee\fb^\bullet=(\fa+\fb)^\bullet$. Indeed, since $\fa^\bullet\subset (\fa+\fb)^\bullet$ and $\fb^\bullet\subset (\fa+\fb)^\bullet$, by Definition-Lemma \ref{deflem:join of filtrations} we have $\fc_\bullet\coloneqq\fa^\bullet\vee\fb^\bullet\subset (\fa+\fb)^\bullet$. On the other hand, we have $\fa\subset\fc_1$ and $\fb\subset\fc_1$ by definition. Hence for any $\lambda\in\bR_{>0}$, we have 
    \[
        (\fa+\fb)^{\lceil\lambda\rceil}\subset\fc_1^{\lceil\lambda\rceil}\subset\fc_{\lceil\lambda\rceil}\subset\fc_\lambda.
    \]
    This proves $(\fa+\fb)^\bullet\subset\fc_\bullet$ and thus the equality holds.
	
    However, in general we only have $(\fa\cap\fb)^\bullet\subset\fa^\bullet\cap\fb^\bullet$ and the inclusion can be strict. Indeed, let $R=\bk[\![x,y]\!]$ for some field $\bk$, $\fa=(x,y^2)$ and $\fb=(x^2,y)$. Then it is not hard to see that $\fa\cap\fb=(x^2,xy,y^2)=\fm^2$ and $(\fa\cap\fb)^2=\fm^4$. Hence $x^2y\in\fa^2\cap\fb^2\backslash (\fa\cap\fb)^2$.
\end{example}

Unlike the intersection, it can be hard to determine the join of two general filtrations. In Section \ref{ssec:saturation}, we will define a similar notion, the \emph{saturated join}, and give a formula for it.
The following proposition, which is a combination of Definition-Lemma \ref{deflem:arbitrary intersection} and Definition-Lemma \ref{deflem:join of filtrations}, justifies the term \emph{join}.

\begin{proposition}\label{prop:lattice}
    The set of all $\fm$-filtrations with the partial order by inclusion forms a lattice, where the meet of two elements is given by their intersection, and the join is defined as in Definition-Lemma \ref{deflem:join of filtrations}. \qed
\end{proposition}

\begin{remark}
    Note that by Example \ref{eg:ideal join}, the lattice $\cI_\fm$ of $\fm$-primary ideals of $R$ does not embed in the lattice of all $\fm$-filtrations, since the meet is not preserved.
\end{remark}

\subsubsection{Norms and valuations}\label{sssec:norms}

We recall the correspondence between filtrations and norms on a local ring. See also \cites{HS06,BFJ14} for some related results.

\begin{definition}\label{defn:norm}
    A \emph{seminorm} on $R$ is a function $\chi:R\to \bR_{\ge 0}\cup\{+\infty\}$ such that 
    \begin{enumerate}
        \item $\chi(0)=+\infty$ and $\chi(-f)=\chi(f)$, 
		
        \item $\chi(f+g)\ge \min\{\chi(f),\chi(g)\}$ for any $f,g\in R$, and
		
        \item(Sub-multiplicativity) $\chi(fg)\ge \chi(f)+\chi(g)$ for any $f,g\in R$. 
    \end{enumerate} 
	
    A seminorm $\chi$ is a \emph{norm} if $\chi(f)=+\infty$ if and only if $f=0$. A (semi)norm $\chi$ is called an $\fm$-(semi)\emph{norm} if it further satisfies
    \begin{enumerate}
        \item[(4)] $\chi(f)\ge 0$ for any $f\in R$ and $\chi(f)>0$ if and only if $f\in\fm$. 
    \end{enumerate}
	
    If $\chi$ is a ($\fm$-)norm on $R$, and we replace condition (3) above by
    \begin{enumerate}
        \item[(3$^\prime$)] (Multiplicativity) $\chi(fg)=\chi(f)+\chi(g)$ for any $f,g\in R$,
    \end{enumerate}
    then $\chi$ is a real valuation on $R$ (centered at $\fm$). Denote the space of real valuations on $R$ centered at $\fm$ by $\Val_{R,\fm}$. 
\end{definition}

If $R$ is a $\bk$-algebra for a field $\bk$, then we will always assume that $\chi$ is a $\bk$-norm, that is, $\chi(af)=\chi(f)$ for any $a\in \bk^\times$ and $f\in R$. 

Here, as in the global case, we adopt the additive terminology here. For a norm $\chi$ on $R$, $\|\cdot\|_\chi:=e^{-\chi(\cdot)}$ defines a non-Archimedean norm on $R$ in the sense of \cite{BGR84}.

There is a correspondence between seminorms and filtrations, given by the lemma below.

\begin{deflem}\label{deflem:fm-equivalence}
    Given an $\fm$-filtration $\fa_\bullet$, its \emph{order function} $\ord_{\fa_\bullet}$ is defined by 
    \begin{equation}\label{eqn:order}
        \ord_{\fa_{\bullet}}(f):=\sup \{\lambda\in\bR_{>0}\mid f\in\fa_\lambda\},
    \end{equation}
    which is a seminorm on $R$. If $\fa_\lambda=\fm$ for $0<\lambda\ll 1$, then it is an $\fm$-seminorm.
	
    Conversely, given an $\fm$-seminorm $\chi$, define $\fa_\bullet(\chi)$ by
    \begin{equation}\label{eqn:associated filtration}
        \fa_\lambda(\chi):=\{f\in R\mid \chi(f)\ge \lambda\}.
    \end{equation}
    Then $\fa_\bullet(\chi)$ is an $\fm$-filtration, called the \emph{associated filtration} of $\chi$.
	
    The above maps give a 1-1 correspondence between $\fm$-seminorms and $\fm$-filtrations satisfying $\fa_\lambda=\fm$ for $0<\lambda\ll 1$. Moreover, $\fm$-norms correspond to $\fm$-filtrations with $\cap_{\lambda>0}\fa_\lambda=\{0\}$. 
\end{deflem}

The associated filtration of a real valuation $v\in\Val_{R,\fm}$ is also called the \emph{valuation ideals} of $v$.

\begin{proof}
    Let $\fa_\bullet$ be an $\fm$-filtration. Clearly $\ord_{\fa_\bullet}$ satisfies condition (1) of Definition \ref{defn:norm}. Condition (2) follows from the fact that each $\fa_\lambda$ is an ideal. Condition (3) follows from the multiplicativity of a filtration. If $\fa_\lambda=\fm$ for $0<\lambda\ll 1$, then condition (4) follows from the convention $\fa_0=R$.
	
    Conversely, let $\chi$ be an $\fm$-seminorm. We first show that for any $\lambda>0$, $\fa_\lambda(\chi)$ defined in \eqref{eqn:associated filtration} is an $\fm$-primary ideal. It is an ideal since $\chi(f-g)\ge\min\{\chi(f),\chi(-g)\}=\min\{\chi(f),\chi(g)\}$ and $\chi(fg)\ge\chi(f)+\chi(g)\ge\min\{\chi(f),\chi(g)\}$. Choose a set of generators $f_1,\ldots,f_r$ of $\fm$, then for any $f\in\fm$ we have $\chi(f)\ge \chi(\fm)\coloneqq\min_{1\le i\le r}\{\chi(f_i)\}>0$. Let $C\coloneqq\lceil\lambda/\chi(\fm)\rceil$, then $\chi(f^C)\ge C\chi(f)\ge \lambda$, that is, $\fm^C\subset\fa_\lambda(\chi)$. Note that the above argument also shows that $\fa_\lambda(\chi)=\fm$ for $0<\lambda<\chi(\fm)$. Clearly the collection of $\fm$-primary ideals $\{\fa_\lambda(\chi)\}$ is decreasing, and it is multiplicative by condition (3) of Definition \ref{defn:norm}. To show left continuity, note that we can write $\fa_\lambda(\chi)$ as a decreasing intersection 
    \[
        \fa_\lambda(\chi)=
        \cap_{\epsilon>0}\fa_{\lambda-\epsilon}(\chi),
    \]
    which terminates since the equality holds in the Artinian ring $R/\fa_\lambda(\chi)$.
\end{proof}

Note that the supremum in the definition is indeed a maximum if the value is finite.

\begin{example}
        The canonical $\fm$-filtration $\fm^\bullet=\{\fm^{\lceil \lambda \rceil}\}_{\lambda\in\bR_{>0}}$ corresponds to the canonical $\fm$-norm $\ord_\fm:R\to\bZ_{\ge 0}\cup\{+\infty\}$ given by
		\[
		\ord_\fm(f):=\sup\{k\in\bZ_{\ge 0}\mid f\in \fm^k\},
		\]
        which has been considered in \cite[Section 4]{BFJ14}.
	
\end{example}

\begin{remark}
    By Lemma \ref{lem:properties of homogeous}, a homogeneous filtration satisfies the additional condition above. In particular, we may identify the spaces of filtrations and norms whenever we consider the metric properties.

    Indeed, since we only consider the asymptotic behavior of a filtration, we can modify the terms of $\fa_\bullet$ in the following way, such that the additional condition is always satisfied. Since $\{\fa_{1/n}\}$ is an increasing sequence of $\fm$-primary ideals, there exists $N\in\bZ_{>0}$ such that $\fa_{1/n}=\fa_{1/N}\eqqcolon \fa$ for any $n\ge N$. If $\fa=\fm$ then we are done. Otherwise, choose $C\in\bZ_{>0}$ such that $\fm^C\subset\fa$ and define 
    \begin{equation*}
        \fb_\lambda\coloneqq\left\{\begin{aligned}
            &\fm, &0<\lambda\le 1/NC,\\
            &\fa_\lambda, &\text{otherwise.}
        \end{aligned}\right.
    \end{equation*}
    Then clearly $\fb_\bullet$ has the same asymptotic information as $\fa_\bullet$. 
\end{remark}

\subsubsection{Divisorial valuations}
Recall that a valuation $v\in\Val_{R,\fm}$ defines a local ring $\cO_v$, the \emph{valuation ring} of $v$, by
\[
    \cO_v\coloneqq \{f\in \mathrm{Frac}(R)\mid v(f)\ge 0\} \text{ with the maximal ideal } \fm_v=\{f\in \cO_v\mid v(f)\ge 0\}.
\]
The \emph{residue field} of $v$ is $\kappa_v\coloneqq \cO_v/\fm_v$.

A valuation $v\in \Val_{R,\fm}$ is \emph{divisorial} if 
\[
{\rm tr.deg}_{\kappa}(\kappa_v) =n-1.
\]
We write $\DivVal_{R,\fm}\subset \Val_{R,\fm}$ for the set of such valuations.

Divisorial valuations appear geometrically.  
If $\mu:Y\to X$ is a proper birational morphism with $Y$ normal and  $E\subset Y$ a prime divisor, 
then there is an induced valuation  $\ord_{E} \colon {\rm Frac}(R)^\times \to \bZ$. 
If $\mu(E) =x$ and  $c\in \bR_{>0}$, then $c\cdot \ord_E \in \DivVal_{R,\fm}$.
When $R$ is excellent, all divisorial valuations are of this form; see e.g. \cite[Lemma 6.5]{CS22}.

\subsubsection{Quasi-monomial valuations}\label{sssec:QM}
In the following construction, we always assume $R$ contains a field. Let $\mu\colon Y:= \Spec(S) \to X=\Spec(R)$ be a birational morphism
with $R\to S$  finite type
and $\eta\in Y$  a not necessarily closed point such that $\cO_{Y,\eta}$ is regular and $\mu(\eta) =x$.
Given a  regular system of parameters $y_1,\ldots, y_r$ of $\cO_{Y,\eta}$
and $\bm{\alpha}=(\alpha_1,\cdots, \alpha_r)\in \bR_{\geq 0}^r\setminus {\bf 0}$, we define a valuation $v_{\bf \alpha}$ as follows. 
For $0\neq f\in \cO_{Y,\eta}$, we can write $f$ in $\widehat{\cO}_{Y,\eta} \simeq k(\eta) [[y_1,\ldots, y_r]]$ as
$\sum_{\bm{\beta}\in \bZ_{\geq 0}^r} c_{\bm{\beta}} {y}^{\bm{\beta}}$ and set 
\[
v_{\bm \alpha}(f)\coloneqq \min \{ \langle \bm{\alpha}, {\bm \beta} \rangle \, \vert \, c_{\bm \beta} \neq 0 \}.
\]
A valuation of the above form is called \emph{quasi-monomial}.

Let $D= D_{1}+\cdots +D_r$ be a reduced divisor on $Y$ such that $y_i=0$ locally defines $D_i$ and $\mu(D_i)= x$ for each $i$. 
We call $\eta\in (Y,D)$ a \emph{log smooth birational model} of $X$. 
We write $\QM_\eta(Y,E)  \subset \Val_{X,x}$ for the set of quasi-monomial valuations that can be described at 
$\eta$ with respect to $y_1,\ldots, y_r$ and note that $\QM_\eta (Y,D) \simeq \bR^r_{\geq 0} \setminus \bm{0}$.

\subsubsection{Multiplicity of a filtration}\label{sssec:mult}

Following \cite{ELS03}, the \emph{multiplicity} of a graded sequence of $\fm$-primary ideals $\fa_\bullet$ is
\[
    \e(\fa_\bullet)\coloneqq \lim_{m\to\infty}\frac{\ell(R/\fa_m)}{m^n/n!} \in [0,\infty).
\]
By  \cites{ELS03,Mus02,LM09,Cut13,Cut14} in increasing generality, the above limit exists and
\begin{equation}\label{eqn:vol=mult}
    \e(\fa_\bullet)=\lim_{m\to\infty}\frac{\e(\fa_m)}{m^n} = \inf_{m} \frac{\e(\fa_m)}{m^n};
\end{equation}
see in particular \cite[Theorem 6.5]{Cut14}.
Also defined in \cite{ELS03}, the \emph{volume} of a valuation $v\in \Val_{R,\fm}$ is $\vol(v):= \e(\fa_\bullet(v))$.
Denote $\Val_{R,\fm}^+\coloneqq \{v\in\Val_{R,\fm}\mid \vol(v)>0\}\subset \Val_{R,\fm}$.

\subsubsection{Linearly boundedness}\label{sssec:linearly bounded}

\begin{definition}\label{defn:linearly bounded}
    An $\fm$-filtration $\fa_\bullet$ is \emph{linearly bounded} if there exists a constant $c\in\bR_{>0}$ such that $\fa_{\bullet} \subset \fm^{c\bullet}$. 
	
    An $\fm$-norm $\chi$ is \emph{linearly bounded} if there exists $c\in\bR_{>0}$ such that $\chi(f)\le c^{-1}\ord_\fm(f)$ for any $f\in R$. 
\end{definition}

By Krull intersection theorem, a linearly bounded $\fm$-filtration satisfies $\cap_{\lambda\in\bR_{\ge 0}}\fa_\lambda=\{0\}$. In particular, by Definition-Lemma \ref{deflem:fm-equivalence}, we have the following 

\begin{lemma}\label{lem:fm-equivalence lin bounded}
    An $\fm$-filtration is linearly bounded if and only if $\ord_{\fa_\bullet}$ is a linearly bounded $\fm$-norm.\qed
\end{lemma}

We will consider filtrations with positive multiplicity, which are exactly the linearly bounded ones by the following lemma.

\begin{lemma}\cite[Corollary 3.18]{BLQ22}\label{lem:e>0 iff linearly bounded}
    An $\fm$-filtration $\fa_\bullet$ is linearly bounded if and only if $\e(\fa_\bullet)>0$.
\end{lemma}

Denote the set of all linearly bounded $\fm$-filtrations on $R$ by $\Fil_{R,\fm}$, and the set of all linearly bounded $\fm$-norms on $R$ by $\cN_{R,\fm}$. When there is no ambiguity, we will write $\Fil$ and $\cN$ respectively. By Definition-Lemma \ref{deflem:fm-equivalence} and Lemma \ref{lem:fm-equivalence lin bounded}, $\cN$ is a subset of $\Fil$. 

There is a natural pointwise partial order $\le$ on $\cN$ on $R$, defined by
\[
    \chi\le \chi' \text{ if and only if } \chi(f)\le \chi'(f) \text{ for any } f\in R,
\]
which corresponds to the partial order defined by inclusion on filtrations. Moreover, the natural $\bR_{>0}$-action of scaling on $\cN$ corresponds to the scaling on $\fm$-filtrations is given by
\[
(c\cdot\chi)(f):=c^{-1}\chi(f).
\]


\subsubsection{Asymptotic invariants}\label{sssec:asymp inv}

\begin{lemma}(cf. \cite[Lemma 2.4]{JM12})\label{lem:evaluation}
    Let $\chi\in\cN_{R,\fm}$ and $w\in\Val_{R,\fm}$. Then
    \[
        w(\fa_\bullet(\chi))=\inf_{f\in\fm}\frac{w(f)}{\chi(f)}.
    \]
\end{lemma}

\begin{proof}
    Let $c\coloneqq\inf_{f\in\fm}\frac{w(f)}{\chi(f)}$. Then for any $f\in \fm$ we have $w(f)\ge c\cdot \chi(f)$, hence by definition, $w(\fa_\lambda(\chi))\ge c\cdot \chi(\fa_\lambda(\chi))\ge c\lambda$ for any $\lambda\in\bR_{>0}$. Letting $\lambda\to\infty$ and dividing by $\lambda$, we get $w(\fa_\bullet(\chi))\ge c$.
	
    Conversely, for any $\epsilon>0$, by definition there exists $f\in \fm$ such that $0<w(f)<(c+\epsilon)\chi(f)$. Note that for any $k\in\bZ_{>0}$, $f^d\in\fa_{d\cdot\chi(f)}(\chi)$ as $\chi(f^d)\ge d\cdot\chi(f)$.
    Thus for $d\in\bZ_{>0}$ we have 
    \begin{align*}
        \frac{w(\fa_{d\cdot\chi(f)}(\chi))}{d\cdot \chi(f)}\le\frac{w(f^d)}{d\cdot\chi(f)}=\frac{d\cdot w(f)}{d\cdot \chi(f)}<c+\epsilon.
    \end{align*}
    Letting $d\to\infty$ we get $w(\fa_\bullet(\chi))\le c+\epsilon$. This finishes the proof.
\end{proof}

\subsection{Saturated filtrations}\label{ssec:saturation}

The saturation of an $\fm$-filtration was introduced in \cite{BLQ22}. We recall some facts about saturations and provide a characterization for saturated filtrations in Proposition \ref{prop:alt def for saturation}. Several of the results in this section seem to be new.

\begin{definition}\label{defn:saturation}
    The \emph{saturation} $\widetilde{\fa}_\bullet$ of an $\fm$-filtration $\fa_\bullet$ is defined by
    \[
        \widetilde \fa_\lambda\coloneqq \{f\in \fm \mid v(f)\ge \lambda\cdot v(\fa_\bullet) \text{ for all } v\in\DivVal_{R,\fm} \}
    \]
    for each $\lambda\in \R_{>0}$.
    We say that $\fa_\bullet$ is \emph{saturated} if $\fa_\bullet=\widetilde\fa_\bullet$. 
\end{definition}

Denote the subset of $\Fil$ of saturated filtrations by $\Fil^s\coloneqq \{\fa_\bullet\in\Fil\mid \fa_\bullet \text{ is saturated}\}$. There is a canonical element in $\Fil^s$, that is, the saturation $\widetilde\fm^\bullet$ of the $\fm$-adic filtration. Denote its order function by $\chi_0:=\ord_{\widetilde\fm^\bullet}=\widetilde\ord_\fm$.

The following lemma asserts that for saturated filtrations, the partial order by inclusion can be detected by evaluating on all divisorial valuations, which is not the case for general filtrations.

\begin{lemma}\label{lem:pointwise order equiv to +-order}
    Let $\fa_\bullet$ be an $\fm$-filtration and $\fb_\bullet\in\Fil^s$. Then $\fa_\bullet\subset\fb_\bullet$ if and only if $v(\fa_\bullet)\ge v(\fb_\bullet)$ for any $v\in\DivVal_{X,x}$.
\end{lemma}

\begin{proof}
    If $\fa_\bullet\subset\fb_\bullet$ then by definition $v(\fa_\bullet)\ge v(\fb_\bullet)$ for any $v\in\Val_{X,x}$. 
	
    Conversely, assume that $\fa_\bullet\nsubseteq\fb_\bullet$. Then there exist $\lambda\in\bR_{>0}$ and $f\in\fa_\lambda$ such that $f\notin\fb_\lambda$. This implies that there exist $\epsilon>0$ and $v\in\DivVal_{X,x}$ such that $v(f)<\lambda\cdot v(\fb_\bullet)-\epsilon$. Hence for any $m\in\bZ_{>0}$, we have
    \[
        v(\fa_{m\lambda})\le v(\fa_\lambda^m)=mv(\fa_\lambda)\le mv(f)<m(\lambda\cdot v(\fb_\bullet)-\epsilon),
    \]
    where the first inequality follows from $\fa_\lambda^m\subset\fa_{m\lambda}$. Dividing by $m\lambda$ and letting $m\to\infty$, we get $v(\fa_\bullet)\le v(\fb_\bullet)-\epsilon/\lambda<v(\fb_\bullet)$. 
\end{proof}

The saturation can be alternatively defined using all valuations with positive volume.

\begin{lemma}\cite[Proposition 3.19]{BLQ22}\label{lem:saturation alt def}
    If $\fa_\bullet\in\Fil_{R,\fm}$ and $\lambda\in\bR_{>0}$, then
    \[
        \widetilde\fa_\lambda=\{f\in\fm\mid v(f)\ge \lambda\cdot v(\fa_\bullet) \text{ for all } v\in\Val^+_{R,\fm}\}.
    \]
\end{lemma}

One of the main properties of saturations is the following generalization of Rees' theorem on multiplicities of ideals.

\begin{theorem}\label{thm:equal volume iff equal saturation}\cite[Theorem 1.4]{BLQ22}
    For $\fm$-filtrations  $\fa_\bullet \subset \fb_\bullet$, $\e(\fa_\bullet)=\e(\fb_\bullet)$ if and only if $\widetilde\fa_\bullet=\widetilde\fb_\bullet$. 
\end{theorem}

In other words, $\widetilde\fa_\bullet$ is the unique maximal $\fm$-filtration $\fb_\bullet\supset\fa_\bullet$ such that $\e(\fa_\bullet)=\e(\fb_\bullet)$. We now prove a formula for saturations, which can be viewed as an analogue of the Berkovich maximal modulus principle, where one considers the subset of divisorial valuations of the Berkovich spectrum.

\begin{lemma}\label{lem:valuative characterization for saturation}
    Let $\fa_\bullet\in\Fil$. Then we have
    \[
        \widetilde\fa_\bullet=\bigcap_{v(\fa_\bullet)=1}
        \fa_\bullet(v)=\bigcap_{v\ge \ord_{\fa_\bullet}}\fa_\bullet(v), 
    \]
    where the first intersection is taken over all $v\in\DivVal_{R,\fm}$ with $v(\fa_\bullet)=1$ and the second over all $v\in\DivVal_{R,\fm}$ such that $v(f)\ge\ord_{\fa_\bullet}(f)$ for any $f\in \fm$.
\end{lemma}

\begin{proof}
    Note that $v(\fa_\bullet)>0$ for any $v\in\DivVal_{R,\fm}$, by Lemma \ref{lem:e>0 iff linearly bounded}. For $v\in\DivVal_{R,\fm}$ with $v(\fa_\bullet)=1$, by \cite[Proposition 3.9]{BLQ22} we know that $v(\widetilde\fa_\bullet)=1$, so $\widetilde\fa_\bullet\subset\fa_\bullet(v)$ by definition. Hence $\widetilde\fa_\bullet\subset\cap_{v(\fa_\bullet)=1}\fa_\bullet(v)$. Conversely, if $f\in\cap_{v(\fa_\bullet)=1}\fa_\lambda(v)$, then for any $v\in\DivVal_{R,\fm}$, since $v'(\fa_\bullet)=1$, where $v':=v/v(\fa_\bullet)$, we have 
    \[
        v(f)=\frac{v}{v(\fa_\bullet)}(f)\cdot v(\fa_\bullet)\ge\lambda\cdot v(\fa_\bullet).
    \]
    Hence $\cap_{v(\fa_\bullet)=1}\fa_\lambda(v)\subset \widetilde\fa_\bullet$ and this proves the first equality.
	
    If $v(\fa_\bullet)=1$ then for any $f\in\fa_m$, we have $v(f)\ge v(\fa_m)\ge mv(\fa_\bullet)$.
    Hence $v(f)\ge \ord_{\fa_\bullet}(f)$, that is, $v\ge\ord_{\fa_\bullet}$, and we get $\cap_{v(\fa_\bullet)=1}\fa_m(v)\supset\cap_{v\ge \ord_{\fa_\bullet}}\fa_m(v)$. Conversely, if $v\ge \ord_{\fa_\bullet}$, then $v(\fa_m)\ge \ord_{\fa_\bullet}(\fa_m)\ge m$, which implies $v(\fa_\bullet)\ge 1$, that is, $v\ge v'\coloneqq v/v(\fa_\bullet)$, where $v'(\fa_\bullet)=1$. Thus we get 
    \[
        \cap_{v(\fa_\bullet)=1}
        \fa_m(v)\subset\cap_{v'}\fa_m(v')\subset \cap_{v\ge \ord_{\fa_\bullet}}\fa_m(v).
    \] 
    This proves the second equality.
\end{proof}

Based on the observation above, we give a characterization for saturated filtrations.

\begin{proposition}\label{prop:alt def for saturation}
    A filtration $\fa\in\Fil$ is saturated if and only if it is of the form 
    \[
        \fa_\bullet=\cap_{v\in \Sigma} \fa_\bullet(v),
    \]
    where $\Sigma$ is a nonempty subset of either $\DivVal_{R,\fm}$ or $\Val^+_{R,\fm}$.
\end{proposition}

\begin{proof}
    If $\fa_\bullet=\cap_{v\in \Sigma}\fa_\bullet(v)$ for some $\Sigma\subset\Val^+_{R,\fm}$, then by \cite[Lemma 3.20]{BLQ22}, $\fa_\bullet$ is saturated.
	
    Conversely, if $\fa_\bullet$ is saturated, then by Lemma \ref{lem:valuative characterization for saturation}, $\fa_\bullet=\widetilde\fa_\bullet=\cap_{v\in \Sigma_v}\fa_\bullet(v)$, where 
    \[
        \Sigma_v\coloneqq\{v\in\DivVal_{R,\fm}\mid v(\fa_\bullet)=1)\}\subset\DivVal_{R,\fm}.
    \]
	
    Since $\DivVal_{R,\fm}\subset\Val^+_{R,\fm}$, the proof is finished.
\end{proof}

    In particular, $\fa_\bullet(v)$ is saturated for any $v\in\Val^+_{R,\fm}$. By the above proposition, saturated filtrations (or norms) can be considered as the local analogue of maximal norms in the sense of \cite[Definition 6.16]{BJ21}.

\begin{corollary}\label{cor:intersection of saturated is saturated}
    Let $\fa_{\bullet,i}\in\Fil^s$ for $i\in I$. If $\cap_{i\in I}\fa_{\bullet,i}\in\Fil$, then $\cap_{i\in I}\fa_{\bullet,i}\in\Fil^s$.
\end{corollary}

\begin{proof}
    By Proposition \ref{prop:alt def for saturation}, for any $i\in I$ we may write $\fa_{\bullet,i}=\cap_{v\in \Sigma_i}\fa_\bullet(v)$ for $\Sigma_i\subset \DivVal_{X,x}$. Thus 
    \[
        \cap_{i\in I}\fa_{\bullet,i}=\cap_{v\in \cup_{i\in I} \Sigma_i}\fa_\bullet(v)\in\Fil^s
    \]
    by Proposition \ref{prop:alt def for saturation} again.
\end{proof}

\begin{deflem}\label{deflem:saturated join}
    Let $\fa_{\bullet,i}\in\Fil$ for $i\in I\ne \emptyset$. If there exists $\fb_\bullet\in\Fil$ such that $\fa_{\bullet,i}\subset\fb_\bullet$ for any $i\in I$, then there exists an $\fm$-filtration, called the \emph{saturated join} of the $\fa_{\bullet,i}$, denoted by $\vee_{s,i\in I}\fa_{\bullet,i}$, such that the following holds. If $\fc_\bullet\in\Fil^s$ satisfies $\fa_{\bullet,i}\subset\fc_\bullet$ for any $i\in I$, then we have $\vee_{s,i\in I}\fa_{\bullet,i}\subset\fc_\bullet$. Moreover, $\vee_{s,i\in I}\fa_{\bullet,i}\in\Fil^s$.
\end{deflem}

\begin{proof}
    Fix $0\in I$. By assumption, the set
    \[
        J:=\{\fc_\bullet\in\Fil^s\mid \fa_{\bullet,i}\subset\fc_\bullet \text{ for any } i\in I\}\subset\Fil^s
    \]
    is non-empty, and $\fa_{\bullet,0}\subset\fc_\bullet$ for any $\fc_\bullet\in J$. Thus by Definition-Lemma \ref{deflem:arbitrary intersection}, the intersection
    \[
        \vee_{s,i\in I}\fa_{\bullet,i}\coloneqq\cap_{\fc_\bullet\in J}\fc_\bullet\in\Fil,
    \]
    where the linearly boundedness is because $\vee_{s,i\in I}\fa_{\bullet,i}\subset\fb_\bullet$ by construction. Now $\vee_{s,i\in I}\fa_i\in\Fil^s$ by Corollary \ref{cor:intersection of saturated is saturated}.
\end{proof}

We now relate the join, the saturated join and the saturation, and give a formula for the saturated join of two saturated filtrations.

\begin{proposition}\label{prop:lattice of Fil^s}
    If $\fa_\bullet,\fb_\bullet\in\Fil$, then $\fa_\bullet\vee_s\fb_\bullet=\widetilde{\fa_\bullet\vee\fb_\bullet}=\widetilde\fa_\bullet\vee_s\widetilde\fb_\bullet$, and
    \begin{equation}\label{eqn:formula for saturated join}
        \fa_\bullet\vee_s\fb_\bullet=\cap_{v\in \DivVal_{R,\fm}} \fa_\bullet(\frac{v}{\min\{v(\fa_\bullet),v(\fb_\bullet)\}}).
    \end{equation}
\end{proposition}

\begin{proof}
    By definition, $\widetilde{\fa_\bullet\vee\fb_\bullet}$ is a saturated filtration containing both $\fa_\bullet$ and $\fb_\bullet$, so by Definition-Lemma \ref{deflem:saturated join} we have $\fa_\bullet\vee_s\fb_\bullet\subset\widetilde{\fa_\bullet\vee\fb_\bullet}$. Conversely, by Definition-Lemma \ref{deflem:saturated join} again $\fa_\bullet\vee_s\fb_\bullet$ is saturated and $\fa_\bullet\vee\fb_\bullet\subset\fa_\bullet\vee_s\fb_\bullet$ by definition, hence $\widetilde{\fa_\bullet\vee\fb_\bullet}\subset\fa_\bullet\vee_s\fb_\bullet$. The inclusion $\fa_\bullet\vee_s\fb_\bullet\subset\widetilde\fa_\bullet\vee_s\widetilde\fb_\bullet$ is obvious. Conversely, if $\fa_\bullet\subset\fc_\bullet$ for $\fc_\bullet\in\Fil^s$ then $\widetilde\fa_\bullet\subset\fc_\bullet$ and similarly for $\fb_\bullet$. Thus $\widetilde\fa_\bullet\vee_s\widetilde\fb_\bullet\subset\fa_\bullet\vee_s\fb_\bullet$. This proves the first equality.
	
    To prove the formula \eqref{eqn:formula for saturated join}, we first rewrite Lemma \ref{lem:valuative characterization for saturation} as 
    \begin{equation}\label{eqn:formula as intersection over all div}
        \fa_\bullet=\cap_{v\in\DivVal_{R,\fm}}
        \fa_\bullet(\frac{v}{v(\fa_\bullet)})
    \end{equation}
    for $\fa_\bullet\in\Fil^s$. Let $\fc_\bullet$ be the $\fm$-filtration defined by the right-handed side of \eqref{eqn:formula for saturated join}. Since
    \[
        \frac{v}{\min\{v(\fa_\bullet),v(\fb_\bullet)\}}=\max\{\frac{v}{v(\fa_\bullet)},\frac{v}{v(\fb_\bullet)}\}=\max\{\frac{v}{v(\widetilde\fa_\bullet)},\frac{v}{v(\widetilde\fb_\bullet)}\},
    \]
    applying the formula \eqref{eqn:formula as intersection over all div} to $\widetilde\fa_\bullet$ and $\widetilde\fb_\bullet$ respectively, we know that $\fa_\bullet\subset\widetilde\fa_\bullet\subset\fc_\bullet$ and similarly $\fb_\bullet\subset\fc_\bullet$. Hence $v(\fc_\bullet)\le\min\{v(\fa_\bullet),v(\fb_\bullet)\}$ for any $v\in\Val_{R,\fm}$. By \eqref{eqn:formula for saturated join}, for any $v\in\DivVal_{R,\fm}$ we have $v(\fc_\bullet)\ge v(\fa_\bullet(v/\min\{v(\fa_\bullet),v(\fb_\bullet)\}))=\min\{v(\fa_\bullet),v(\fb_\bullet)\}$. Thus 
    \begin{equation}\label{eqn:join is pointwise minimum}
        v(\fc_\bullet)=\min\{v(\fa_\bullet),v(\fb_\bullet)\}.
    \end{equation}
	
    We now verify that $\fc_\bullet\in\Fil$. Take $C\in\bR_{>0}$ such that $\fa_{C\bullet}\subset\fm^\bullet$ and $\fb_{C\bullet}\subset\fm^\bullet$. Since $\fm^\bullet\subset\widetilde\fm^\bullet$, for any $v\in\DivVal_{R,\fm}$, we have $Cv(\fa_\bullet)\ge v(\fm^\bullet)=v(\widetilde\fm^\bullet)$ and similarly $Cv(\fb_\bullet)\ge v(\widetilde\fm^\bullet)$. By \eqref{eqn:join is pointwise minimum}, $Cv(\fc_\bullet)\ge v(\widetilde\fm^\bullet)$. So by Lemma \ref{lem:pointwise order equiv to +-order}, $\fc_\bullet\subset\widetilde\fm^{\bullet/C}$ is linearly bounded. 
	
    By Corollary \ref{cor:intersection of saturated is saturated}, $\fc_\bullet\in\Fil^s$, which implies $\fa_\bullet\vee_s\fb_\bullet\subset\fc_\bullet$. By definition and \eqref{eqn:join is pointwise minimum} we have $v(\fa_\bullet\vee_s\fb_\bullet)\le\min\{v(\fa_\bullet),v(\fb_\bullet)\}=v(\fc_\bullet)$ for any $v\in\DivVal_{R,\fm}$. Since $\fa_\bullet\vee_s\fb_\bullet\in\Fil^s$ by Definition-Lemma \ref{deflem:saturated join}, we may apply Lemma \ref{lem:pointwise order equiv to +-order} to conclude that $\fc_\bullet\subset\fa_\bullet\vee_s\fb_\bullet$. This proves the formula \eqref{eqn:formula for saturated join} and finishes the proof of the Proposition.
\end{proof}

We now show that the lattice operations on $\Fil^s$ satisfy the distributive law.

\begin{proposition}\label{prop:distributivity}
    If $\fa_\bullet,\fb_\bullet,\fc_\bullet\in\Fil^s$, then we have
    $\fa_\bullet\vee_s(\fb_\bullet\cap\fc_\bullet)=(\fa_\bullet\vee_s\fb_\bullet)\cap(\fa_\bullet\vee_s\fc_\bullet)$.
\end{proposition}

\begin{proof}
    We have $\fa_\bullet\vee_s(\fb_\bullet\cap\fc_\bullet)\subset(\fa_\bullet\vee_s\fb_\bullet)\cap(\fa_\bullet\vee_s\fc_\bullet)$, since $\fa_\bullet\vee_s(\fb_\bullet\cap\fc_\bullet)\subset\fa_\bullet\vee_s\fb_\bullet$ and $\fa_\bullet\vee_s(\fb_\bullet\cap\fc_\bullet)\subset\fa_\bullet\vee_s\fc_\bullet$.
	
    Conversely, assume that $f\in (\fa_\bullet\vee_s\fb_\bullet)_\lambda\cap (\fa_\bullet\vee_s\fc_\bullet)_\lambda$. Then for any $v\in\DivVal_{R,\fm}$ we know that
    \begin{align*}
        v(f)/\lambda\ge &\max\{v(\fa_\bullet\vee_s\fb_\bullet), v(\fa_\bullet\vee_s\fc_\bullet)\}\\
        =&\max\{\min\{v(\fa_\bullet),v(\fb_\bullet)\},\min\{v(\fa_\bullet),v(\fc_\bullet)\}\}\\
        \ge&\min\{v(\fa_\bullet),\max\{v(\fb_\bullet),v(\fc_\bullet)\}\},
    \end{align*}
    where the second inequality follows from Proposition \ref{prop:lattice of Fil^s}, and the last inequality is easy to check. By Proposition \ref{prop:lattice of Fil^s} again, we get $f\in(\fa_\bullet\vee_s(\fb_\bullet\cap\fc_\bullet))_\lambda$. This shows that $\fa_\bullet\vee_s(\fb_\bullet\cap\fc_\bullet)\supset(\fa_\bullet\vee_s\fb_\bullet)\cap(\fa_\bullet\vee_s\fc_\bullet)$ and the equality must hold.
\end{proof}

\begin{example}\label{eg:saturated ideal join}
    We continue to consider the adic filtration $\fa^\bullet$ of $\fa\in\cI_\fm$. By definition, 
    \[
        (\widetilde{\fa^\bullet})_\lambda=\{f\in\fm\mid v(f)\ge\lambda\cdot v(\fa^\bullet) \text{ for any } v\in\DivVal_{R,\fm}\}.
    \]
    Since $v(\fa^\bullet)=\lim_{\lambda\to\infty} v(\fa^{\lceil\lambda\rceil})/\lambda=\lim_{\lambda\to\infty} \lceil\lambda\rceil v(\fa)/\lambda=v(\fa)$, for $m\in\bZ_{>0}$ we have
    \[
        (\widetilde{\fa^\bullet})_m=\{f\in\fm\mid v(f)\ge mv(\fa)=v(\fa^m) \text{ for any } v\in\DivVal_{R,\fm}\}.
    \]
    Hence $(\widetilde{\fa^\bullet)}_m=\overline{\fa^m}$ is the integral closure of $\fa^m$. See for example \cite[Theorem 6.8.3]{HS06}. Thus we have a map $\cI_\fm\to\Fil^s$, $\fa\mapsto\widetilde{\fa^\bullet}$, which is an injection when restricted to the set $\cI_\fm^{\ic}$ of integrally closed $\fm$-primary ideals. Note that in general, $\widetilde{\fa^\bullet}$ is not an adic filtration. For example, $\widetilde{(\fm^2)^\bullet}$ has $1/2$ as a jumping number.
	
    Even if $\fa$ is integrally closed, the powers of $\fa$ are in general not integrally closed, hence an $\fa$-adic filtration is in general \emph{not} saturated. By Example \ref{eg:ideal join}, the join of two adic filtrations is still an adic filtration, which means that the inclusion $\fa_\bullet\vee\fb_\bullet\subset \fa_\bullet\vee_s\fb_\bullet$ can be strict for general $\fa_\bullet,\fb_\bullet\in\Fil$. It is unclear to the author whether $\fa_\bullet\vee\fb_\bullet=\fa_\bullet\vee_s\fb_\bullet$ if $\fa_\bullet,\fb_\bullet\in\Fil^s$.
	
    It is easy to see that $\widetilde{\fa_\bullet\cap\fb_\bullet}\subset\widetilde\fa_\bullet\cap\widetilde\fb_\bullet$, and the inclusion can be strict. Indeed, let $R=\bk[\![x,y]\!]$, $\fa=(x^2,y^{10})$ and $\fb=(x^{10},y^2)$, then it is not hard to get $\overline\fa=(x^2,xy^5,y^{10})$ and $\overline\fb=(x^{10},x^5y,y^2)$. By looking at the Newton polytope of $\fa\cap\fb=(x^2y^2,x^{10},y^{10})$, one can compute
    \[
        \overline{\fa\cap\fb}=(x^2y^2,xy^6,x^6y,x^{10},y^{10}).
    \]
    In particular, by the calculation for adic filtrations above, $f\coloneqq x^5y\in \overline\fa\cap\overline\fb=(\widetilde\fa^\bullet\cap\widetilde\fb^\bullet)_1$, but $f\notin \overline{\fa\cap\fb}=(\widetilde{\fa_\bullet\cap\fb_\bullet})_1$.
\end{example}

Next we compute the saturation of an ideal-adic filtration. First note that to prove two saturated filtrations are equal, it suffices to check the equality along an arbitrary unbounded sequence.

\begin{lemma}\label{lem:test equality}
    Let $\fa_\bullet,\fb_\bullet\in\Fil^s$. If there exists $\lambda_k\in\bR_{>0}$ with $\lambda_k\to\infty$ such that $\fa_{\lambda_k}=\fb_{\lambda_k}$, then $\fa_\bullet=\fb_\bullet$.
\end{lemma}

\begin{proof}
    By assumption, for any $v\in\DivVal_{R,\fm}$ we have 
    \[
        v(\fa_\bullet)= \lim_{k\to\infty} v(\fa_{\lambda_k})/\lambda_k=\lim_{k\to\infty} v(\fb_{\lambda_k})/\lambda_k=v(\fb_\bullet).
    \]
    Hence by Lemma \ref{lem:pointwise order equiv to +-order}, $\fa_\bullet=\fb_\bullet$. 
\end{proof}

Recall that for an ideal $\fa\in\cI_R$, the set $\RV(\fa)$ of \emph{Rees valuations} of $\fa$ is the minimal set of divisorial valuations such that for any $m\in\bZ_{>0}$, 
\[
    \overline{\fa^m}=\{f\in R\mid v(f)\ge mv(\fa) \text{ for any } v\in\RV(\fa)\}.
\]
This is a finite set for any $\fa$. We refer to \cite[Chapter 10]{HS06} for more about Rees valuations.
\begin{proposition}\label{prop:saturation and Rees}
    Let $\fa\in\cI_\fm$ be an $\fm$-primary ideal. Then
    \begin{equation}\label{eqn:saturation of adic filtration}
        \widetilde{\fa^\bullet}=
        \cap_{v\in RV(\fa)}\fa_\bullet(\frac{v}{v(\fa)}),
    \end{equation}
    where $RV(\fa)$ is the set of Rees valuations of $\fa$.
\end{proposition}

\begin{proof}
    Denote $\fb_\bullet\coloneqq \cap_{v\in RV(\fa)}\fa_\bullet(\frac{v}{v(\fa)})$, which is saturated by Proposition \ref{prop:alt def for saturation}. By the definition of Rees valuations and the calculation in Example \ref{eg:saturated ideal join}, we know that $(\widetilde{\fa^\bullet})_m=\fb_m$ for $m\in\bZ_{>0}$. Thus \eqref{eqn:saturation of adic filtration} follows from Lemma \ref{lem:test equality}.
\end{proof}

\subsection{Homogeneous filtrations}\label{ssec:homogeneous}

\begin{defn}\label{defn:homogeneity}
    A norm $\chi\in\cN$ is said to be \emph{power-multiplicative}, or \emph{homogeneous}, if it satisfies $\chi(f^d)=d\cdot \chi(f)$ for any $f\in R$ and $d\in\bZ_{>0}$.
	
    A filtration $\fa_\bullet$ is said to be \emph{homogeneous} if $f^d\in\fa_\lambda$ implies $f\in\fa_{\lambda/d}$ for any $\lambda>0$ and $d\in\bZ_{>0}$. 
\end{defn}

It is easy to see that a saturated filtration is homogeneous, and that a homogeneous filtration satisfies the additional condition in Definition-Lemma \ref{deflem:fm-equivalence}.

\begin{lemma}\label{lem:properties of homogeous}
    Let $\fa_\bullet$ be an $\fm$-filtration. Then
    \begin{enumerate}
        \item if $\fa_\bullet$ is saturated, then it is homogeneous, and

        \item if $\fa_\bullet$ is homogeneous, then $\fa_\lambda=\fm$ for $0<\lambda\ll 1$. In particular, $\fa_\bullet$ is homogeneous if and only if $\ord_{\fa_\bullet}$ is homogeneous.
    \end{enumerate}
\end{lemma}

\begin{proof}
    (1) If $\fa_\bullet$ is saturated and $f^d\in\fa_\lambda$, then $v(f^d)=dv(f)\ge\lambda\cdot v(\fa_\bullet)$, which implies $v(f)\ge \lambda/d\cdot v(\fa_\bullet)$. Hence $f\in\fa_{\lambda/d}$ by definition.

    (2) Choose a set $\{f_1,\ldots,f_r\}$ of generators of $\fm$ and $D\in\bZ_{>0}$ such that $\fm^D\subset \fa_1$. Then $f_i^D\in\fa_1$, and hence $f_i\in\fa_{1/D}$ by assumption. So $\fm\subset\fa_{1/D}$ and hence $\fa_\lambda=\fm$ for $0<\lambda<1/D$. The last assertion is straightforward by Definition-Lemma \ref{deflem:fm-equivalence}.
\end{proof}

\begin{definition}(c.f. \cite[Definition 6.9.3]{HS06} and \cite[Section 4]{BFJ14})\label{defn:homogenization}
    Given a norm $\chi$ on $R$, its \emph{homogenization} $\widehat\chi:R\to \bR_{\ge 0}\cup \{+\infty\}$ is defined by $\widehat\chi(f):=\lim_{k\to\infty}\frac{\chi(f^k)}{k}$ for any $f\in R$.

    The filtration $\widehat\fa_\bullet\coloneqq \fa(\widehat\chi)$ is called the \emph{homogenization} of $\fa_\bullet\coloneqq\fa(\chi)$. In \cite[Definition 3.5]{CP22}, $\widehat\chi$ is also called the \emph{asymptotic Samuel function} of $\fa_\bullet$, following \cite{HS06}.
\end{definition}

The limit exists by Fekete's lemma, see, for example, \cite[Theorem 3.4]{CP22}. Hence it is not hard to see that $\widehat{\fa}_\bullet$ is the smallest homogeneous filtration that contains $\fa_\bullet$. Moreover, the 
homogenization has the following property.

\begin{theorem}\cite{CP22}\label{thm:homogenization}
    Let $\fa_\bullet$ be an $\fm$-filtration. Then  \begin{enumerate}
        \item $\widehat\fa_\bullet$ is the unique largest $\fm$-filtration $\fb_\bullet$ such that $\widehat{\ord_{\fa_\bullet}}=\widehat{\ord_{\fb_\bullet}}$, and

        \item  $\widehat\fa_\bullet\subset\widetilde{\fa}_\bullet$. 
    \end{enumerate}
\end{theorem}

Before proceeding, we briefly recall the definition of the Berkovich analytification. 
Fix $\chi\in\cN$. By linear boundedness, we know that the $\fm$-adic completion $\hat R$ is complete with respect to $\chi$. Let $\cM=\cM(R,\chi)$ be the Berkovich spectrum of the Banach ring $(\hat R,\chi)$, that is, the set of multiplicative $\fm$-seminorms $w$ on $\hat R$ such that $w\ge \chi$. Note that the homogenization corresponds to the spectral radius construction of \cite{Ber90}, since we adopted the multiplicative notation. We thank the referee for pointing out a mistake in the statement of the previous version.

Next, we give a characterization for homogeneous norms similar to Proposition \ref{prop:alt def for saturation}, which is an easier local analogue of \cite[Theorem 3.17]{Fin22} and \cite[Theorem 2.16]{BJ22}.

\begin{lemma}\label{lem:alt def for homogeneous}
    A norm $\chi\in\cN$ is homogeneous if and only if it is of the form
    \[
        \chi=\inf_{w\in \Sigma} w,
    \]
    where $\Sigma\subset\cM$ is a non-empty subset.
\end{lemma}

\begin{proof}
    If $\chi(f)=\inf_{w\in I} w(f)$, then $\chi$ is homogeneous, since the same is true for each $w$.

    The reverse implication is essentially a restatement of the Berkovich Maximum Modulus principle, see \cite[Theorem 1.3.1]{Ber90}.
\end{proof}

Denote the subspace of $\Fil$ consisting of homogeneous filtrations by $\Fil^h$.

For adic filtrations, it is not hard to see that homogenization and saturation coincide.

\begin{lemma}\label{lem:homogenization and saturation for adic}
    Let $\fa\in\Fil_{R,\fm}$. Then $\widetilde{\fa^\bullet}=\widehat{\fa^\bullet}$. 
\end{lemma}

\begin{proof}
    By Example \ref{eg:saturated ideal join}, $(\widetilde{\fa^\bullet})_m=\overline{\fa^m}$ is the integral closure of $\fa^m$ for any $m\in\bZ_{>0}$. The same calculation implies $(\widehat{\fa^\bullet})_m=\overline{\fa^m}$ by \cite[Proposition 6.10]{HS06}. So the lemma follows from a simple calculation since both filtrations are homogeneous by Lemma \ref{lem:properties of homogeous}.
\end{proof}


\subsection{Log canonical thresholds and normalized volumes}\label{ss:nvol}\label{ssec:nvol}

We say $x\in(X,\Delta)$ is a \emph{klt singularity} if $\Delta$ is an $\bR$-divisor on $X$ such that $K_{X}+\Delta$ is $\R$-Cartier, and $(X,\Delta)$ is klt as defined in \cite{KM98}. The following invariant was first introduced by C. Li \cite{Li18} and  plays an important role in the study of Fano cone singularities and K-stability of Fano varieties.\footnote{The original motivation to introduce the metric $d_1$ is to study local volumes. Some results in this direction will appear elsewhere.}

\begin{definition}\label{def:nvol}
    For a klt singularity $x\in(X,\Delta)$, the \emph{normalized volume function} $\nvol_{(X,\Delta),x}:\Val_{X,x}\to (0,+\infty]$ is defined by
    \begin{equation*}
        \nvol_{(X,\Delta),x}(v)\coloneqq \left\{\begin{aligned}
        &A_{X,\Delta}(v)^n\cdot \vol(v), &\text{ if } A_{X,\Delta}(v)<+\infty\\
        &+\infty, &\text{ if } A_{X,\Delta}(v)=+\infty,
        \end{aligned}\right.
    \end{equation*}
    where $A_{X,\Delta}(v)$ is the log discrepancy of $v$ as defined in \cites{JM12,BdFFU15}. 
    The \emph{local volume} of a klt singularity $x\in(X,\Delta)$ is defined as
    \[
        \nvol(x,X,\Delta)\coloneqq \inf_{v\in\Val_{X,x}} \nvol_{(X,\Delta),x}(v).
    \]
    The above infimum is indeed a minimum by \cites{Blu18, Xu20}.
\end{definition}

Denote $\Val^{<+\infty}_{X,x}:=\{v\in\Val_{X,x}\mid A_X(v)<+\infty\}$.

Recall that the \emph{log canonical threshold (lct)} of an $\fm$-filtration $\fa_\bullet$ is defined to be
\[
    \lct(X;\fa_\bullet)\coloneqq\inf_{v\in\Val^{<+\infty}_{X,x}}\frac{A_X(v)}{v(\fa_\bullet)}.
\]
We will write $\lct(\fa_\bullet)$ when there is no ambiguity.

For more properties of lct, we refer to \cite{JM12}. We refer to \cites{LLX18,Zhu23} for a comprehensive survey on normalized volumes.

\begin{lemma}(c.f. \cite[Lemma 3.4]{Zhu22})\label{lem:lct control}
    There exists $c_0=c_0(n)>0$ depending only on $n$ such that for any $n$-dimensional klt singularity $x\in X$ and $v\in\Val_{X,x}$, 
    \begin{equation}\label{eqn:A/v control}
        \frac{A_X(v)}{v(\fa_\bullet)}\le c_0\cdot\frac{\nvol_{x,X}(v)}{\nvol(x,X)}\cdot\lct(X;\fa_\bullet)
    \end{equation}
    for any $\fa_\bullet\in\Fil_x(X)$.
\end{lemma}

The following proof is inspired by communications with Z. Zhuang.

\begin{proof}
    Note that $v(\fa_\bullet)=\lim_m v(\fa_m)/m$ and $\lct(X;\fa_\bullet)=\lim_m m\cdot \lct(X;\fa_m)$, it is enough to prove \eqref{eqn:A/v control} for $\fm$-primary ideals, 
    \[
        \frac{A_X(v)}{v(\fa)}\le c_0\cdot\frac{\nvol_{x,X}(v)}{\nvol(x,X)}\cdot\lct(X;\fa).
    \]
	
    Assume $\fa=(f_1,\ldots,f_m)$. By lifting $\fa$ to a power, we may assume that $\lct(\fa)<1$. Then by Bertini's theorem, for general $c_1,\ldots,c_m\in \bk$, $\lct(\sum_{i=1}^m c_if_i)=\lct(\fa)$.		Moreover, by definition, for $f\ne g\in\fa$, $v(f+cg)>\min\{v(f),v(g)\}$ for at most one $c\in\bk$. Hence one can choose the $c_i\in\bk$ such that $v(\sum c_if_i)=\min_{1\le i\le m}\{v(f_i)\}=v(\fa)$. So we are further reduced to the case of a principal ideal $\fa=(f)$. Now the lemma follows from \cite[Lemma 3.4]{Zhu22}.
\end{proof}

\subsection{Partially ordered sets and lattices}\label{ssec:posets}

We recall some basic definitions regarding lattices in order theory.

\begin{definition}\label{defn:lattice}
    A partially ordered set $(\cS,\le)$ is called a \emph{lattice} if for any $a,b\in \cS$, 
    \begin{enumerate}
        \item there exists $a\vee b\in \cS$, called the \emph{supremum}, or \emph{join}, such that for any $c\in \cS$, $c\ge a$ and $c\ge b$ imply $c\ge a\vee b$, and
			
        \item there exists $a\wedge b\in\cS$, called the \emph{infimum}, or \emph{meet}, such that for any $c\in\cS$, $c\le a$ and $c\le b$ imply $c\le a\wedge b$.
		
    \end{enumerate}
	
    A lattice $\cS$ is \emph{modular} if for any $a,b,c\in\cS$ with $a\le b$, we have $a\vee (c\wedge b)=(a\wedge c)\vee b$. A lattice $\cS$ is \emph{distributive} if for any $a,b,c\in\cS$, we have $a\vee(b\wedge c)=(a\wedge b)\vee (a\wedge c)$.

    A lattice $\cS$ is \emph{complete} if any subset $A\subset \cS$ has an infimum $a=\inf A$ and a supremum $b=\sup A$.
\end{definition}

\subsection{Metric spaces}\label{ssec:metric spaces}

Recall that a \emph{pseudometric} on a set $S$ is a function $d:S\times S\to\bR_{\ge 0}$ such that 
\begin{itemize}
    \item (non-negative) $d(x,y)\ge 0$ for any $x,y\in S$,
	
    \item (symmetric) $d(x,y)=d(y,x)$ for any $x,y\in S$, and
	
    \item (triangle inequality) $d(x,z)\le d(x,y)+d(y,z)$ for any $x,y,z\in S$. 
\end{itemize}
A pseudometric is a \emph{metric} if and only if it is non-degenerate, that is, $d(x,y)=0$ if and only if $x=y$. 

\begin{definition}\label{defn:convex metric space}
    A point $z$ in a metric space $(S,d)$ is said to be \emph{between} two points $x$ and $y$, if all three points are distinct, and $d(x,y)=d(x,z)+d(z,y)$.
	
    A metric space $(S,d)$ is \emph{convex} if for any two distinct points $x,y\in S$, there exists a point $z$ between $x$ and $y$.
\end{definition}

Given a pseudometric space $(S,d)$, its \emph{Hausdorff quotient} is the quotient space $(\bar S\coloneqq S/\sim,\bar d)$, where $x\sim y$ if and only if $d(x,y)=0$, and $\bar d$ is the natural induced metric on $\bar S$.

Recall that the \emph{length} of a continuous curve $\gamma:[a,b]\to S$ in a metric space $(S,d)$ is defined by
\[
    \ell(\gamma):=\sup_{\mathcal{P}}\sum_{i=1}^n d(x_{i-1},x_i),
\]
where the supremum is taken over all finite partitions $\mathcal{P}=\{a=x_0\le x_1\le\ldots\le x_n=b\}$ of the interval $[a,b]$.

\begin{definition}\label{defn:geodesic space}
    A metric space $(S,d)$ is a \emph{length space} if $d$ coincides with the \emph{intrinsic metric}, that is, $d(x,y)=\inf_\gamma \ell(\gamma)=:d_I(x,y)$ for any $x,y\in S$, where the infimum is taken over all continuous curves $\gamma:[0,1]\to S$ with $\gamma(0)=x$ and $\gamma(1)=y$.
	
    $(S,d)$ is a \emph{geodesic space} if the above infimum is a minimum, that is, the distance can always be calculated by a continuous curve.
\end{definition}


\subsection{Space of valuations and filtrations}\label{ssec:spaces}

From the previous sections we have the following subspaces of $\Val_{X,x}$:

\begin{itemize}
    \item the space of divisorial valuations, denoted by $\DivVal_{R,\fm}$,
	
    \item the space of quasi-monomial valuations when $R$ contains a field, denoted by $\QMVal_{R,\fm}$,
	
    \item the space of valuations with finite log discrepancy, when $x\in X=\Spec R$ is a klt singularity over a field of characteristic $0$, denoted by $\Val^{<+\infty}_{X,x}$, and
	
    \item the space of valuations with positive volume, denoted by $\Val^+_{X,x}$, which is equal to the space of linearly bounded valuations by Lemma \ref{lem:e>0 iff linearly bounded}.
	
\end{itemize}

Moreover, we have the following spaces of linearly bounded filtrations on $R$:

\begin{itemize}
    \item the space of linearly bounded $\fm$-filtrations, denoted by $\Fil_{R,\fm}$,
	
    \item the space of power-multiplicative $\fm$-filtrations which are linearly bounded, denoted by $\Fil^h_{R,\fm}$, and
	
    \item the space of saturated $\fm$-filtrations which are linearly bounded, denoted by $\Fil^s_{R,\fm}$.
\end{itemize}

We have the following inclusions
\begin{equation}\label{eqn:inclusion of valuation spaces}
\DivVal_{X,x}\subset\QMVal_{X,x}\subset
\Val^{<+\infty}_{X,x}\subset\Val^+_{X,x}\subset \Fil^s_{R,\fm}\subset \Fil^h_{R,\fm}\subset \Fil_{R,\fm}
\end{equation}
where the third inclusion follows from Li's properness estimate \cite[Theorem 1.4]{Li18}. 
By Lemma \ref{lem:saturation alt def}, to compute the saturation of a filtration, one can use any of the above spaces of valuations. 


\subsubsection{Weak topologies on $\Fil$}\label{sssec:weak topologies}

There are more ways to define a topology on the space $\Fil$ (as well as on $\Fil^s$ and $\Fil^h$), depending on the point of view. 

For example, a filtration $\fa_\bullet$ acts naturally on the ring $R$ as a function $\chi_{\fa_\bullet}$ (or dually). The product topology, or the ``weak-* topology'' induced by this action will be called the \emph{weak topology}, which is the coarsest topology such that for any $f\in R\backslash\{0\}$, the function $\phi_f:\Fil\to \bR_{\ge 0}$, $\fa_\bullet\mapsto \ord_{\fa_\bullet}(f)$ is continuous. In other words, a \emph{subbase} of the weak topology is given by $\{\phi_f^{-1}(U)\mid U\subset \bR_{\ge 0} \text{ is open, } f\in R\backslash\{0\}\}$.  Under this topology, all valuation spaces in \eqref{eqn:inclusion of valuation spaces} equipped with the weak topology, defined in \cite{JM12}, are subspaces of $\Fil$. It is not hard to check that a sequence $\{\fa_{\bullet,k}\}\subset\Fil$ converges weakly to $\fa_\bullet\in\Fil$ if and only if $\chi_k(f)\to\chi(f)$ for any $f\in R\backslash\{0\}$.

Similarly, a valuation has a natural action on the spaces of filtrations (or dually), that is, any $v\in\Val_{X,x}$ gives a function $\ev_v: \Fil\to \bR_{\ge 0}$, $\fa_\bullet\mapsto v(\fa_\bullet)$. Hence one can consider the ``weak topology'' induced by this action, that is, the coarsest topology on $\Fil^s$ such that $\ev_v$ is continuous for any $v\in X^{\an}$, $\Val_{X,x}$, $\Val^+_{X,x}$ or $\DivVal_{X,x}$, respectively. Similarly, one can check that a sequence $\{\fa_{\bullet,k}\}$ converges to $\fa_\bullet$ under these topologies if and only if $v(\fa_{\bullet,k})\to v(\fa_\bullet)$ for any $v$ in the corresponding valuation space. In this paper, we will mainly consider the topology defined using $\Val^+_{X,x}$, and call it the \emph{$+$-topology}.
The product topology defined using $\DivVal_{X,x}$ as above was considered in \cite{BdFFU15} for $b$-divisors. We will thus call it the \emph{coefficientwise topology}. 

\begin{remark}
    Since $v(\fa_\bullet)=v(\widetilde\fa_\bullet)$ by \cite[Proposition 3.9]{BLQ22} for $v\in\DivVal_{X,x}$, the functions $\ev_v$ do not distinguish $\fa_\bullet$ and its saturation. Hence the $+$-topology (and the coefficientwise topology) only makes sense on $\Fil^s$. Similar issues appear if we define the topology on $\Fil$ using other valuation spaces. But all the topologies are well-defined on $\Fil^s$.
\end{remark}

Later we will define metrics on the space $\Fil^s$. See Section \ref{ssec:comparison} for a comparison of these topologies.

\section{Darvas metrics on the space of filtrations}\label{sec:d_1}

In this section, we introduce the pseudometric $d_1$ and consider its first properties. 

\subsection{Darvas metrics and multiplicities}

Throughout this section we will use the language of filtrations. Recall that $\Fil=\Fil_{R,\fm}$ denotes the set of all linearly bounded $\fm$-filtrations on $R$.

\begin{defn}\label{defn:d_1 metric}
    We define a function $d_1:\Fil\times\Fil\to \bR_{\ge 0}$ as follows.
	
    Given $\fa_\bullet, \fb_\bullet\in\Fil$. If $\fa_\bullet\subset\fb_\bullet$, then let
    \begin{equation}\label{eqn:limit defn for d_1}
        d_1(\fa_\bullet,\fb_\bullet)=
        d_1(\fb_\bullet,\fa_\bullet):=\e(\fa_\bullet)-\e(\fb_\bullet).
    \end{equation}
	
    In general, define 		
        \begin{equation}\label{eqn:d_1 additive}
        d_1(\fa_\bullet,\fb_\bullet):=
        d_1(\fa_\bullet,\fa_\bullet\cap \fb_\bullet)+d_1(\fa_\bullet\cap \fb_\bullet,\fb_\bullet).
    \end{equation}
\end{defn}

\begin{proposition}\label{prop:d_1 is a pseudo metric}
    $d_1$ satisfies the triangle inequality, that is, for $\fa_\bullet$, $\fb_\bullet$,  $\fc_\bullet\in\Fil$, we have
    \begin{equation}\label{eqn:triangle}
        d_1(\fa_\bullet,\fc_\bullet)\le d_1(\fa_\bullet,\fb_\bullet)+d_1(\fb_\bullet,\fc_\bullet).
    \end{equation}
    Hence $d_1$ is a pseudometric on $\Fil$.
\end{proposition}

\begin{proof}
    By definition, we have
    \[
        d_1(\fa_\bullet,\fc_\bullet)=
        2\e(\fa_\bullet\cap\fc_\bullet)-\e(\fa_\bullet)-\e(\fc_\bullet)
    \]
    and 
    \[
        d_1(\fa_\bullet,\fb_\bullet)+
        d_1(\fb_\bullet,\fc_\bullet)=
        2\e(\fa_\bullet\cap\fb_\bullet)-\e(\fa_\bullet)-\e(\fb_\bullet)+2\e(\fb_\bullet\cap\fc_\bullet)-\e(\fb_\bullet)-\e(\fc_\bullet)
    \]
    Hence to prove \eqref{eqn:triangle}, it suffices to show that
    \[
        \e(\fa_\bullet\cap\fc_\bullet)\le \e(\fa_\bullet\cap\fb_\bullet)+\e(\fb_\bullet\cap\fc_\bullet)-\e(\fb_\bullet),
    \]
    which will follow immediately from the following comparison on colengths
    \begin{equation}\label{eqn:colength triangle}
        \ell(R/\fa_\lambda\cap\fc_\lambda)\le \ell(R/\fa_\lambda\cap\fb_\lambda)+\ell(R/\fb_\lambda\cap\fc_\lambda)-\ell(R/\fb_\lambda)
    \end{equation}
    for any $\lambda\in\bR_{>0}$. But we compute
    \begin{align*}
        \ell(R/\fa_\lambda\cap\fb_\lambda)+
        &\ell(R/\fb_\lambda\cap\fc_\lambda)-\ell(R/\fb_\lambda)-\ell(R/\fa_\lambda)=\ell(\fa_\lambda/\fa_\lambda\cap\fb_\lambda)+\ell(\fb_\lambda/\fb_\lambda\cap\fc_\lambda)\\
        =&\ell((\fa_\lambda+\fb_\lambda)/\fb_\lambda)+\ell(\fb_\lambda/\fb_\lambda\cap\fc_\lambda)\\
        =&\ell((\fa_\lambda+\fb_\lambda)/\fb_\lambda\cap\fc_\lambda)\\
        \ge&\ell((\fa_\lambda+\fb_\lambda)/(\fa_\lambda+\fb_\lambda)\cap\fc_\lambda)\\
        =&\ell((\fa_\lambda+\fb_\lambda+\fc_\lambda)/\fc_\lambda)\\
        \ge&\ell((\fa_\lambda+\fc_\lambda)/\fc_\lambda)\\
        =&\ell(\fa_\lambda/\fa_\lambda\cap\fc_\lambda)=\ell(R/\fa_\lambda\cap\fc_\lambda)-\ell(R/\fa_\lambda).
    \end{align*}
    This proves \eqref{eqn:colength triangle}, and hence the triangle inequality \eqref{eqn:triangle} holds. 
	
    Clearly $d_1(\fa_\bullet,\fa_\bullet)=0$, and $d_1(\fa_\bullet,\fb_\bullet)=d_1(\fb_\bullet,\fa_\bullet)$ by definition. So we know that $(\Fil,d_1)$ is a pseudometric space.
\end{proof}

Using the identification between $\Fil$ and $\cN$ Lemma \ref{deflem:fm-equivalence}, we get a pseudometric $d_1$ on $\cN$. The following lemma asserts that $(\Fil,d_1)$ is a \emph{rooftop metric space} in the sense of \cite{Xia19}, which is an analogue of \cite[Proposition 3.35]{Dar19} and \cite[Lemma 3.2]{BJ21}.

\begin{lemma}\label{lem:rooftop}
    For $\fa_\bullet,\fb_\bullet,\fc_\bullet\in\Fil$, we have
    \begin{equation}\label{eqn:strong triangle}
        d_1(\fa_\bullet\cap \fb_\bullet,\fa_\bullet\cap\fc_\bullet)\le d_1(\fb_\bullet,\fc_\bullet).
    \end{equation}
\end{lemma}

\begin{proof}
    By definition, we have
    \[
        d_1(\fb_\bullet,\fc_\bullet)= \e(\fb_\bullet\cap \fc_\bullet)-\e(\fb_\bullet)+
        \e(\fb_\bullet\cap \fc_\bullet)-\e(\fc_\bullet)
    \]
    and
    \[
        d_1(\fa_\bullet\cap\fb_\bullet,
        \fa_\bullet\cap\fc_\bullet)=
        \e(\fa_\bullet\cap\fb_\bullet\cap \fc_\bullet)-\e(\fa_\bullet\cap\fb_\bullet)+\e(\fa_\bullet\cap\fb_\bullet\cap \fc_\bullet)-\e(\fa_\bullet\cap\fc_\bullet).
    \]
    So by symmetry, it suffices to prove the inequality
    \[
        \e(\fa_\bullet\cap\fb_\bullet\cap \fc_\bullet)-\e(\fa_\bullet\cap\fb_\bullet)\le\e(\fb_\bullet\cap \fc_\bullet)-\e(\fb_\bullet),
    \]
    which follows from the inclusion 
    \begin{align*}
        (\fa_\lambda\cap\fb_\lambda)/
        (\fa_\lambda\cap\fb_\lambda\cap \fc_\lambda)\cong 
        &(\fa_\lambda\cap\fb_\lambda+
        \fc_\lambda)/\fc_\lambda\\
        \subset (\fb_\lambda+\fc_\lambda)/\fc_\lambda\cong &\fb_\lambda/(\fb_\lambda\cap\fc_\lambda).	
    \end{align*}	
    This finishes the proof. 
\end{proof}

\subsection{Measures and geodesics}\label{ssec:geodesic}

Recall that given $\fa_{\bullet,0},\fa_{\bullet,1}\in\Fil$, the \emph{geodesic} between them is defined to be the segment $\{\fa_{\bullet,t}\}_{t\in[0,1]}$ of $\fm$-filtrations, where
\[
    \fa_{\lambda,t}=\sum_{(1-t)\mu+t\nu=\lambda} \fa_{\mu,0}\cap\fa_{\nu,1}.
\]
The goal of this subsection is to justify the term \emph{geodesic}, in the sense that it computes the distance between $\fa_{\bullet,0}$ and $\fa_{\bullet,1}$. In other words, it is a distance-minimizing, continuous curve with respect to the metric $d_1$.

We first give a formula for the geodesic between two saturated filtrations, and use it to show that $\fa_{\bullet,t}\in\Fil^s$ when $\fa_{\bullet,0},\fa_{\bullet,1}\in\Fil^s$.

\begin{proposition}\label{prop:saturated geodesic is saturated}
    Let $\fa_{\bullet,0},\fa_{\bullet,1}\in\Fil^s$. Then for any $t\in(0,1)$, we have
    \begin{equation}\label{eqn:formula for geodesic}
        \fa_{\bullet,t}=\cap_{v\in\DivVal_{R,\fm}}\fa_\bullet(\frac{v}{c_{t,v}}),
    \end{equation}
    where 
    \[
        c_{t,v}\coloneqq\frac{v(\fa_{\bullet,0})v(\fa_{\bullet,1})}{tv(\fa_{\bullet,0})+(1-t)v(\fa_{\bullet,1})}.
    \]
    In particular, $\fa_{\bullet,t}\in\Fil^s$ for any $t\in(0,1)$.
\end{proposition}

\begin{proof}
    Denote the right-handed side of \eqref{eqn:formula for geodesic} by $\fc_{\bullet,t}$ and fix $\lambda\in\bR_{>0}$. First assume $f\in\fc_{\lambda,t}$. Then for any $v\in\DivVal_{R,\fm}$, $v(f)\ge \lambda c_{t,v}$. Since $\fa_{\bullet,0},\fa_{\bullet,1}\in\Fil^s$, this implies that $f\in\fa_{\mu,0}\cap\fa_{\nu,1}$, where $\mu\coloneqq \frac{\lambda v(\fa_{\bullet,1})}{tv(\fa_{\bullet,0})+(1-t)v(\fa_{\bullet,1})}$ and $\nu\coloneqq \frac{\lambda v(\fa_{\bullet,0})}{tv(\fa_{\bullet,0})+(1-t)v(\fa_{\bullet,1})}$. Since $(1-t)\mu+t\nu=\lambda$, we get $f\in\fa_{\lambda,t}$ and hence $\fc_{\bullet,t}\subset\fa_{\bullet,t}$.
	
    Conversely, if $f\in\fa_{\lambda,t}$, then for any $v\in\DivVal_{R,\fm}$, we have
    \begin{align*}
        v(f)\ge &v\left(\sum_{(1-t)\mu+t\nu=\lambda}\fa_{\mu,0}\cap\fa_{\nu,1}\right)=\min_{(1-t)\mu+t\nu=\lambda}\{v(\fa_{\mu,0}\cap\fa_{\nu,1})\}\\
        \ge&\min_{(1-t)\mu+t\nu=\lambda}\max\{v(\fa_{\mu,0}),v(\fa_{\nu,1})\}\\
        \ge&\min_{0\le\mu\le\lambda/(1-t)}\max\{\mu v(\fa_{\bullet,0}),\frac{\lambda-(1-t)\mu}{t}v(\fa_{\bullet,1})\}.
    \end{align*}
    Now it is not hard to check that the last minimum of the piecewise linear function of $\mu$ is achieved at $\mu_0:=\frac{\lambda v(\fa_{\bullet,1})}{tv(\fa_{\bullet,0})+(1-t)v(\fa_{\bullet,1})}\in(0,\lambda/(1-t))$, with value $\mu_0v(\fa_{\bullet,0})=\lambda c_{t,v}$. This shows that $f\in\fa_\lambda(v/c_{t,v})$. Since $v$ is arbitrary, we get $f\in\fc_{\lambda,t}$. Thus $\fa_{\bullet,t}\subset\fc_{\bullet,t}$ and the equality holds.
	
    It is easy to check that $\fa_{\bullet,t}\in\Fil$ for any $t\in(0,1)$ (see also the proof of Proposition \ref{prop:distance minimizing} below), hence by Proposition \ref{prop:alt def for saturation} we know that $\fa_{\bullet,t}\in\Fil^s$. The proof is finished.
\end{proof}

\subsubsection{Distance-minimizing property}

We first make the following easy but useful observation.

\begin{lemma}\label{lem:subinterval}
    Let $\fa_{\bullet,i}\in\Fil$ for $i=0,1$, and let $\{\fa_{\bullet,t}\}_{t\in[0,1]}$ be the geodesic between them. For $0\le t<t'\le 1$, we have $\fa_{\bullet,0}\cap\fa_{\bullet,t'}\subset\fa_{\bullet,t}$. 
\end{lemma}

\begin{proof}
    For any $\lambda\in\bR_{>0}$, we compute
    \begin{align*}
        \fa_{\lambda,t'}\cap\fa_{\lambda,0}=&\left(\sum_{(1-t')\mu+t'\nu=\lambda}\fa_{\mu,0}\cap\fa_{\nu,1}\right)\cap\fa_{\lambda,0}\\
        =&\sum_{\mu\le \lambda}\fa_{\lambda,0}\cap\fa_{\nu,1}+\sum_{\mu>\lambda}\fa_{\mu,0}\cap\fa_{\nu,1}\\
        =&\fa_{\lambda,0}\cap\fa_{\lambda,1}+\sum_{\mu>\lambda}\fa_{\mu,0}\cap\fa_{\nu,1}.
    \end{align*}
	
    Since $\fa_{\lambda,0}\cap\fa_{\lambda,1}\subset\fa_{\lambda,s}$ for any $s\in[0,1]$, it suffices to show that for any $(\mu,\nu)\in\bR_{\ge 0}^2$ satisfying $(1-t')\mu+t'\nu=\lambda$ and $\mu>\lambda$, we have
    \begin{equation}\label{eqn:reduce mu}
        \fa_{\mu,0}\cap\fa_{\nu,1}\subset\fa_{\lambda,t}.
    \end{equation}
    To see this, let $\mu':=\frac{(1-t')\mu+(t'-t)\nu}{1-t}$, then $(1-t)\mu+t\nu=(1-t')\mu+t'\nu=\lambda$, and we have
    \[
        \mu-\mu'=\frac{(1-t)\mu-(1-t')\mu-(t'-t)\nu}{1-t}=\frac{(t'-t)(\mu-\nu)}{1-t}>0,
    \]
    where the last inequality follows from the condition $\mu>\lambda>\nu$. This implies
    \[
        \fa_{\mu,0}\cap\fa_{\nu,1}\subset \fa_{\mu',0}\cap\fa_{\nu,1}\subset \fa_{\lambda,t},
    \]
    hence \eqref{eqn:reduce mu} holds and the proof is finished.
\end{proof}

Next we prove that the geodesic is always distance-minimizing.

\begin{proposition}\label{prop:distance minimizing}
    Let $\fa_{\bullet,0},\fa_{\bullet,1}\in\Fil$ and let $\{\fa_{\bullet,t}\}_{t\in[0,1]}$ be the geodesic between them. Then for any $0\le t<t'\le 1$, we have
    \begin{equation}\label{eqn:distance minimizing}
        d_1(a_{\bullet,0},\fa_{\bullet,t'})=
        d_1(\fa_{\bullet,0},\fa_{\bullet,t})+
        d_1(\fa_{\bullet,t},\fa_{\bullet,t'}).
    \end{equation}
\end{proposition}

\begin{proof}
    For simplicity, denote $\fb_\bullet:=\fa_{\bullet,0}\cap\fa_{\bullet,t}$, $\fb'_\bullet:=\fa_{\bullet,0}\cap\fa_{\bullet,t'}$ and $\fc_\bullet:=\fa_{\bullet,t}\cap\fa_{\bullet,t'}$. By Lemma \ref{lem:subinterval} we know that $\fb'_\bullet\subset\fb_\bullet\subset\fa_{\bullet,0}$ and $\fb'_\bullet\subset\fc_\bullet\subset\fa_{\bullet,t'}$. We claim that for any $\lambda\in\bR_{>0}$, 
    \begin{equation}\label{eqn:parallelogram}
        \fa_{\lambda,t}/\fb_\lambda\cong \fc_\lambda/\fb'_\lambda.
    \end{equation}
	
    Granting the claim for now, we know that $d_1(\fb_\bullet,\fa_{\bullet,t})=d_1(\fb'_\bullet,\fc_\bullet)$. So we get
    \[
        d_1(\fb_\bullet,\fa_{\bullet,t})+
        d_1(\fa_{\bullet,t},\fc_\bullet)=
        d_1(\fb'_\bullet,\fc_\bullet)+
        d_1(\fa_{\bullet,t},\fc_\bullet)
        =d_1(\fa_{\bullet,t},\fb'_\bullet)
        =d_1(\fa_{\bullet,t},\fb_\bullet)+
        d_1(\fb_\bullet,\fb'_\bullet),
    \]
    which implies $d_1(\fb_\bullet,\fb'_\bullet)=d_1(\fa_{\bullet,t},\fc_\bullet)$. Now we compute
    \begin{align*}
        d_1(a_{\bullet,0},\fa_{\bullet,t'})=
        &d_1(\fa_{\bullet,0},\fb'_\bullet)+
        d_1(\fb'_\bullet,\fa_{\bullet,t'})\\
        =&d_1(\fa_{\bullet,0},\fb_\bullet)+d_1(\fb_\bullet,\fb'_\bullet)+d_1(\fb'_\bullet,\fc_\bullet)+d_1(\fc_\bullet,\fa_{\bullet,t'})\\
        =&d_1(\fa_{\bullet,0},\fb_\bullet)+d_1(\fa_{\bullet,t},\fc_\bullet)+d_1(\fb_\bullet,\fa_{\bullet,t})+d_1(\fc_\bullet,\fa_{\bullet,t'})\\
        =&d_1(\fa_{\bullet,0},\fa_{\bullet,t})+d_1(\fa_{\bullet,t},\fa_{\bullet,t'})
        \end{align*}
	
    It remains to prove the claim \eqref{eqn:parallelogram}. Indeed, by definition, any $f\in\fa_{\lambda,t}$ can be written as
    \[
        f=\sum_{(1-t)\mu+t\nu=\lambda}f_{\mu,\nu},
    \]
    where $f_{\mu,\nu}\in\fa_{\mu,0}\cap\fa_{\nu,1}$. If $\mu\ge \lambda$, then $f_{\mu,\nu}\in\fa_{\mu,0}\subset \fa_{\lambda,0}$. If $\mu\le\lambda$, then $\nu\ge \lambda$, and hence $f_{\mu,\nu}\in\fa_{\nu,1}\subset \fa_{\lambda,1}$. Thus we have 
    \[
        f=\sum_{\mu\le\lambda}f_{\mu,\nu}+
        \sum_{\mu\ge\lambda}f_{\mu,\nu}\in
        \fa_{\lambda,t}\cap\fa_{\lambda,1}+\fa_{\lambda,0},
    \]
    and hence $\fc_\lambda\subset\fa_{\lambda,t}\subset\fa_{\lambda,t}\cap\fa_{\lambda,1}+\fa_{\lambda,0}\subset\fc_\lambda+\fa_{\lambda,0}$, where the last inclusion follows from Lemma \ref{lem:subinterval}. So we get
    \[
        \fa_{\lambda,t}/\fb_\lambda\cong
        (\fa_{\lambda,t}+\fa_{\lambda,0})/\fa_{\lambda,0}=(\fc_\lambda+\fa_{\lambda,0})/\fa_{\lambda,0}\cong\fc_\lambda/(\fc_\lambda\cap\fa_{\lambda,0})=\fc_\lambda/\fb'_\lambda,
    \]
    where the last equality follows again from Lemma \ref{lem:subinterval}. This proves \eqref{eqn:parallelogram}, and finishes the proof of the proposition.
\end{proof}

\begin{remark}
    When $R$ contains a field, the proposition can be proved (easily) using the measure as in the next section. We present the above more tedious but elementary argument, which does not rely on this assumption.
\end{remark}

\subsubsection{Duistermaat-Heckman measures and continuity}

In this subsection, we assume that $R$ contains a field $\bk$. We first briefly recall the measure $\mu$ introduced in \cite[Section 4]{BLQ22} encoding multiplicities and its properties.

Given $\fa_{\bullet,0},\fa_{\bullet,1}\in\Fil$. Define the function $H:\bR^2\to \bR_{\ge 0}$ by
\[
    H(x,y):=\e(\fa_{x\bullet,0}\cap\fa_{y\bullet,1}),
\]
where we adopt the convention $\e(R)=0$. Fix $D\in\bZ_{>1}$ such that $\fa_{D\bullet,0}\subset \fa_{\bullet,1}$ and $\fa_{D\bullet,1}\subset\fa_{\bullet,0}$. 

\begin{lemma}\cite{BLQ22}\label{lem:properties of the measure}
    The distributional derivative 
    \[
        \mu:=-\frac{\partial^2 H}{\partial x \partial y}
    \]
    is a measure on $\bR^2$ and for any $a,b\in\bR_{>0}$, we have
    \begin{enumerate}
        \item $\mu(\{(1-t)x+ty< a \})=\e(\fa_{a\bullet,t})=a^n\e(\fa_{\bullet ,t})$,
		
        \item $\mu\left( \{x<a\} \cup \{y<b\}\right)= \e(\fa_{a\bullet,0 } \cap \fa_{b\bullet,1})$, 
		
        \item $\mu$ is homogeneous of degree $n$, that is, $\mu(cA))=c^n\mu(A)$ for any Borel set $A\subset \bR^2$ and $c\in\bR_{>0}$, and
		
        \item $\supp(\mu)\subset \{(x,y)\in\bR_{\ge 0}^2\mid \frac{1}{D}x\le y\le Dx\}$.
    \end{enumerate}
\end{lemma}

\begin{proof}
    The first two statements are \cite[Proposition 4.7]{BLQ22} and the last two are \cite[Proposition 4.9]{BLQ22}.
\end{proof}

For a rectangle, the measure $\mu$ is controlled by the Lebesgue measure.

\begin{lemma}\label{lem:absolute continuity}
    Fix $(a,b)\in\bR^2$. Let $A:=[a,a+\alpha]\times [b,b+\beta]\subset \bR^2$ be a rectangle, where $\alpha,\beta\in\bR_{>0}$. Then there exists $M:=M(a,b)\in\bR_{>0}$ such that $\mu(A)\le M\alpha\beta$ when $\alpha,\beta\le 1$.
\end{lemma}

\begin{proof}
    We may assume that $(a,b)\in\bR_{\ge 0}^2$. Using the inequality 
    \[
        (x+\delta)^n-x^n\le n(x+\delta)^{n-1} \delta
    \]
    for $x\ge 0$ and $\delta>0$, we compute 
    \begin{align*}
    \mu(A)=& \mu([0,a+\alpha]\times[b,b+\beta])-\mu([0,a]\times[b,b+\beta])\\
    =& ((a+\alpha)^n-a^n)\mu([0,1]\times [b,b+\beta])\\
    =& ((a+\alpha)^n-a^n)(((b+\beta)^n-b^n)\mu([0,1]\times[0,1])\\
    \le& ((a+\alpha)^n-a^n)(((b+\beta)^n-b^n)\e(\fa_{\bullet,0}\cap\fa_{\bullet,1})\\
    \le& n^2(a+1)^{n-1}(b+1)^{n-1}\e(\fa_{\bullet,0}\cap\fa_{\bullet,1})\cdot\alpha\beta,
    \end{align*}
    where the second and the third equality follow from Lemma \ref{lem:properties of the measure}(3), and the first inequality follows from Lemma \ref{lem:properties of the measure}(2). Hence the proof is finished with $M:=n^2(a+1)^{n-1}(b+1)^{n-1}\e(\fa_{\bullet,0}\cap\fa_{\bullet,1})$.
\end{proof}

Now we can prove that the function $t\mapsto \fa_{\bullet,t}$ is Lipschitz continuous with respect to $d_1$. 

\begin{lemma}\label{lem:continuity of geodesic}
    Suppose $R$ contains a field. Given $\fa_{\bullet,0},\fa_{\bullet,1}\in\Fil$, the map $[0,1]\to\Fil$, $t\mapsto\fa_{\bullet,t}$ is Lipschitz continuous, where $[0,1]$ is equipped with the Euclidean topology and $\Fil$ is equipped with the topology induced by the pseudometric $d_1$.
\end{lemma}

\begin{proof}
    Given $0\le t_1<t_2\le 1$, denote $\Delta_j:=\{(x,y)\in\bR^2_{>0}\mid (1-t_j)x+t_jy\le 1\}$ for $j=1,2$. An easy calculation using Lemma \ref{lem:properties of the measure}(1) gives
    \begin{equation}\label{eqn:geodesic distance}
        d_1(\fa_{\bullet,t_1},\fa_{\bullet,t_2})=
        \mu(\Delta_1\backslash\Delta_2)+
        \mu(\Delta_2\backslash\Delta_1).
    \end{equation}
	
    First assume that $0<t_1<t_2<1$. By Lemma \ref{lem:properties of the measure}(4) we have
    \[
        \mu(\Delta_1\backslash\Delta_2)
        =\mu((\Delta_1\backslash\Delta_2)\cap\{y\le Dx\})=\mu(\Delta)
    \]
    where $\Delta:=\{(x,y)\in\bR^2_{\ge 0}\mid 0\le x\le 1,y\le Dx, (1-t_2)x+t_2y\le 1\le (1-t_1)x+t_1y\})$. By an elementary calculation, the area of $\Delta$ is at most
    \[
        \mathrm{Area}(\Delta)\le S:=\frac{D(D-1)(t_2-t_1)}{(1-t_1+t_1D)(1-t_2+t_2D)}\le D(D-1)(t_2-t_1).
    \]
    Note that $(x,y)\in\Delta$ satisfies $x\le 1$ and $y\le D$. So $\Delta$ can be covered by finitely many rectangles $R_j$ of the form $[a_j,a_j']\times [b_j,b_j']$, satisfying 
    \[
        a_j\le 1,\ b_j\le D, \text{ and } \sum_j \mathrm{Area}(R_j)\le 2S.
    \]
    Hence we may apply Lemma \ref{lem:absolute continuity} to get
    \begin{align*}
        \mu(\Delta_1\backslash\Delta_2)
        \le &n^22^{n-1}(D+1)^{n-1}\e(\fa_{\bullet,0}\cap\fa_{\bullet,1})\sum_j \mathrm{Area}(R_j)\\
        =&n^2 2^n(D+1)^{n-1}\e(\fa_{\bullet,0}\cap\fa_{\bullet,1})S\\
        \le&n^2 2^n(D+1)^{n-1}D(D-1)\e(\fa_{\bullet,0}\cap\fa_{\bullet,1})(t_2-t_1)
    \end{align*}
    Similarly $\mu(\Delta_1\backslash\Delta_1)$ satisfies the same estimate. Hence by \eqref{eqn:geodesic distance}, we have
    \begin{equation}\label{eqn:Lip mid}
        d_1(\fa_{\bullet,t_1},\fa_{\bullet,t_2})\le n^2 2^{n+1}(D+1)^{n-1}D(D-1)\e(\fa_{\bullet,0}\cap\fa_{\bullet,1})(t_2-t_1).
    \end{equation}
	
    Now assume that $t_1=0$. We may assume $t_2<\frac{1}{2}$. We have
    \begin{align*}
        \mu(\Delta_1\backslash\Delta_2)
        =& \mu((\Delta_1\backslash\Delta_2)\cap\{y\le Dx\})\\
        \le &\mu([\frac{1-t_2D}{1-t_2},1]\times [1,D])\\
        \le &n^2 (\frac{1-t_2D}{1-t_2})^{n-1} 2^{n-1} (D-1)\frac{(D-1)t_2}{1-t_2}\\
        \le &n^2 2^{2n-1}(D-1)^2 \e(\fa_{\bullet,0}\cap\fa_{\bullet,1})t_2,
    \end{align*}
    where the first equality follows from Lemma \ref{lem:properties of the measure}(4) and the second inequality follows from Lemma \ref{lem:absolute continuity}, and
    \begin{align*}
        \mu(\Delta_2\backslash\Delta_1)
        \le&\mu([1,\frac{1}{1-t_2}]\times[0,1])\\
        \le &n^2 2^{n-1}\frac{t_2}{1-t_2}\\
        \le &n^2 2^n\e(\fa_{\bullet,0}\cap\fa_{\bullet,1}) t_2
    \end{align*}
    where the second inequality follows from Lemma \ref{lem:absolute continuity}. Hence
    \begin{equation}\label{eqn:Lip end}
        d_1(\fa_{\bullet,0},\fa_{\bullet,t_2})\le n^2 2^n (2^{n-1}(D-1)^2+1)\e(\fa_{\bullet,0}\cap\fa_{\bullet,1})t_2.
    \end{equation}
	
    The case when $t_2=1$ can be treated similarly. Combining the cases together, we get 
    \[
        d_1(\fa_{\bullet,t_1},\fa_{\bullet,t_2})\le L(t_2-t_1).
    \]
    where the Lipschitz constant can be taken to be 
    \[
        L:=n^2 2^{n+1} (D+1)^{n-1}D(D-1)\e(\fa_{\bullet,0}\cap\fa_{\bullet,1}),
    \]
    where we used the assumption that $D\in\bZ_{>1}$. 
\end{proof}

\begin{remark}
    It is desirable to have a proof for Lemma \ref{lem:continuity of geodesic} without involving the measure, and thus one can remove the assumption that $R$ contains a field.
\end{remark}

\subsection{Continuity along monotonic sequences}\label{ssec:monitonic}


In this subsection we prove some convergence results along monotonic sequences. The results are partially inspired by the study of the completion of certain metric spaces in the global setting, for example, \cites{Dar17,Xia19,BJ21}. 

\subsubsection{Okounkov bodies and continuity along increasing sequences}

In this section, using the local variant of Okounkov bodies introduced in \cites{LM09,KK14}, we prove that the multiplicity function $\e(\bullet)$ is continuous along increasing sequences of filtrations. 


Given a semigroup $S\subset\bZ^{n+1}_{\ge 0}$, denote its \emph{section at} $m\in\bZ_{>0}$ by $S_m:=S\cap (\bZ^n_{\ge 0}\times \{m\})$. Recall that the \emph{closed convex cone} $C(S)\subset \bR^{n+1}_{\ge 0}$ of $S$ is defined to be the closure of the set $\{\sum_i a_is_i\mid a_i\in\bR_{\ge 0},\ s_i\in S\}$, and its \emph{Newton-Okounkov body} is
\[
    \Delta(S):=C(S)\cap (\bR^n_{\ge 0}\times \{1\}).
\]

\begin{lemma}\cite{LM09}\label{lem:Okounkov body}
    Let $\Gamma\subset \bZ_{\ge 0}^{n+1}$ be a semigroup. If $\Gamma$ satisfies the following
    \begin{enumerate}
        \item $\Gamma_0=\{\bm{0}\}$,
		
        \item there exists finitely many vectors $(v_i,1)\in\bZ_{\ge0}^{n+1}$ generating a semigroup $B\subset \bZ_{\ge 0}^{n+1}$, such that $\Gamma\subset B$, and
		
        \item $\Gamma$ contains a set of generators of $\bZ^{n+1}$ as a group,
    \end{enumerate}
    then 
    \[
        \lim_{m\to\infty}\frac{\#\Gamma_m}{m^n}=\vol(\Delta(\Gamma)).
    \]
\end{lemma}

We need the following formula, which relates the co-length of a primary ideal to the points of its semigroups.

\begin{proposition}\cite{Cut14}\label{prop:genral Okounkov body}
    Let $(R,\fm)$ be a Noetherian analytically irreducible local domain of dimension $n$, and $\fa_\bullet\in\Fil_{R,\fm}$. Then there exists $t=t(R)\in\bZ_{>0}$ depending only on $R$, and $2t$ semigroups $\bar\Gamma\supj,\Gamma\supj\subset\bZ_{\ge0}^{n+1}$, $j=1,\ldots,t$ satisfying the conditions of Lemma \ref{lem:Okounkov body}, such that 
    \begin{equation}\label{eqn:point counting}
        \ell(R/\fa_m)=\sum_{j=1}^t(\#\bar\Gamma\supj_m-\#\Gamma\supj_m).
    \end{equation}
\end{proposition}

\begin{proof}
    We include a sketch of the constructions for the reader's convenience. For more details about the proof, we refer to \cite[Section 4]{Cut14}. 
	
    First note that we may replace $R$ by its $\fm$-adic completion $\hat R$ and each $\fa_\lambda$ by $\fa_\lambda\cdot R$ since $\ell(R/\fa_\lambda)=\ell(\hat R/\fa_\lambda\cdot R)$. Thus we may assume that $R$ is complete, and hence excellent. Let $\pi:Z\to\Spec R$ be the normalized blowup of $\fm$. Then there exists a closed point $z\in \pi^{-1}(\fm)$ such that $R':=\cO_{Z,z}$ is regular since $Z$ is normal. 
    Now $\pi$ is of finite type since $R$ is Nagata. Hence $R'$ is essentially of finite type over $R$. Denote $\kappa':=R'/\fn$, where $\fn$ is the maximal ideal of $R'$, then we know that $t:=[\kappa':\kappa]\in\bZ_{>0}$.
	
    We proceed to define a valuation $v$ on the quotient field $K:=Q(R)=Q(R')$. Fix a regular system of parameters $(x_1,\ldots,x_n)$ of $R'$ and $\alpha_1,\ldots,\alpha_n\in \bR_{>1}$, which are linearly independent over $\bQ$. 
    By Cohen's structure theorem, the $\fn$-completion $(\widehat{R'},\hat \fn)$ of $R'$ has a coefficient ring $C$ as a regular local ring. Hence for any $f\in R'$ and $d\in\bZ_{>0}$, we can write
    \[
        f=\sum_{i_1+\ldots i_n<d} c_{i_1,\ldots,i_n}x_1^{i_1}\cdots x_n^{i_n}+g_d\in \widehat{R'},
    \]
    where $c_{i_1,\ldots,i_n}\in C$ are units and $g_d\in \hat\fn^d$. Take $d_0$ such that $c_{I_1,\ldots,I_n}\ne 0$ for some $I_1+\ldots+I_n<d_0$. It is not hard to check that if $d>d_0+\sum \alpha_j I_j$, then the number 
    \[
        v(f):=\min\{\sum_{j=1}^n \alpha_j i_j\mid c_{i_1,\ldots,i_n}\ne 0\}
    \]
    is independent of the choice of the units and $d$, hence this defines a function $v:K^{\times}\to\bR$ by $v(f/g)=v(f)-v(g)$ for any $f,g\in R'$. 
    One can check that $v$ is a real valuation. Denote its value group by $\Gamma_v:=v(K^\times)$ and its valuation ring by $(\cO_v,\fm_v)$ and for any $\lambda\in\bR_{\ge 0}$, denote
    \[
        \fb_\lambda:=\{f\in\cO_v\mid v(f)\ge\lambda\} \text{ and } \fb_{>\lambda}:=\{f\in\cO_v\mid v(f)>\lambda\}.
    \]
	
    Since $\alpha_j$ are linearly independent over $\bQ$, we know that 
    \begin{equation*}
        \fb_\lambda/\fb_{>\lambda}=\left\{\begin{aligned}
        &R'/\fn=\kappa', &\text{if } \lambda\in\Gamma_v,\\
        &0, &\text{otherwise.}
        \end{aligned}\right.
    \end{equation*}
    Hence for any $m\in\bZ_{>0}$ and $\lambda\in\bR_{>0}$, $\dim_\kappa \fa_m\cap\fb_\lambda/\fa_m\cap\fb_{>\lambda}\le \dim_\kappa R\cap\fb_\lambda/R\cap\fb_{>\lambda}\le [\kappa':\kappa]=t$.
	
    Fix $c_1\in\bZ_{>0}$ such that $\fm^{c_1}\subset\fa_1$. By \cite[Lemma 4.3]{Cut13}, there exists $c_2\in\bZ_{>0}$ such that $\fb_{c_2m}\cap R\subset \fm^m$ for any $m\in\bZ_{>0}$. Set $c:=c_1c_2$. Then for any $m\in\bZ_{>0}$ we have
    \begin{equation}\label{eqn:Izumi control}
        \fb_{cm}\cap R\subset \fm^{c_1m}\subset \fa_1^m\subset \fa_m.
    \end{equation}
	
    To simplify the notation, for $\bm\beta:=(\beta_1,\ldots,\beta_n)\in\bR_{\ge 0}^n$, we denote $\xi(\bm\beta):=\sum_{i=1}^n \alpha_i\beta_i$ and $|\bm\beta|:=\sum_{i=1}^n \beta_i$. Now for $1\le j\le t$, define
    \[
        \Gamma^{(j)}:=\{(\bm\beta,m)\in\bZ_{\ge 0}^{n+1}\mid \dim_\kappa \fa_m\cap\fb_{\xi(\bm\beta)}/\fa_m\cap\fb_{>\xi(\bm\beta)}\ge j \text{ and } |\bm\beta|\le cm\}
    \]
    and
    \[
        \bar\Gamma^{(j)}:=\{(\bm\beta,m)\in\bZ_{\ge 0}^{n+1}\mid \dim_\kappa R\cap\fb_{\xi(\bm\beta)}/R\cap\fb_{>\xi(\bm\beta)}\ge j \text{ and } |\bm\beta|\le cm\}.
    \]
    One can check that $\Gamma^{(j)}$ and $\bar\Gamma^{(j)}$ are semigroups satisfying the conditions of Lemma \ref{lem:Okounkov body}.

    \begin{align*}
        \ell(R/\fa_m)=&\ell(R/R\cap\fb_{cm})-\ell(\fa_m/\fa_m\cap\fb_{cm})\\
        =&\sum_{0\le\lambda<cm} \dim_\kappa (R\cap\fb_{\lambda}/R\cap\fb_{>\lambda})-\sum_{0\le\lambda<cm} \dim_\kappa (\fa_m\cap\fb_{\lambda}/\fa_m\cap\fb_{>\lambda})\\
        =&\sum_{\bm\beta\in\bZ^n_{\ge 0},\xi(\bm\beta)<cm} \dim_\kappa (R\cap\fb_{\xi(\bm\beta)}/R\cap\fb_{\xi(\bm\beta)})\\
        &-\sum_{\bm\beta\in\bZ^n_{\ge 0},\xi(\bm\beta)<cm} \dim_\kappa (\fa_m\cap\fb_{\xi(\bm\beta)}/\fa_m\cap\fb_{\xi(\bm\beta)})\\
        =&\sum_{j=1}^t\#\{\bm\beta\in\bZ^n_{\ge 0}\mid \dim_\kappa (R\cap \fb_{\xi(\bm\beta)}/R\cap \fb_{>\xi(\bm\beta)})\ge j,\ \xi(\bm\beta)<cm\}\\
        &-\sum_{j=1}^t\#\{\bm\beta\in\bZ^n_{\ge 0}\mid \dim_\kappa (\fa_m\cap \fb_{\xi(\bm\beta)}/\fa_m\cap \fb_{>\xi(\bm\beta)})\ge j,\ \xi(\bm\beta)<cm\}\\
        =&\sum_{j=1}^t\#\bar\Gamma^{(j)}_m-\sum_{j=1}^t\#\Gamma^{(j)}_m
    \end{align*}
    where the first equality follows from \eqref{eqn:Izumi control}, an the last is because when $|\bm\beta|\le cm$ but $\xi(\bm\beta)\ge cm$, by \eqref{eqn:Izumi control} again we have
    \[
        \fb_{\xi(\bm\beta)}\cap R\subset \fb_{cm}\cap R\subset \fa_m.
    \]
    This proves \eqref{eqn:point counting} and finishes the proof.
\end{proof}

It is not hard to see that taking the closed convex cone commutes with taking the union.

\begin{lemma}\label{lem:commute}
    Let $\{\Gamma_k\subset \bR^{n+1}_{\ge 0}\}$ be an increasing sequence of semigroups with respect to inclusion. Then $C(\cup_k \Gamma_k)=\overline{\cup_k C(\Gamma_k)}$. 
	
    In particular, $\vol(\Delta(\cup_k \Gamma_k))=\vol(\cup_k \Delta(\Gamma_k))=\lim_{k\to\infty} \vol(\Delta(\Gamma_k))$.
\end{lemma}

\begin{proof}
    Since $C(\Gamma_k)\subset C(\cup_k \Gamma_k)$ for any $k\in\bZ_{>0}$, we know that $\cup_k C(\Gamma_k)\subset C(\cup_k \Gamma_k)$. So $\overline{\cup_k C(\Gamma_k)}\subset C(\cup_k \Gamma_k)$. 
	
    To prove the reverse inclusion, it suffices to note that $\cup_k C(\Gamma_k)$ is a convex cone, which implies $\overline{\cup_k C(\Gamma_k)}=C(\cup_k C(\Gamma_k))\supset C(\cup_k \Gamma_k)$. 
	
    The last equality follows immediately.
\end{proof}

Now we are ready to show that the multiplicity is continuous along increasing sequences. The reader may want to compare the following result with similar statements in the global setting, for example, \cite[Proposition 6.11]{Dar17} and \cite[Lemma 4.7]{BJ21}.

\begin{proposition}\label{prop:continuous along increasing sequences}
    Let $\{\fa_{\bullet,k}\}_{k\in\bZ_{>0}}$ be an increasing sequence in $\Fil_{R,\fm}$ such that $\cup \fa_{\bullet,k}\in\Fil_{R,\fm}$. Then
    \[
        \lim_{k\to\infty} d_1(\fa_\bullet,\fa_{\bullet,k})=0,
    \]
    where $\fa_\bullet:=\{\cup_k \fa_{\lambda,k}\}_{\lambda\in\bR_{>0}}$.
\end{proposition}

\begin{proof}
    We construct $2t$ sequence of semigroups $\Gamma_{,k}^{(j)}$ and $\bar\Gamma_{,k}^{(j)}$ associated to $\fa_{\bullet,k}$ as in Proposition \ref{prop:genral Okounkov body}. Let $\Gamma^{(j)}$ and $\bar\Gamma^{(j)}$ be the semigroups of $\fa_\bullet$. By the proof of Proposition \ref{prop:genral Okounkov body} it is easy to see that $\bar\Gamma_{,k}^{(j)}$ is indeed independent of $k$, that is, $\bar\Gamma_{,k}^{(j)}=\bar\Gamma^{(j)}$ for any $1\le j\le t$ and $k\in\bZ_{>0}$. We claim that for any $1\le j\le t$,
    \begin{equation}
        \Gamma^{(j)}=\cup_k \Gamma_{,k}^{(j)}.
    \end{equation}
	
    Indeed, by definition $\Gamma_{,k}^{(j)}\subset\Gamma^{(j)}$ for any $k$ and $j$, so $\cup_k \Gamma_{,k}^{(j)}\subset\Gamma^{(j)}$. Conversely, for any $(\bm\beta,m)\in\Gamma^{(j)}$ where $1\le j\le t$, we know that
    \[
        \dim_\kappa \fa_m\cap\fb_{\xi(\bm\beta)}/\fa_m\cap\fb_{>\xi(\bm\beta)}\ge j.
    \]
    Choose $\kappa$-linearly independent elements $\bar{f_1},\ldots,\bar{f_j}\in \fa_m\cap\fb_{\xi(\bm\beta)}/\fa_m\cap\fb_{>\xi(\bm\beta)}$ and let $f_1,\ldots,f_j$ be a lifting to $\fa_m\cap\fb_{\xi(\bm\beta)}$. Since $\fa_m=\cup_k \fa_{m,k}$, there exists $K\in\bZ_{>0}$ such that $f_1,\ldots,f_j\in\fa_{m,K}$. Now their image in $\fa_{m,K}\cap\fb_{\xi(\bm\beta)}/\fa_{m,K}\cap\fb_{>\xi(\bm\beta)}$ must be $\kappa$-linearly independent, which implies
    \[
        \dim_\kappa \fa_{m,K}\cap\fb_{\xi(\bm\beta)}/\fa_{m,K}\cap\fb_{>\xi(\bm\beta)}\ge j.
    \]
    This shows that $(\bm\beta,m)\in\Gamma_{,K}^{(j)}$. Thus we get $\Gamma^{(j)}\subset \cup_k \Gamma_{,k}^{(j)}$ and the proof of the claim is finished.
	
    Now we can apply Lemma \ref{lem:Okounkov body}, Proposition \ref{prop:genral Okounkov body} and Lemma \ref{lem:commute} to get
    \begin{align*}
        \lim_{k\to\infty} d_1(\fa_\bullet,\fa_{\bullet,k})=&\lim_{k\to\infty} (\e(\fa_{\bullet,k})-\e(\fa_\bullet))\\
        =&\lim_{k\to\infty}\lim_{m\to\infty}\frac{\ell(R/\fa_{m,k})-\ell(R/\fa_m)}{m^n/n!}\\
        =&n!\lim_{k\to\infty}\lim_{m\to\infty} \sum_{j=1}^t\frac{(\#\bar\Gamma^{(j)})-\#\Gamma_{,k}^{(j)})-(\#\bar\Gamma^{(j)}-\#\Gamma^{(j)})}{m^n}\\
        =&n!\sum_{j=1}^t\lim_{k\to\infty}(\vol(\Delta(\Gamma^{(j)}))-\vol(\Delta(\Gamma_{,k}^{(j)})\\
        =0.
    \end{align*}
    This finishes the proof.
\end{proof}

\begin{remark}
    When $x\in X$ is an isolated normal singularity, Proposition \ref{prop:continuous along increasing sequences} can be obtained from \cite[Theorem 4.14]{BdFF12}, since the corresponding $b$-divisors form a decreasing convergent sequence in the pointwise topology.
\end{remark}

The following example shows that $\e(\bullet)$ may fail to be continuous along a decreasing sequence of filtrations. However, we will show that continuity along decreasing sequences holds under some geometric conditions in the next section.

\begin{example}\label{eg:jumping up along decreasing sequences}
    Let $R=k[\![x]\!]$ with $\fm=(x)$. For $k\in\bZ_{>0}$, define $\fa_{\bullet,k}\in\Fil$ by
    \begin{equation*}
        \fa_{\lambda,k}:=
        \left\{\begin{aligned}
        &\fm^{\lceil 2\lambda \rceil}, &\lambda\le k,\\
        &\fm^{2k}, &k<\lambda\le 2k,\\
        &\fm^{\lceil \lambda \rceil}, &\lambda> 2k.
        \end{aligned}\right.
    \end{equation*}
    Then $\e(\fa_{\bullet,k})=\e(\fm^\bullet)=1$. But since $\cap_k \fa_{\bullet,k}=\fm^{2\bullet}$, we have $\e(\cap_k \fa_{\bullet,k})=\e(\fm^{2\bullet})=2\e(\fm^\bullet)=2$.  
\end{example}

As an application of Proposition \ref{prop:continuous along increasing sequences}, we prove a convergence result which might be of its own interest. 

\begin{proposition}\label{prop:modified d_1-converging seq}
    Let $x\in X$ be a klt singularity over a field of characteristic $0$. Let $v_k\in\Val_{X,x}$, $k\in\bZ_{>0}$ be a sequence of valuations such that $A_X(v_k)=1$ and $\nvol_{x,X}(v_k)<V$ for some $V>0$. Possibly passing to a subsequence, there exists an increasing sequence $\{\fa_{\bullet,k}\}\subset\Fil^s$ with $\fa_{\lambda,k}\subset\fa_\lambda(v_k)$, which converges to some $v\in\Val_{X,x}$ both weakly and with respect to the $d_1$-metric. 
\end{proposition}


\begin{proof}
    By \cite[Theorem 1.1]{Li18}, there exists $\delta>0$ such that $v_k(\fm)>\delta$ for any $k$. By \cite[Proposition 3.9]{LX16},
    the set
    \[
        \{v\in\Val_{X,x}\mid v(\fm)\ge\delta,\ A_X(v)\le 1\}
    \]
    is sequentially compact under the weak topology. Hence there is a subsequence of $\{v_k\}$ which converges to some $v\in\Val_{X,x}$. From now on, we assume that $v_k\to v$ weakly.
	
    Let $\chi_k\coloneqq\inf_{j\ge k}v_j\le v$. Since $v_j\to v$ weakly, given any $f\in\fm$ and any $\epsilon>0$, there exists $J\in\bZ_{>0}$ such that $v_j(f)>v(f)-\epsilon$ for any $j\ge J$. Thus
    \[
        v(f)\ge \chi_k(f)=\inf_{j\ge k}v_j(f)\ge v(f)-\epsilon
    \]
    as long as $k\ge J$. This proves $\chi_k\to v$ weakly as an increasing sequence. So we may apply Proposition \ref{prop:continuous along increasing sequences} to get $\lim_{k\to\infty}d_1(\chi_k,v)=0$. By Proposition \ref{prop:alt def for saturation}, $\fa_{\bullet,k}\coloneqq\fa(\chi_k)=\cap_{j\ge k}\fa(v_j)$ is saturated, and the proof is finished. 
\end{proof}

\subsubsection{Normalized volumes and continuity along decreasing sequences}\label{ssec:continuity along dec seq}

In this section, we assume that $R$ contains an algebraically closed field $\bk$ with $\mathrm{char}\bk=0$. As usual, we denote the singularity by $(x\in X):=(\fm\in\Spec R)$. The result is not used in the proof of the main results, though it might be of its own interest.

\begin{proposition}\label{prop:continuity along intersection with nvol bounded}
    Let $x\in X$ be a klt singularity. Let $\{\fa_{\bullet,k}\}_{k\in\bZ_{>0}}$ be a decreasing sequence in $\Fil$. Assume each $\fa_{\bullet,k}$ is of the form $\fa_{\bullet,k}=\cap_{i\in I_k}\fa_\bullet(v_i)$, where for each $k$, $\{v_i\}_{i\in I_k}\subset \Val^+_{X,x}$ is a set of valuations satisfying the following conditions.
	
    For some constants $V>0$ and $C\in\bZ_{>0}$, 
    \begin{enumerate}
        \item for any $k$ and $i\in I_k$, $\nvol_{x,X}(v_i)<V$,
		
        \item for any $k$ and $i\in I_k$, $v_i(\fm)>1/C$, and
		
        \item for any $k$, there exists $i_0\in I_k$ such that $\fa_{C\bullet}(v_{i_0})\subset \fm^\bullet$.
    \end{enumerate}
    Then
    \[
        \lim_{k\to\infty}d_1(\fa_\bullet,\fa_{\bullet,k})=0,
    \]
    where $\fa_\bullet:=\cap_k\fa_{\bullet,k}$. 
\end{proposition}

\begin{remark}
    If $I_k$ has a common element $v_0\in\Val^+_{X,x}$, then condition (3) in the proposition is automatic, since by Lemma \ref{lem:e>0 iff linearly bounded}, $\fa_\bullet(v_0)$ is linearly bounded.
\end{remark}

\begin{proof}
    By construction, $\fa_\bullet\in\Fil$ and $\e(\fa_\bullet)\ge \e(\fa_{\bullet,k})$ for any $k$. So $\lim_k \e(\fa_{\bullet,k})\in\bR_{>0}$. 
	
    For any fixed $m\in\bZ_{>0}$, we have $\cap_k \fa_{m,k}=\fa_m$. So actually there exists $K_m\in\bZ_{>0}$ such that for any $k\ge K_m$, $\fa_{m,k}=\fa_m$, since $R/\fa_m$ is Artinian. In particular, $\e(\fa_{m,k})=\e(\fa_m)$. Now for any $\epsilon>0$, by Lemma \ref{lem:uniform approximation}, there exists $M=M(C,V,\epsilon)>0$ such that for any $k$ and any $m\ge M$
    \[
        \e(\fa_{\bullet,k})\le \frac{\e(\fa_{m,k})}{m^n}\le \e(\fa_{\bullet,k})+\epsilon.
    \]
	
    Thus for $m\ge M$ and $k\ge K_m$, we have 
    \[
        \e(\fa_{\bullet,k})\le \frac{\e(\fa_m)}{m^n}\le \e(\fa_{\bullet,k})+\epsilon.
    \]
    Letting $k\to\infty$, for $m\ge M$ we get
    \[
        \lim_k \e(\fa_{\bullet,k})\le \frac{\e(\fa_m)}{m^n} \le \lim_k \e(\fa_{\bullet,k})+\epsilon.
    \]
    Taking limit with respect to $m$ and using \eqref{eqn:vol=mult}, we get $\lim_k\e(\fa_{\bullet,k})\le \e(\fa_\bullet)\le \lim_k\e(\fa_{\bullet,k})+\epsilon$. Since $\epsilon>0$ is arbitrary, we get $\e(\fa_\bullet)=\lim_k \e(\fa_{\bullet,k})$, which finishes the proof.
\end{proof}

In the proof of the proposition, we need the following lemma, which is a slight generalization of the uniform approximation proved in \cite{Blu18}. For the definition of multiplier ideals in the proof below, we also refer to \emph{ibid}.

\begin{lemma}(c.f. \cite[Proposition 3.7]{Blu18})\label{lem:uniform approximation}
    Let $\epsilon,V>0$ and $C\in\bZ_{>1}$ be constants. Then there exists $\Lambda=\Lambda(C,V,\epsilon)\in\bR_{>0}$ such that the following holds.
	
    Let $\fa_\bullet=\cap_{i\in I}\fa_\bullet(v_i)$, where $\{v_i\}_{i\in I}\subset\Val^+_{X,x}$ is a set of valuations such that 
    \begin{enumerate}
        \item $\nvol_{x,X}(v_i)<V$ for any $i\in I$,
		
        \item $v_i(\fm)>1/C$ for any $i\in I$, and
		
        \item $\fa_{C\bullet}(v_{i_0})\subset \fm^\bullet$ for some $i_0\in I$.
    \end{enumerate}
    Then for any $\lambda>\Lambda$, we have
    \[
        \e(\fa_\bullet)\le\frac{\e(\fa_\lambda)}{\lambda^n}< \e(\fa_\bullet)+\epsilon.
    \]
\end{lemma}

\begin{proof}
    Under the conditions, we have
    \[
        \lct(X;\fa_\bullet)\le C\cdot\lct(X;\fm).
    \]
    Hence by Lemma \ref{lem:lct control} for any $i\in I$,
    \[
        \frac{A_X(v_i)}{v_i(\fa_\bullet)}\le c_0\cdot \frac{\nvol_{x,X}(v_i)}{\nvol(x,X)}\cdot\lct(X;\fa_\bullet)\le \frac{c_0C\cdot V\cdot \lct(X;\fm)}{\nvol(x,X)}.
    \]
    By the valuative characterization of multiplier ideals \cite[Theorem 1.2]{BdFFU15}, we have 
    \begin{equation}\label{eqn:control J}
        \cJ(X;\lambda\cdot\fa_\bullet)\subset \fa_{\lambda-E},
    \end{equation}
    where $E:=\frac{c_0CV\lct(X;\fm)}{\nvol(x,X)}>0$ depends only on $C$ and $V$. Replacing the inclusion
    \[
        \cJ(X;m\cdot\fa_\bullet(v))\subset \fa_{m-E}
    \]
    in the proof of \cite[Proposition 3.7]{Blu18}, which follows from the assumption $A(v)\le E$, by \eqref{eqn:control J}, the proposition follows from the same argument therein.
\end{proof}

\subsection{The toric case and Newton-Okounkov bodies}\label{ssec:toric}

We deal with the toric case in this subsection. Throughout, we work over an algebraically closed field $\bk$ of characteristic $0$ and follow the notation in \cite{Ful93} for toric varieties. Our goal is to identify the subspace $(\Fil^{s,\mathrm{mon}}_{R,\fm},d_1)$ of saturated monomial filtrations on a toric singularity with a subspace of the Fr\'echet-Nykodym-Aronszajn metric on cobounded sets of the dual cone. See \cite[Section 8]{JM12}, \cite[Section 6]{KK14} and \cite[Section 8]{Blu18} for some relevant constructions. 

Let $N$ be a free abelian group of rank $n\ge 1$ and $M=N^*$ its dual. Let $\sigma\subset N_\bR=N\otimes \bR$ be a strongly convex rational polyhedral cone of maximal dimension. We get an affine toric variety $X_\sigma=\Spec R_\sigma=\Spec \bk[\sigma^\vee\cap M]$ with a unique torus invariant point $x$, where $\sigma^\vee\subset M_\bR=M\otimes\bR$ is the dual cone of $\sigma$. Let $R$ be the local ring of $X_\sigma$ at $x$ and $\fm$ its maximal ideal. 

Recall that there is a $1$-to-$1$ correspondence between toric valuations $v_u\in\Val^{\mathrm{toric}}_R$ and $u\in\sigma$, given by 
\[
    v_u(\sum_{m\in\sigma^\vee\cap M}c_m\chi^m)=\min\{\langle u,m\rangle\mid c_m\ne 0\},
\]
and $v_u$ is centered at $\fm$ if and only if $u\in\Int(\sigma)$. For $u\in\Int(\sigma)$ and $\lambda\in\bR_{>0}$, denote $H_u(\lambda)\coloneqq \{\beta\in M_\bR\mid \langle \beta,u \rangle\ge \lambda\}$. Then we have 
\[
    \fa_\lambda(v_u)=\mathrm{span}\{\chi^m\mid m\in H_u(\lambda)\cap\sigma^\vee\cap M\},
\]
and
\[
    \vol(v_u)=n!\cdot \vol(\sigma^\vee\backslash H_u(1)).
\]

We now reproduce the generalization of the above construction to the case of monomial filtrations following \cite{KK14}, see also \cite{Mus02}. An $\fm$-filtration $\fa_\bullet$ on $R$ is called \emph{monomial} if for any $\lambda\in\bR_{>0}$, $\fa_\lambda$ is a monomial ideal. Denote the set of (resp. saturated) linearly bounded monomial $\fm$-filtrations by $\Fil^\mathrm{mon}_{R,\fm}$ (resp. $\Fil^{s,\mathrm{mon}}_{R,\fm}$).

Given an $\fm$-primary monomial ideal $\fa$, the Newton polyhedron $P(\fa)$ of $\fa$ is the convex hull of $\{m\in\sigma^\vee\cap M\mid \chi^m\in\fa\}$. Moreover, we have $\e(\fa)=\vol(\sigma^\vee\backslash P(\fa))$. Recall that the \emph{Newton-Okounkov body} $P(\fa_\bullet)$ of a monomial $\fm$-filtration $\fa_\bullet\in\Fil^\mathrm{mon}_{R,\fm}$ is defined to be
\[
    P(\fa_\bullet)\coloneqq\overline{\cup_{\lambda\in\bR_{>0}}\{m/\lambda\mid m\in P(\fa_\lambda)\}}\subset\sigma^\vee,
\]
which is a convex set such that $\sigma^\vee\backslash P(\fa_\bullet)$ is bounded. It is known that the covolume of the Newton-Okounkov body computes the multiplicity of the filtration.

\begin{lemma}\cite[Theorem 6.5]{KK14}\label{lem:vol covol}
    Let $\fa_\bullet$ be a monomial $\fm$-filtration on $R$. Then we have 
    \[
    \e(\fa_\bullet)=\vol(\sigma^\vee\backslash P(\fa_\bullet)).
    \]
\end{lemma}

The next lemma gives an alternative characterization for $P(\fa_\bullet)$. 

\begin{lemma}\label{lem:toric Okounkov}
    For $\fa_\bullet\in\Fil^{\mathrm{mon}}_{R,\fm}$, $P(\fa_\bullet)=\cap_{u\in\Int(\sigma)\cap N} H_u(v_u(\fa_\bullet))$.
\end{lemma}

\begin{proof}
    For $m\in\sigma^\vee\cap M$ with $\chi^m\in\fa_\lambda$, we have $\langle u,m \rangle=v_u(\chi^m)\ge v_u(\fa_\lambda)\ge \lambda v(\fa_\bullet)$ by definition, that is, $m/\lambda\in H_u(v_u(\fa_\bullet))$. Hence $P(\fa_\bullet)\subset \cap_{u\in\Int(\sigma)\cap N} H_u(v_u(\fa_\bullet))$.

    Conversely, assume $p\in \cap_{u\in\Int(\sigma)\cap N} H_u(v_u(\fa_\bullet))$. Assume $p\notin P(\fa_\bullet)$. Since $P(\fa_\bullet)$ is convex, possibly by nudging $u$, there exist $u\in N$ and $b\in\bR_{>0}$ such that $\langle u,p \rangle<b$ but $P(\fa_\bullet)\subset H_u(b)$. Since $\sigma^\vee\backslash P(\fa_\bullet)$ is bounded, we know that $u\in \Int(\sigma)$. Now $b>\langle u,p\rangle \ge v_u(\fa_\bullet)$ by the choice of $p$. Hence there exist $\lambda\in\bR_{>0}$ and $\chi^m\in\fa_\lambda$ such that $\langle u,m \rangle=v_u(\chi^m)=v_u(\fa_\lambda)<\lambda\cdot b$. In particular, $m/\lambda\in P(\fa_\bullet)\backslash H_u(b)$, a contradiction. So $\cap_{u\in\Int(\sigma)\cap N}H_u(v_u(\fa_\bullet))\subset P(\fa_\bullet)$ and the proof is finished.
\end{proof}

Next we show that the saturation of a monomial filtration can be computed using only toric divisorial valuations.

\begin{lemma}\label{lem:toric saturation}
    A monomial filtration $\fa_\bullet\in\Fil^\mathrm{mon}_{R,\fm}$ is saturated if and only if it is of the form 
    \begin{equation}\label{eqn:toric saturation}
        \fa_\bullet=\cap_{u\in\Int(\sigma)\cap N}\fa_\bullet(v_u/a_u)
    \end{equation}
    where $a_u:\Int(\sigma)\to\bR_{>0}$ is a homogeneous function for any $u\in \Int(\sigma)\cap N$.
\end{lemma}

\begin{proof}
    If $\fa_\bullet$ is of the form \eqref{eqn:toric saturation}, then $\fa_\bullet$ is saturated by Proposition \ref{prop:alt def for saturation} since $v_u\in\Val^+_{R,\fm}$.

    Conversely, assume that $\fa_\bullet$ is saturated. Then we know that 
    \[
        \fa_\bullet\subset \fa'_\bullet\coloneqq\cap_{u\in\Int(\sigma)\cap N} \fa_\bullet(\frac{v_u}{v_u(\fa_\bullet)})
    \]
    by Lemma \ref{lem:valuative characterization for saturation}. Moreover, $v_u(\fa_\bullet')=v_u(\fa_\bullet)$ since $\fa_\bullet\subset\fa'_\bullet\subset\fa_\bullet(v_u/v_u(\fa_\bullet))$ for any $u\in \Int(\sigma)\cap N$. So $P(\fa_\bullet)=P(\fa'_\bullet)$ by Lemma \ref{lem:toric Okounkov}, and hence $\e(\fa_\bullet)=\e(\fa'_\bullet)$ by Lemma \ref{lem:vol covol}. Now $\fa'_\bullet$ is saturated by Proposition \ref{prop:alt def for saturation} again, and so by Theorem \ref{thm:equal volume iff equal saturation}, we get $\fa_\bullet=\fa'_\bullet$. 
\end{proof}

Now we can characterize the intersection of two monomial filtrations.

\begin{lemma}\label{lem:toric intersection}
    Let $I$ be a countable index set and $\fa_{\bullet,i}\in\Fil^{s,\mathrm{mon}}_{X,x}$ for $i\in I$ and assume that $\cap_{i\in I}\fa_{\bullet,i}\in\Fil^{s,\mathrm{mon}}_{X,x}$. Then 
    \begin{equation}\label{eqn:toric intersection}
        v_u(\cap_{i\in I}\fa_{\bullet,i})=\sup_{i\in I}v_u(\fa_{\bullet,i})
    \end{equation}
    for any $u\in\Int(\sigma)$. As a result, we have
    $P(\cap_{i\in I}\fa_{\bullet,i})=\cap_{i\in I} P(\fa_{\bullet,i})$.
\end{lemma}

\begin{proof}
    Since $\cap_i\fa_{\bullet,i}\subset\fa_{\bullet,i}$ for any $i$, we have $v(\cap_i\fa_{\bullet,i})\ge \sup_{i\in I}v(\fa_{\bullet,i})$ for any $v\in\Val_{X,x}$. 

    To prove the reverse inequality for toric valuations, we first treat the case where $I$ is finite. It suffices to prove $v_u(\fa_\bullet\cap\fb_\bullet)=\max\{v_u(\fa_\bullet),v_u(\fb_\bullet)\}$ for any $u\in\Int(\sigma)$ and apply induction. Denote $\fc_\bullet\coloneqq\fa_\bullet\cap\fb_\bullet$, and $a_u\coloneqq v_u(\fa_\bullet)$, $b_u\coloneqq v_u(\fb_\bullet)$, $c_u\coloneqq v_u(\fc_\bullet)$ and $c'_u\coloneqq\max\{a_u,b_u\}$. Assume to the contrary that there exists $w\in\Int(\sigma)$ such that $c_w>c'_w$. Fix $\delta\in\bR_{>0}$ such that $c_w(1-\delta)>c'_w$. Choose $p\in P(\fa_\bullet)\cap P(\fb_\bullet)$ such that $\langle p,w \rangle=\min_{q\in P(\fa_\bullet)\cap P(\fb_\bullet)}\langle q,w \rangle$. Let
    \[
        C\coloneqq \{p\}+(\sigma^\vee\backslash H_w(\delta c'_w))\subset \sigma^\vee,
    \]
    where $+$ denotes the Minkowski sum. Take $p_i\in C\cap M_\bQ$ such that $p_i\to p$ and $k_i\in\bZ_{>0}$ such that $k_ip_i\in M$. By Lemma \ref{lem:toric Okounkov}, for any $u\in\Int(\sigma)$, $\langle u,p_i \rangle \ge \langle u,p \rangle\ge c'_u$. Hence $\chi^{k_ip_i}\in\fa_{k_i}\cap\fb_{k_i}$by Lemma \ref{lem:toric saturation}, which implies
    \[
        v_w(\fa_{k_i}\cap\fb_{k_i})/k_i\le \langle w,p_i \rangle< \langle w,p \rangle+\delta c'_w = (1+\delta)c'_w,
    \]
    where the last equality follows from Lemma \ref{lem:toric Okounkov}. Letting $k_i\to\infty$ we get $c_w\le (1+\delta)c'_w$. Thus we get $c'_w<(1-\delta)c_w<(1-\delta^2)c'_w$, a contradiction. 

    The proof for the general case is quite similar. For simplicity let $I=\bZ_{>0}$. By the case where $I$ is finite, we may replace $\fa_{\bullet,i}$ by $\cap_{j\le i}\fa_{\bullet,j}$ and assume that $\fa_{\bullet,i}$ is decreasing. Assume to the contrary that there exists $w\in\Int(\sigma)$ such that $a_w>\sup_{i\in I} a_{w,i}$. Fix $\delta\in\bR_{>0}$ such that $a_w(1-4\delta)>\sup_i a_{w,i}$. For each $i\in I$, there exists $\lambda_i\in\bR_{>0}$ and $\chi^{m_i}\in\fa_{\lambda_i,i}$ such that 
    \[
        \lambda_i\cdot a_{w,i}\le v_w(\chi^{m_i})=\langle w,m_i \rangle <\lambda_i\cdot a_w(1-3\delta),
    \]
    where the first inequality follows from the assumption that $\chi^{m_i}\in\fa_{\lambda_i,i}$ and the second inequality follows from the definition $v_w(\fa_{\bullet,i})=\lim_\lambda v(\fa_{\lambda,i})/\lambda$. In particular, $m_i/\lambda_i$ lies in a bounded region and hence after passing to a subsequence, we may assume that $m_i/\lambda_i\to p\in\sigma^\vee\backslash H_u(a_u(1-2\delta))$. Now consider the region
    \[
        C\coloneqq\{p\}+(\sigma^\vee\backslash H_w(\delta a_w))\subset \sigma^\vee
    \]
    where $+$ denotes the Minkowski sum. Then there exists $p_j\in C\cap M_\bQ$ such that $p_j\to p$. Take $k_j\in\bZ_{>0}$ such that $k_jp_j\in M$. Then for any $u\in \Int(\sigma)$, we have
    $\langle u,p \rangle \le \langle u,p_j \rangle$
    for any $u\in\Int(\sigma)$ and
    \begin{equation}\label{eqn:toric upper bound}
        \langle w,p_j \rangle< a_w(1-2\delta)+\delta a_w=a_w(1-\delta).
    \end{equation}
    By the first inequality, for any $u\in\Int(\sigma)$ we have 
    \[
        \langle u,k_jp_j \rangle/k_j\ge \langle u,p \rangle=\lim_i \langle u,m_i/\lambda_i \rangle \ge \sup_i a_{u,i},
    \]
    where in the last inequality we used the fact that $\{\fa_{\bullet,i}\}$ is decreasing. Hence $\chi^{k_jp_j}\in\cap_i\fa_{k_j,i}$ by Lemma \ref{lem:toric saturation}, which implies $\langle u,p_j \rangle = \langle u,k_jp_j \rangle/k_j\ge a_u$ for any $u\in\Int(\sigma)$. In particular, we get
    \[
        a_w\le \langle w,p_j\rangle < a_w-\delta,
    \]
    a contradiction. This contradiction proves \eqref{eqn:toric intersection} and the last assertion follows from \eqref{eqn:toric intersection} and Lemma \ref{lem:toric Okounkov}.
\end{proof}

Let $\cP(\sigma^\vee)\coloneqq \{P\subsetneq \sigma^\vee \mid P \text{ is closed and convex, }\sigma^\vee\backslash P \text{ is bounded} \}$. We prove that taking the Newton-Okounkov body gives a bijection between $\Fil^{s,\mathrm{mon}}_{R,\fm}$ and $\cP(\sigma^\vee)$. Note that the following lemma also justifies the statement that ``if $P(\fa_\bullet)=P(\fb_\bullet)$, then $\fa_\bullet$ and $\fb_\bullet$ should be regarded as equisingular'' in \cite[Section 8]{JM12}.

\begin{lemma}\label{lem:toric bijection betweem Fil and Okounkov}
    $P:\Fil^{s,\mathrm{mon}}_{R,\fm}\to \cP(\sigma^\vee)$ is a bijection, whose inverse is given by $P\mapsto \fa_\bullet(P)$, where
    \begin{equation}\label{eqn:inverse}
        \fa_\lambda(P)\coloneqq \mathrm{span} \{\chi^m, m\in\sigma^\vee\cap M\mid \langle m,u\rangle \ge \lambda\cdot \inf_{p\in P}\langle p,u\rangle \text{ for any } u\in\Int(\sigma)\},
    \end{equation}
\end{lemma}

\begin{proof}
    Let $\fa_\bullet\in\Fil^{s,\mathrm{mon}}_{X,x}$. It is known that $P(\fa_\bullet)$ is convex (e.g. \cite{KK14}). Fix $u\in\Int(\sigma)$.  Since $\fa_\bullet$ is linearly bounded, by definition there exists $c'\in\bR_{>0}$ such that $\fa_\bullet\subset \fm^{\bullet/c'}$, which implies $v_u(\fa_\bullet)\ge v_u(\fm^{\bullet/c'})=c'v_u(\fm)\eqqcolon c>0$. By Lemma \ref{lem:toric saturation}, we get $P(\fa_\bullet)\subset H_u(c)\subsetneq \sigma^\vee$. Fix $d\in\bZ_{>0}$ such that $\fm^d\subset\fa_1$. Then for any $m\in\bZ_{>0}$ we have $\fm^{dm}\subset\fa_m$. By Proposition \ref{prop:saturation and Rees}, there exist finitely many $u_i\in\Int(\sigma)$, $i=1,\ldots,k$ such that $\widetilde{\fm^{d\bullet}}=\cap_{i=1}^k \fa_\bullet(v_{u_i})$. By Lemma \ref{lem:toric Okounkov} and Lemma \ref{lem:toric intersection}, we have 
    \[
          P(\fa_\bullet)\supset P(\fm^{d\bullet)}= P(\widetilde{\fm^{d\bullet}})=\cap_{i=1}^k P(\fa_\bullet(v_{u_i})=\cap_{i=1}^k H_{u_i}(1).
    \]
    Thus $\sigma^\vee\backslash P(\fa_\bullet)\subset \cup_{i=1}^k \sigma^\vee\backslash H_{u_i}(1)$, which is bounded. This proves $P:\Fil^{s,\mathrm{mon}}_{X,x}\to \cP(\sigma^\vee)$.

    Now let $P\in\cP(\sigma^\vee)$. We need to prove that $\fa_\bullet(P)$ is a linearly bounded monomial $\fm$-filtration. It is easy to see from \eqref{eqn:inverse} that $\fa_\lambda(P)$ is an $\fm$-primary monomial ideal. Condition (1) of Definition \ref{defn:filtration} follows immediately from \eqref{eqn:inverse} and (3) follows fron the linearity of the function $\langle \bullet,u \rangle$. To prove (2), it suffices to note that $\sigma^\vee\cap M$ is discrete, and hence for any $\lambda\in\bR_{>0}$ and $u\in\Int(\sigma)$, there exists $\epsilon=\epsilon(\lambda,u)$ such that $H_u(\lambda-\epsilon)\backslash H_u(\lambda)\cap \sigma^\vee\cap M=\emptyset$.

    It is then routine to check the four inclusions. We include the proof for the reader's convenience.

    $\fa_\bullet\subset \fa_\bullet(P(\fa_\bullet))$. If $\chi^m\in\fa_\lambda$, then $m/\lambda\in P(\fa_\bullet)$, and hence $\chi^m\in\fa_\lambda(P(\fa_\bullet))$ by \eqref{eqn:inverse}.

    $\fa_\bullet(P(\fa_\bullet))\subset \fa_\bullet$. By Lemma \ref{lem:toric Okounkov}, $\langle p,u \rangle \ge v_u(\fa_\bullet)$ for any $p\in P(\fa_\bullet)$ and $u\in\Int(\sigma)$. Hence if $\chi^m\in\fa_\lambda(P(\fa_\bullet))$, then 
    \[
        \langle m,u \rangle \ge \lambda\cdot \inf_{p\in P(\fa_\bullet)} \langle p,u \rangle \ge \lambda v_u(\fa_\bullet)
    \]
    for any $u\in\Int(\sigma)$. By Lemma \ref{lem:toric saturation}, we know $\chi^m\in\fa_\lambda$. 

    $P\subset P(\fa_\bullet(P))$. For $p\in P$, take $p_i\in P\cap M_\bQ$ such that $p_i\to p$ and $k_i\in\bZ_{>0}$ such that $m_i\coloneqq k_ip_i\in P\cap M$. Then by \eqref{eqn:inverse}, $\chi^{m_i}\in \fa_\bullet(P)$, and hence $p=\lim_i m_i/k_i\in P(\fa_\bullet(P))$. 

    $P(\fa_\bullet(P))\subset P$. If $p\in P(\fa_\bullet(P))$, then by definition, there exists $m_i\in\sigma^\vee\cap M$ and $\lambda\in\bR_{>0}$ such that $\chi^{m_i}\in\fa_{\lambda_i}$ and $m_i/\lambda_i\to p$. By \eqref{eqn:inverse}, this implies $\langle p,u \rangle \ge \inf_{q\in P} \langle q,u \rangle$ for any $u\in\Int(\sigma)$. Applying the same argument as in the proof of Lemma \ref{lem:toric Okounkov}, we know that $p\in P$.
\end{proof}

Recall that the symmetric difference metric (also called the Fr\'echet-Nikodym-Aronszayn distance) $d$ on $\cP(\sigma^\vee)$ is defined by 
\[
    d(P_1,P_2)\coloneqq \vol(P_1\Delta P_2),
\]
where $P_1\Delta P_2=P_1\backslash P_2\cup P_2\backslash P_1$ is the symmetric difference. We are now ready to prove

\begin{proposition}\label{prop:toric isometry}
    The map $P:(\Fil^s_{R,\fm},d_1)\to (P(\sigma^\vee,d)$, $\fa_\bullet\mapsto P(\fa_\bullet)$ is an isometry.
\end{proposition}

\begin{proof}
    By Lemma \ref{lem:toric bijection betweem Fil and Okounkov}, $P$ is a bijection. 
    
    For simplicity, denote $\covol(P)\coloneqq \vol(\sigma^\vee\backslash P)$ below. For $\fa_\bullet,\fb_\bullet\in\Fil^s_{R,\fm}$, we have
    \begin{align*}
        d_1(\fa_\bullet,\fb_\bullet)=&2\e(\fa_\bullet\cap\fb_\bullet)-\e(\fa_\bullet)-\e(\fb_\bullet)\\
         =&2\covol(P(\fa_\bullet\cap\fb_\bullet))-\covol(P(\fa_\bullet))-\covol(P(\fb_\bullet))\\
         =&2\covol(P(\fa_\bullet)\cap P(\fb_\bullet))-\covol(P(\fa_\bullet))-\covol(P(\fb_\bullet))\\
         =&\vol(P(\fa_\bullet)\backslash P(\fb_\bullet))+\vol(P(\fb_\bullet)\backslash P(\fa_\bullet))\\
         =&\vol(P(\fa_\bullet)\Delta P(\fb_\bullet))=d(P(\fa_\bullet),P(\fb_\bullet)),
    \end{align*}
    where the second equality follows from Lemma \ref{lem:toric Okounkov} and the third from \ref{lem:toric intersection}. The proof is finished. 
\end{proof}

\section{Supnorm metrics on the space of filtrations}\label{sec:d_infty}

In this section, we study the basic properties of $d_\infty$, prove the local Lipschitz continuity of log canonical thresholds, and compare several different topologies on the filtration spaces.

\subsection{Supnorm metrics and homogenization}

Throughout this subsection, we will use the language of norms. Recall that $\cN$ is the set of linearly bounded $\fm$-norms on $R$.

\begin{definition}\label{defn:d_infty metric}
    Fix a norm $\rho\in\cN$. We define a function $d_{\infty,\rho}:\cN\times\cN\to\bR_{\ge 0}$ by
    \[
        d_{\infty,\rho}(\chi,\chi'):=\limsup_{\lambda\to\infty}\sup_{\rho(f)\ge\lambda}\left\{\frac{|\chi(f)-\chi'(f)|}{\rho(f)}\right\}.
    \]
    We will denote $d_{\infty,\chi_0}=:d_\infty$.
\end{definition}

It is not hard to see that the $d_\infty$-topology is independent of the choice of the reference norm.

\begin{lemma}\label{lem:d_infty is p-metric}
    For any $\rho\in\cN$, $d_{\infty,\rho}$ defines a pseudometric on $\cN$. Moreover, for any $\rho,\rho'\in\cN$, there exists $M\in\bR_{>0}$ such that $M^{-1}d_{\infty,\rho}\le d_{\infty,\rho'}\le Md_{\infty,\rho}$. 
\end{lemma}

\begin{proof}
    By definition, $d_{\infty,\rho}$ is symmetric and non-negative for any $\rho$. Given $\chi,\chi',\chi''\in\cN$, since $|\chi(f)-\chi''(f)|\le |\chi(f)-\chi'(f)|+|\chi'(f)-\chi''(f)|$, the triangle inequality holds. This proves that $d_{\infty,\rho}$ is a pseudometric.
	
    The second claim follows from the fact that given $\rho,\rho'\in\cN$, there exists $M\gg 0$ such that $M^{-1}\rho\le \rho'\le M\rho$.
\end{proof}

On $\cN^h$, we do not need to take $\limsup$ in the definition of $d_\infty$.

\begin{lemma}\label{lem:d_infty for homogeneous}
    If $\chi,\chi'\in\cN^h$, then
    \[
        d_\infty(\chi,\chi')=\sup_{f\in\fm}\frac{|\chi(f)-\chi(f')|}{\chi_0(f)}.
    \]
\end{lemma}

\begin{proof}
    For any $\epsilon>0$, take $g\in\fm$ such that $\frac{|\chi(g)-\chi'(g)|}{\chi_0(g)}>\sup_{f\in\fm}\frac{|\chi(f)-\chi(f')|}{\chi_0(f)}-\epsilon\eqqcolon d-\epsilon$. Denote $a:=\chi_0(g)>0$. Since $\chi$, $\chi'$ and $\chi_0$ are homogeneous, we have
    \begin{align*}
        \sup_{\chi_0(f)=ka}\frac{|\chi(f)-\chi(f')|}{\chi_0(f)}
        \ge &\frac{|\chi(g^k)-\chi'(g^k)|}{\chi_0(g^k)}
        =\frac{|\chi(g)-\chi'(g)|}{\chi_0(g)}>d-\epsilon.
    \end{align*}
    So 
    \[
        d_\infty(\chi,\chi')=\limsup_{\lambda\to\infty} \sup_{\chi_0(f)\ge \lambda}\frac{|\chi(f)-\chi(f')|}{\chi_0(f)}\ge d-\epsilon.
    \]
    Thus $d_\infty(\chi.\chi')\ge d$. Since the reverse inequality is obvious, this finishes the proof.
\end{proof}

\begin{lemma}\label{lem:d_infty of hom goes down}
    Let $\chi,\chi'\in\cN$, then $d_\infty(\widehat\chi,\widehat{\chi}')\le d_\infty(\chi,\chi')$.
\end{lemma}

\begin{proof}
    Denote $d:=d_\infty(\widehat{\chi},\widehat{\chi}')$. By Lemma \ref{lem:d_infty for homogeneous}, we may assume that $d=\sup\frac{\widehat\chi(f)-\widehat\chi'(f)}{\chi_0(f)}$. For any $\epsilon>0$, there exists $f\in \fm$ such that 
    \begin{align*}
        (d-\epsilon)\chi_0(f)<\widehat\chi(f)-\widehat{\chi}'(f)=\lim_{k\to\infty}\frac{\chi(f^k)}{k}-\lim_{k\to\infty}\frac{\chi'(f^k)}{k}.
    \end{align*}
    Thus for $k\gg 0$ we have 
    \[
        \chi(f^k)-\chi'(f^k)\ge (d-2\epsilon)k\chi_0(f)=(d-2\epsilon)\chi_0(f^k)
    \]
    since $\chi_0$ is homogeneous. Similar to the proof of Lemma \ref{lem:d_infty for homogeneous} we get $d_\infty(\chi,\chi')\ge d$. 
\end{proof}

The following lemma shows that $(\cN^h,d_\infty)$ is complete.

\begin{lemma}\label{lem:d_infty is complete}
    Let $\{\chi_k\}_{k\in\bZ_{> 0}}$ be a Cauchy sequence in $(\cN^h,d_\infty)$. Then there exists $\chi\in\cN^h$ such that 
    \[ 
        \lim_{k\to\infty} d_\infty(\chi_k,\chi)=0.
    \]
\end{lemma}

\begin{proof}
    Given any $\epsilon>0$, by assumption there exists $k\in\bZ_{>0}$ such that for any $k,k'>K$, $d_\infty(\chi_k,\chi_{k'})<\epsilon$. For any $f\in\fm$ and $k,k'>K$, by Lemma \ref{lem:d_infty for homogeneous} we know that 
    \[
        |\chi_k(f)-\chi_{k'}(f)|\le \chi_0(f)d_\infty(\chi,\chi')<\chi_0(f)\cdot\epsilon.
    \]
    Hence $\{\chi_k(f)\}$ is a Cauchy sequence and we may define $\chi(f):=\lim_{k\to\infty}\chi_k(f)$. 
    It is elementary to verify that $\chi\in\cN^h$. Then we can apply Lemma \ref{lem:d_infty for homogeneous} to get $\lim_{k\to\infty} d_\infty(\chi_k,\chi)=0$.
\end{proof}

\subsection{Local Lipschitz continuity of log canonical thresholds}\label{ssec:effective}

In this subsection we prove that the log canonical threshold is locally Lipschitz continuous with respect to the $d_\infty$-topology.

\begin{proposition}\label{prop:lip of lct}
    Let $\chi\in\cN$. For any $\epsilon>0$, there exists $\delta=\delta(\chi,\epsilon)>0\in\bR_{>0}$ such that for any $\chi'\in\cN$ with $d_\infty(\chi,\chi')<\delta$, we have $|\lct(\chi')-\lct(\chi)|\le \epsilon$.
\end{proposition}

\begin{proof}
    Since $\lct(\chi)=\lct(\widehat\chi)$ for any $\chi\in\cN$, by Lemma \ref{lem:d_infty of hom goes down}, we may assume that $\chi,\chi'\in\cN^h$. 
    Fix $c=c(\chi):=\inf_{f\in\fm}\frac{\chi(f)}{\chi_0(f)}>0$ and denote $a:=\lct(\chi)>0$.

    Assume that $d_\infty(\chi,\chi')<\delta$ for some $\delta\in(0,c)$. By Lemma \ref{lem:d_infty for homogeneous}, we know that for any $f\in R$, 
    \[
        \frac{\chi'(f)}{\chi_0(f)}\ge \frac{\chi(f)}{\chi_0(f)}-\delta\ge c-\delta.
    \]
    Now we compute 
    \begin{equation}\label{eq:ratio}\begin{aligned}
        \frac{v(f)}{\chi'(f)}=&\frac{v(f)}{\chi(f)}\left(1+\frac{\chi(f)-\chi'(f)}{\chi'(f)}\right)\\
        =&\frac{v(f)}{\chi(f)}\left(1+\frac{\chi_0(f)}{\chi'(f)}\frac{\chi(f)-\chi'(f)}{\chi_0(f)}\right)\\
        \le&\frac{v(f)}{\chi(f)}\left(1+\frac{\delta}{c-\delta}\right)
    \end{aligned}\end{equation}
    For $\delta\in(0,c/2)$, by Lemma \ref{lem:evaluation} we get that $v(\fa_\bullet(\chi'))\le(1+2\delta/c)v(\fa_\bullet(\chi))$ for any $v$. Let 
    \[
        \delta:=\min\{\frac{c}{2},\frac{c\epsilon}{a}\}>0,
    \]
    which depends only on $\epsilon$ and $\chi$. Then for any $\chi'\in\Fil^s$ with $d_\infty(\chi,\chi')<\delta$ and any $v\in\Val^{<+\infty}_{X,x}$, we have
    \[
        \frac{A(v)}{v(\fa_\bullet(\chi'))}> (1-\frac{\delta}{2c})\frac{A(v)}{v(\fa_\bullet(\chi))}\ge (1-\frac{\delta}{2c})\lct(\chi)=a-\frac{a}{2c}\delta>a-\epsilon.
    \]
    Taking infimum among all $v$ gives $\lct(\chi')\ge a-\epsilon$.

    Interchanging $\chi$ and $\chi'$ in \eqref{eq:ratio}, by Lemma \ref{lem:evaluation} we get $v(\fa_\bullet(\chi))\le (1+\delta/c)v(\fa_\bullet(\chi'))$. By \cite[Theorem B.1]{Blu18}, there exists $v_0\in\Val^{<+\infty}_{X,x}$ computing the lct of $\fa_\bullet(\chi)$, thus
    \[
        \lct(\chi')\le\frac{A(v_0)}{v_0(\fa_\bullet(\chi'))}\le (1+\frac{\delta}{c})\frac{A(v_0)}{v_0(\fa_\bullet(\chi))}=(1+\frac{\delta}{c})a\le a+\epsilon.
    \]
    The proof is finished.
\end{proof}

\subsection{Comparison of different topologies on $\Fil^s$}\label{ssec:comparison}

Now we would like to compare the topologies on the space $\Fil^s$. \footnote{Since the product topology is in general not first countable and hence not obviously sequential, we avoid using the term \emph{finer} for them, but only compare convergence for sequences.} First we show that a $d_1$-converging sequence converges in the $+$-topology.

\begin{lemma}\label{lem:d_1 finer than +}
    Let $\fa_\bullet$,$\fa_{\bullet,k}\in\Fil^s$ for $k\in\bZ_{>0}$. If $\lim_k d_1(\fa_{\bullet,k},\fa_\bullet)=0$, then $\fa_{\bullet,k}\to \fa_\bullet$ in the $+$-topology.
\end{lemma}

\begin{proof}
    We claim that if $\fa_{\bullet,k}\subset\fb_{\bullet,k}$ satisfies $\lim_{k\to\infty}d_1(\fa_{\bullet,k},\fb_{\bullet,k})=0$, then for any $v\in\Val^+_{X,x}$, we have $\lim_{k\to\infty} (v(\fa_{\bullet,k})-v(\fb_{\bullet,k}))=0$.
	
    For any $v\in\Val^+_{X,x}$, applying the claim to $\fa_{\bullet,k},\fa_{\bullet,k}\cap\fa_\bullet$ and $\fa_{\bullet,k}\cap\fa_\bullet,\fa_\bullet$ separately yields
    \begin{align*}
        \lim_{k\to\infty} |v(\fa_{\bullet,k})-v(\fa_\bullet)|\le \lim_{k\to\infty} |v(\fa_{\bullet,k})-v(\fa_{\bullet,k}\cap\fa_\bullet)|+\lim_{k\to\infty}|v(\fa_{\bullet,k}\cap\fa_\bullet)-v(\fa_\bullet)|=0.
    \end{align*}
	
    Now it remains to prove the claim, which follows essentially from the same argument as \cite[Proposition 3.12]{BLQ22}. We include a sketch for the reader's convenience. Assume to the contrary that there exist $\epsilon>0$ and $v\in\Val^+_{X,x}$ such that $v(\fa_{\bullet,k})>v(\fb_{\bullet,k})+\epsilon$ for any $k$. Then for $k\in\bZ_{>0}$, we may choose $\ell_k\in\bZ_{>0}$ and $f_k\in\fb_{\ell_k,k}$ such that 
    \[
        \frac{v(f_k)}{\ell_k}=\frac{v(\fb_{\ell_k,k})}{\ell_k}<v(\fa_{\bullet,k})-\epsilon.
    \]
	
    For $m\in\bZ_{>0}$, consider the morphism $\phi_{k,m}:R\to R/\fa_{m\ell_k,k}$, $g\mapsto f_k^mg+\fa_{m\ell_k,k}$. Since $f_k^m\in\fb_{\ell_k,k}^m\subset\fb_{m\ell_k,k}$, the image is contained in $\fb_{m\ell_k,k}/\fa_{m\ell_k,k}$. Moreover, one can check that $\ker(\phi_{k,m})\subset \fa_{m\ell_k\epsilon}(v)$. Hence 
    \[
        \ell(\fb_{m\ell_k,k}/\fa_{m\ell_k,k})\ge \ell(R/\ker(\phi_{k,m})\ge \ell(R/\fa_{m\ell_k\epsilon}(v).
    \]
    Letting $m\to\infty$ we get $\e(\fa_{\bullet,k})-\e(\fb_{\bullet,k})\ge \epsilon^n\vol(v)>0$, which is a contradiction.
\end{proof}

Next we compare the two metrics $d_1$ and $d_\infty$.

\begin{lemma}\label{lem:d_infty finer than d_1}
    Let $\fa_\bullet,\fb_\bullet\in\Fil$ be filtrations satisfying $\fm^{C\bullet}\subset\fa_\bullet$ and $\fm^{C\bullet}\subset\fb_\bullet$ for some $C\in\bR_{>1}$. Then there exists a constant $M:=M(C)>0$ such that 
    \begin{equation}\label{eqn:d_1 d_infty comparison}
        d_1(\fa_\bullet,\fb_\bullet)\le M\cdot d_\infty(\fa_\bullet,\fb_\bullet).
    \end{equation}
    As a consequence, the $d_\infty$-topology is finer than the $d_1$-topology. 
\end{lemma}

\begin{proof}
    By Lemma \ref{lem:d_infty of hom goes down} we may assume $\fa_\bullet,\fb_\bullet\in\Fil^h$. Since $d_1(\fa_\bullet,\fb_\bullet)=d_1(\fa_\bullet,\fc_\bullet)+d_1(\fb_\bullet,\fc_\bullet)$ and $d_\infty(\fa_\bullet,\fb_\bullet)\ge\max\{d_\infty(\fa_\bullet,\fc_\bullet),d_\infty(\fb_\bullet,\fc_\bullet)\}$, where $\fc_\bullet\coloneqq\fa_\bullet\cap\fb_\bullet$, we may further assume that $\fa_\bullet\subset\fb_\bullet$ (but replace $M$ by $2M$).
	
    Denote $d:=d_\infty(\fa_\bullet,\fb_\bullet)$. Then we know that 
    \[
        0\le\frac{\chi_\fb(f)}{\chi_\fa(f)}-1=\frac{\chi_0(f)}{\chi_\fa(f)}\cdot\frac{\chi_\fb(f)-\chi_\fa(f)}{\chi_0(f)}\le Cd,
    \]
    which implies
    \[
        \fa_\bullet\subset\fb_\bullet\subset \fa_{\bullet/(1+Cd)}.
    \]
    Hence we get 
    \[
        \e(\fa_\bullet)\ge\e(\fb_\bullet)\ge (1+Cd)^{-n}\e(\fa_\bullet),
    \]
    that is
    \begin{align*}
        d_1(\fa_\bullet,\fb_\bullet)=&\e(\fa_\bullet)-\e(\fb_\bullet)\\
        \le&\left(1-(1+Cd)^{-n}\right)\e(\fa_\bullet)\\
        \le&nCd\cdot \e(\fa_\bullet)\\
        \le&nC^{n+1}\e(\fm)\cdot d,
    \end{align*}
    where in the last inequality we used the fact that $(1+x)^{-n}\ge 1-nx$ for $x\ge 0$. This proves \eqref{eqn:d_1 d_infty comparison} with $M\coloneqq 2nC^{n+1}\e(\fm)$.
	
    Given $\fa_\bullet\in\Fil$, we may take $C\in\bZ_{>1}$ such that $\fm^C\subset \fa_1$ which implies $\fm^{C\bullet}\subset\fa_\bullet$. Now for $0<\epsilon<1/2C$ and $\fb_\bullet\in\Fil$ with $d_\infty(\fa_\bullet,\fb_\bullet)<\epsilon$, by definition we have $\fm^{2C\bullet}\subset\fb_\bullet$. Hence the second assertion follows from \eqref{eqn:d_1 d_infty comparison}.
\end{proof}

To summarize, we have the following diagram for the convergence of sequences
\begin{equation*}
d_\infty\text{-convergence}\left\{\begin{aligned}
&\implies\text{weak convergence}\\ 
&\implies d_1\text{-convergence} \implies +\text{-convergence} \implies \text{coefficientwise convergence}
\end{aligned}\right.
\end{equation*}
Here the upper line is obvious; the first two arrows of the lower line follow from Lemma \ref{lem:d_infty finer than d_1} and Lemma \ref{lem:d_1 finer than +} respectively; and the last arrow follows from the fact that $\DivVal_{X,x}\subset \Val^+_{X,x}$. It is not hard to see that the upper line is not an equivalence. We do not have any counterexamples for the other reverse arrows.

\section{Proof of the main results}\label{sec:proof}

\begin{proof}[Proof of Theorem \ref{thm:d_1 metric}]
    (1) By Proposition \ref{prop:d_1 is a pseudo metric}, $d_1$ is a pseudometric on $\Fil$. 
	
    By \cite[Corollary 3.17]{BLQ22}, $\widetilde\fa_\bullet=\widetilde\fb_\bullet$ if and only if $\e(\fa_\bullet)=\e(\fb_\bullet)=\e(\fa_\bullet\cap\fb_\bullet)$, which holds if and only if $d_1(\fa_\bullet,\fb_\bullet)=0$ by definition. This proves $\Fil^s=\Fil/\sim_1$. 
	
    (2) Since $\fa_{\bullet,0}\ne\fa_{\bullet,1}$, we may assume that $\fa_{\bullet,0}\not\subseteq\fa_{\bullet,1}$. Then there exist $f\in R$ and $\lambda\in\bR_{>0}$ such that $f\in\fa_{\lambda,0}\backslash\fa_{\lambda,1}$, and $f\notin\fa_{\lambda',0}$ for $\lambda'>\lambda$. Take $t\in(0,1)$ and $\mu,\nu$ with $(1-t)\mu+t\nu=\lambda$. We claim that $f\notin\fa_{\lambda,t}$. Indeed, if $\mu\le \lambda$, then $\nu\ge\lambda$, so $f\notin\fa_{\nu,1}$ since $f\notin\fa_{\lambda,1}$. If $\mu>\lambda$, then by the choice of $f$ and $\lambda$ we have $f\notin\fa_{\mu,0}$. In either case, $f\notin\fa_{\mu,0}\cap\fa_{\nu,1}$. Thus $f\notin\fa_{\lambda,t}$ and hence $\fa_{\lambda,t}\ne \fa_{\lambda,0}$.  
    
    If $\fa_{\bullet,1}\nsubseteq\fa_{\bullet,0}$, then we can repeat the above construction to show that $\fa_{\bullet,t}\ne\fa_{\bullet,1}$ and hence $\fa_{\bullet,t}$ is between $\fa_{\bullet,0}$ and $\fa_{\bullet,1}$ for any $t\in(0,1)$. 
	
    Otherwise, $\fa_{\bullet,1}\subsetneq \fa_{\bullet,0}$. We claim that there exist $\lambda\in\bR_{>0}$ and $\epsilon>0$ such that $\fa_{\lambda+\epsilon,0}\supsetneq \fa_{\lambda,1}$ and $\fa_{\lambda-\epsilon,1}\supsetneq\fa_{\lambda,1}$. 
    Indeed, if all the jumping numbers of $\fa_{\bullet,0}$ and $\fa_{\bullet,1}$ coincide then the claim is clear. Otherwise, choose a jumping number $a$ of $\fa_{\bullet,1}$ with $\fa_{a,1}\subsetneq \fa_{a,0}$ which is not a jumping number of $\fa_{\bullet,0}$. Then there exists $a'>a$ such that $\fa_{a',0}=\fa_{a,0}$. Let $\lambda:=(a+a')/2$ and $\epsilon:=\lambda-a$, then 
    \[
        \fa_{\lambda-\epsilon,1}=\fa_{a,1}\supsetneq\fa_{\lambda,1}
    \]
    since $a$ is a jumping number of $\fa_{\bullet,1}$, and 
    \[
        \fa_{\lambda+\epsilon,0}=\fa_{a',0}= \fa_{a,0}\supsetneq\fa_{a,1}\supset\fa_{\lambda,1}
    \]
    by construction. So we get
    \[
        \fa_{\lambda,1}\subsetneq\fa_{\lambda+\epsilon,0}\cap\fa_{\lambda-\epsilon,1}\subset \sum_{\mu+\nu=2\lambda}\fa_{\mu,0}\cap\fa_{\nu,1}=\fa_{\lambda,1/2},
    \]
    and hence $\fa_{\bullet,1/2}$ is between $\fa_{\bullet,0}$ and $\fa_{\bullet,1}$.  
    Now it follows from Proposition \ref{prop:distance minimizing} that $(\Fil^s,d_1)$ is convex.

    Otherwise, $\fa_{\lambda,0}\subsetneq\fa_{\lambda,1}$ for any $\lambda\in\bR_{>0}$.
	
	
    When $R$ contains a field, by Lemma \ref{lem:continuity of geodesic}, the curve $\gamma:[0,1]\to \Fil^s$, $t\mapsto \widetilde\fa_{\bullet,t}$ is continuous. It computes the distance between $\fa_{\bullet,0}$ and $\fa_{\bullet,1}$ by Proposition \ref{prop:distance minimizing}. Hence $\Fil^s$ is a geodesic metric space. 
\end{proof}

\begin{proof}[Proof of Theorem \ref{thm:toric metric}]
    The first assertion follows from Proposition \ref{prop:toric isometry}.

    Note that by Lemma \ref{lem:rooftop}, $(\Fil^{s,\mathrm{mon}}_{R,\fm},d_1)$ is a rooftop metric space (\cite[Definition 5.2]{Xia19}. It is easy to verify that a decreasing (resp. increasing) Cauchy sequence $\fa_{\bullet,k}$ in $\Fil^{s,\mathrm{mon}}_{R,\fm}$ satisfies $\cap_k \fa_{\bullet,k}\in\Fil^{s,\mathrm{mon}}_{R,\fm}$ (resp. $\cup_k \fa_{\bullet,k}\in\Fil^{s,\mathrm{mon}}_{R,\fm}$). Therefore by Proposition \ref{prop:continuous along increasing sequences}, Lemma \ref{lem:toric intersection} and \cite[Proposition 5.4]{Xia19}, $\Fil^{s,\mathrm{mon}}_{R,\fm}$ is complete.
\end{proof}

\begin{proof}[Proof of Theorem \ref{thm:d_infty metric}]
    Assertion (1) follows from Lemma \ref{lem:d_infty is p-metric}. 

    By Lemma \ref{lem:d_infty for homogeneous}, $d_\infty$ is non-degenerate when restricted to $\cN^h$. Moreover, by Lemma \ref{lem:d_infty is complete}, the metric space $(\cN^h,d_\infty)$ is complete. Thus assertion (2) is proved.
\end{proof}

\begin{proof}[Proof of Theorem \ref{thm:lattice}]
    (1) By Definition-Lemma \ref{deflem:arbitrary intersection} and Definition-Lemma \ref{deflem:join of filtrations}, the set of $\fm$-filtrations is a lattice. It is easy to see that the intersection and join of two linearly bounded $\fm$-filtrations are again linearly bounded, hence $(\Fil,\subset,\cap,\vee)$ is a sublattice of all $\fm$-filtrations. Since $\fa\subset\fm$, we know that $\fa^{\lceil\lambda\rceil}\subset\fm^{\lfloor\lambda\rfloor}$ and the filtration $\fa^\bullet$ is linearly bounded. Now the map $\cI_\fm\to\Fil$, $\fa\mapsto\fa^\bullet$ is clearly injective, and by Example \ref{eg:ideal join}, it is a join morphism. 
	
    (2) It follows from Definition-Lemma \ref{deflem:arbitrary intersection}, Definition-Lemma \ref{deflem:saturated join} and Proposition \ref{prop:distributivity} that $(\Fil^s,\subset,\cap,\vee_s)$ is a distributive lattice. By Proposition \ref{prop:lattice of Fil^s}, the saturation from $\Fil$ to $\Fil^s$ is a join morphism, which is clearly surjective. By Example \ref{eg:saturated ideal join}, the map $\cI^{\ic}_{\fm}\to\Fil^s$, $\fa\mapsto\widetilde{\fa^\bullet}$ is injective, and is a join morphism as the composition of two join morphisms. 
\end{proof}

\subsection{Continuity properties of log canonical thresholds}

In this section, we assume $(R,\fm)$ is a klt singularity over a field $\bk$ of characteristic $0$. 

\begin{proposition}\label{prop:lct is weakly lsc}
    Let $\fb_{\bullet,k}\in\Fil^s$ be a sequence of filtrations that converges weakly to $\fa_\bullet\in\Fil^s$. Then
    \begin{equation}\label{eqn:lct is lsc}
        \lct(\fa_\bullet)\le \liminf_{k\to\infty}\lct(\fb_{\bullet,k}).
    \end{equation}
\end{proposition}

\begin{proof}

    Consider the sequence $\fa_{\bullet,k}:=\cap_{m=k}^\infty \fb_{\bullet,m}\subset\fb_{\bullet,k}$ as in the proof of Proposition \ref{prop:modified d_1-converging seq}. Then the argument shows that $\fa_{\bullet,k}$ is an increasing sequence in $\Fil$ that converges to $\fa_\bullet$ weakly. Note that for any $k\in\bZ_{>0}$, $\lct(\fa_{\bullet,k})\le \lct(\fb_{\bullet,k})$, so $\lim_{k\to\infty} \lct(\fa_{\bullet,k})\le \liminf_{k\to\infty}\lct(\fb_{\bullet,k})$. Thus to prove \eqref{eqn:lct is lsc}, it suffices to show that 
    \begin{equation}\label{eqn:lct along increasing seq}
        \lct(\fa_\bullet)\le \lim_{k\to\infty}\lct(\fa_{\bullet,k}).
    \end{equation}
    Assume to the contrary that $\lct(\fa_\bullet)>\lim_k\lct(\fa_{\bullet,k})=:c$. Fix $\epsilon>0$ such that $\lct(\fa_\bullet)>c+\epsilon$. Since $\lct(\fa_\bullet)=\lim_\lambda \lambda\cdot\lct(\fa_\lambda)$, there exists $\Lambda\in\bR_{>0}$ such that $\lambda\cdot\lct(\fa_\lambda)>c+\epsilon$ for any $\lambda\ge\Lambda$. As $\fa_{\bullet,k}$ converges to $\fa_\bullet$ weakly, that is, $\fa_\Lambda=\cup_k \fa_{\Lambda,k}$, there exists $k_1\in\bZ_{>0}$ such that $\fa_\Lambda=\fa_{\Lambda,k_1}$.In particular, 
    \begin{equation}\label{eqn:lower bound}
        \Lambda\cdot\lct(\fa_{\Lambda,k_1})>c+\epsilon.
    \end{equation}  
	
    Now take $k_2\in\bZ_{>0}$ such that $\lct(\fa_{\bullet,k})=\sup_\lambda \lambda\cdot\lct(\fa_{\lambda,k})<c+\epsilon$ for any $k\ge k_2$. For $k\ge\max\{k_1,k_2\}$, combining this with \eqref{eqn:lower bound} we get
    \[
        c+\epsilon<\Lambda\cdot\lct(\fa_{\Lambda,k_1})
        =\Lambda\cdot\lct(\fa_{\Lambda,k})
        \le\sup_\lambda\lambda\cdot\lct(\fa_{\lambda,k})<c+\epsilon,
    \]
    where the equality follows from $\fa_{\Lambda,k_1}=\fa_\Lambda=\fa_{\Lambda,k}$. This is a contradiction, so \eqref{eqn:lct along increasing seq} holds and the proof is finished. 
\end{proof}

\begin{remark}
    Note that in the above proof, one always has $\lct(\fa_\bullet)\ge \lim_k\lct(\fa_{\bullet,k})$, hence equality holds indeed.
\end{remark}

\begin{proof}[Proof of Theorem \ref{thm:lct}]
    Assertion (1) follows immediately from Proposition \ref{prop:lct is weakly lsc}. 
	
    To see assertion (2), assume that $\fa_{\bullet,k}\to\fa_\bullet\in\Fil^s$ in the $+$-topology. By \cite[Theorem B.1]{Blu18}, there exists $v\in\Val^{<+\infty}_{X,x}$ such that $\lct(\fa_\bullet)=\frac{A_X(v)}{v(\fa_\bullet)}$. By definition, we have $\lim_k v(\fa_{\bullet,k})=v(\fa_\bullet)$, hence
    \[
        \limsup_{k\to\infty} \lct(\fa_{\bullet,k})= \limsup_{k\to\infty} \inf_{w\in\Val^{<+\infty}_{X,x}}\frac{A_X(w)}{w(\fa_{\bullet,k})} \le \lim_{k\to\infty} \frac{A_X(v)}{v(\fa_{\bullet,k})}=\frac{A_X(v)}{v(\fa_\bullet)}=\lct(\fa_\bullet).
    \]
	
    Assertion (3) follows from Proposition \ref{prop:lip of lct}. The proof is finished.
\end{proof}

\section{Discussions}\label{sec:discussion}

In this section, we discuss the relation to previous work, especially the work in the global case. We also propose some questions related to the main results and provide some toy examples.


\subsection{Relation to global results}
As was mentioned earlier, our definition for the metric $d_1$ is inspired by the Darvas metric $d_1$ on $\cH(X,\omega)$, the space of smooth K\"ahler potentials of a compact K\"ahler manifold $(X,\omega)$, introduced in \cite{Dar15}, and is a local analogue of the metric $d_1$ on $R=R(X, -rK_X)$, the section ring of a Fano manifold $X$, introduced in \cite{BJ21}. We now explain an approach to equip the global object with a structure in our setting via the cone construction. We work over the complex numbers $\bC$ for simplicity.

Let $(V,L)$ be a smooth polarized variety, that is, $V$ is a smooth projective variety and $L$ is an ample line bundle on $V$. Let $R(V,L)\coloneqq\oplus_{m=0}^\infty H^0(V,mL)=\oplus_{m=0}^\infty R_m$, then we know that $R_0=\bC$. We may replace $L$ by $rL$ for some $r\in\bZ_{>0}$ such that $R$ is generated by $R_1$ as an $R_0$-algebra, and in the sequel, we will always assume $L$ is positive enough such that the above condition is satisfied. In this case, the cone $C(V,L)\coloneqq\Spec R(V,L)$ has a normal, isolated singularity at the vertex $0$ defined by $R_+(V,L)\coloneqq\oplus_{m>0} R_m$. 

Denote the localization of $R(V,L)$ at $R_+(V,L)$ by $R$. Then we have a normal local domain $(R,\fm)$ which is essentially of finite type over $\bC$. Now the canonical saturated filtration $\widetilde{\fm^\bullet}=\fa_\bullet(\ord_V)$, where $V\subset \Bl_0 C(V,L)\to C(V,L)$ is identified as the exceptional divisor of the blow-up of the cone point. Any $v\in\Val_V$ defines a graded filtration $\cF^\bullet_v$ on $R$ by 
\[
\cF^\lambda_vR_m\coloneqq\{s\in R_m\mid v(s)\ge \lambda\},
\]
which is linearly bounded when $v\in\Val^{<+\infty}_V$. This is not an $\fm$-filtration, but we could draw a ray from the canonical filtration $\widetilde{\fm^\bullet}$, by 
\[
s=\sum s_l\in\fa_{\lambda,t}\Leftrightarrow \min\left\{\frac{\ell+tv(s_\ell)}{1+t}\mid s_\ell\ne 0\right\}\ge \lambda,
\]
for $t\in[0,+\infty)$. Denote the corresponding norm by $\chi_t$, then $\chi_0=\ord_V$ by definition, and for any $s\in R$, we have $\lim_{t\to +\infty}\chi_t(s)=v(s)$. For any $t\in[0,+\infty)$, by definition, we have $\fm^{\lceil(1+t)\lambda\rceil}\subset\fa_{\lambda,t}$. Moreover, since $\cF_v$ is linearly bounded, there exists $C\in\bR_{>0}$ such that for any $s\in R_l$, we have $v(s)\le Cl$. So $s\in\fa_{\lambda,t}$ will force $\ord_V(s)\ge \frac{(1+t)\lambda}{1+C}$, that is, $\fa_{\lambda,t}\subset \fm^{\lfloor \lambda(1+t)/(1+C)\rfloor}$. This shows that $\fa_{\bullet,t}\in\Fil$. \footnote{Indeed $\fa_{\bullet,t}=\fa_\bullet(v_t)$ for some $v\in\Val^{<+\infty}_{X,x}$.} 

Thus we may view $v\in\Val^{<+\infty}_V$ as an element sitting on the boundary of $\Fil^s_{R,\fm}$. 
We do not go further in this direction. See also \cite{Fin23} and the references therein for some related results.

\subsection{Examples}

In this section, we consider several examples. Note that we only work with linearly bounded filtrations, or equivalently, filtrations with positive multiplicity.

\subsubsection{The metric space $(\Fil^s,d_1)$ in low dimensions}

\begin{example}\label{eg:DVR}
    If $R$ is a DVR, then $\Fil^s_{R,\fm}\cong \bR_{>0}$. Let $\pi\in\fm$ be a uniformizer. For $\fa_\bullet\in\Fil^s_{R,\fm}$, let $c\coloneqq \ord_\pi(\fa_\bullet)\in\bR_{>0}$. Since $\ord_\pi$ is the only divisorial valuation up to scaling, for $\lambda\in\bR_{>0}$ we have
    \[
        \fa_\lambda=\{f\in R\mid \ord_\pi(f)\ge \lambda\cdot c\}=(\pi^{\lceil c\lambda\rceil})=\fm^{\lceil c\lambda\rceil},
    \]
    that is, $\fa_\bullet=\fm^{c\bullet}$. Thus we have a bijection of sets $\phi:\Fil^s\to\bR_{>0}$, $\fm^{c\bullet}\mapsto c$. It is easy to see that $\e(\fm^{c\bullet})=c\cdot\e(\fm^\bullet)=c$, hence the bijection $\phi$ is an (order-reversing) isometry. 
\end{example}

However, the metric space $(\Fil^s_{R,\fm},d_1)$ becomes very complicated as the dimension goes up. When $R=\bC[\![x,y]\!]$, by \cite{FJ04}, the valuation space $\Val_{R,\fm}$ is a cone over a tree, and $\Val^+_{R,\fm}$ is a subset of $\Fil^s$.

\subsubsection{The metric on spaces of valuations.}
By \eqref{eqn:inclusion of valuation spaces}, Theorem \ref{thm:d_1 metric} gives a metric on the valuation space $\Val^+_{X,x}$, defined by
\[
d_1(v_0,v_1):= 2\e(\fa_\bullet(v_0)\cap\fa_\bullet(v_1))-\vol(v_0)-\vol(v_1).
\]

\begin{example}
    Let $R=\bk[\![x,y]\!]$. Let $v_0$ be the weighted blowup along $x,y$ with weights $(2,1)$ and $v_1$ the weighted blowup with weights $(1,2)$. It is easy to see that $\vol(v_i)=2$ for $i=0,1$. Computing the Newton-Okounkov body of the intersection, we get $\e(\fa_\bullet(v_0)\cap\fa_\bullet(v_1))=8/3$. Hence $d_1(v_0,v_1)=4/3$. Similarly, let $v_0\supl$ and $v_1\supl$ be the weighted blowup with weights $(\ell,1)$ and $(1,\ell)$ respectively, then we can compute
    \[
        d_1(v_0\supl,v_1\supl)=\frac{4\ell^2}{\ell+1}-2\ell=\frac{2\ell^2-2\ell}{\ell+1}
    \]
	
    In particular, $\lim_{\ell\to\infty}d_1(v_0\supl,v_1\supl)=+\infty$, and hence the metric $d_1$ induces a topology on the dual complex of the simple normal crossing pair $(X=\Spec R,L_0+L_1)$, where $L_0$ and $L_1$ are the lines $(x=0)$ and $(y=0)$ respectively, which is different from the Euclidean topology.
\end{example}

\subsubsection{The metrics on spaces of ideals.}
Similar to Theorem \ref{thm:d_1 metric}, one can also define a pseudometric $d_1$ on the set $\cI_\fm$ by
\[
d_1(\fa,\fb)=2\e(\fa\cap\fb)-\e(\fa)-\e(\fb).
\]
By Rees' Theorem, the Hausdorff quotient of $(\cI,d_1)$ can be identified with $(\cI^{\ic}_\fm,d_1)$. Note that with this distance, the inclusion $\cI^{\ic}_\fm\hookrightarrow\Fil^s$ is not an embedding but just a contraction map, since in general we only have $\e(\fa\cap\fb)\ge \e(\widetilde{\fa^\bullet}\cap\widetilde{\fb^\bullet})$. 

Note that the spaces $\Fil^s$ and $\Val_{X,x}^+$ both allow an action by $\bR_{\ge 0}$. If we consider the set of $\fm$-primary ideals ``with coefficients'', that is, the set 
\[
    \cI_{\fm,A}:=\{\fa^c\mid \fa\in\cI_\fm,\ c\in A\},
\]
where $A=\bQ_{\ge 0}$ or $\bR_{\ge 0}$. Clearly, the following metric makes sense.
\[
    d_{1,A}(\fa^c,\fb^d):=2\e(\fa^{c\bullet}\cap\fb^{d\bullet})-c^n\e(\fa)-d^n\e(\fb).
\]

\noindent\textbf{Warning}. These are in general two different metrics on $\cI_\fm$.

\subsection{The structure of finitely generated filtrations}

Recall that a filtration $\fa_\bullet$ is \emph{finitely generated}, or \emph{Noetherian}, if the \emph{associated graded ring} $\gr_{\fa_\bullet}R\coloneqq \oplus_{\lambda\in\bR_{\ge 0}}\fa_\lambda/\fa_{>\lambda}$ is a finitely generated $\kappa$-algebra. 

The finite generation of valuation ideals has been extensively studied since \cite{BCHM10}. See \cites{LXZ22,XZ25} for recent progress based on the theory of special complements.

As another application of Proposition \ref{prop:continuous along increasing sequences}, we show that finitely generated filtrations are dense in $\Fil_{R,\fm}$, which is the local analogue of \cite[Corollary 3.19]{BJ21}.

\begin{proposition}\label{prop:fg density}
    The subset of finitely generated filtrations is dense in $\Fil_{R,\fm}$ with respect to $d_1$.
\end{proposition}

\begin{proof}
    The following construction is a local analogue of \cite[Definition-Lemma 3.56]{Xu25}. For $\fa\in\Fil_{R,\fm}$ and $k\in\bZ_{>0}$, let

    \[
        \fa_{\lambda,k}\coloneqq\left\{\begin{aligned}
            &\fa_\lambda, &\text{ if } \lambda\le k,\\
            &\sum_{\vec{\mu}} \fa_{\mu_1}\cdot \ldots \cdot \fa_{\mu_s}, &\text{ otherwise,}
        \end{aligned}
        \right.
    \]
    where the sum is taken over all vectors $\vec\mu=(\mu_1,\ldots,\mu_s)\in \bR_{>0}^s$ satisfying $\sum_i \mu_i\ge\lambda$ and $\mu_i\le k$ for any $1,\ldots,s$. Note that the sum is finite since the set of jumping numbers of $\fa_\bullet$ contained in $(0,k]$ is finite.

    It is easy to check that $\fa_{\bullet,k}\in\Fil_{R,\fm}$. Moreover, $\gr_{\fa_{\bullet,k}} R$ is generated by elements of degree at most $k$, and hence is finitely generated again by the finiteness of jumping numbers.

    Clearly $\{\fa_{\bullet,k}\}$ is an increasing sequence and $\cup_k \fa_{\bullet,k}=\fa_\bullet$. So by Proposition \ref{prop:continuous along increasing sequences}, we know that $d_1(\fa_\bullet,\fa_{\bullet,k})\to 0$, which finishes the proof.
\end{proof}

\begin{remark}\label{rmk:valuation fg}
    In general, the saturation of a finitely generated filtration is not finitely generated. Hence, the above proposition says nothing about the structure or distribution of finitely generated \emph{saturated} filtrations, which remains rather mysterious.
    
    Indeed, in \cite{LX24}, Y. Liu and Xu proposed several conjectures regarding the structure of Koll\'ar valuations, which are special classes of finitely generated filtrations. See \cites{ABB+23,LXZ22,Pen25} for some special cases.
    
    Moreover, in view of \cite[Section 4]{LXZ22} and \cite[Section 4.3]{XZ25}, it seems that the specialty of quasi-monomial valuations inside a single dual complex is related to the geodesics. 
\end{remark}

\subsection{Questions}

We propose a few questions related to the metric space $(\Fil^s,d_1)$.

\subsubsection{Completeness under the Darvas metric}\label{ssec:completeness}

Recall that by \cite{Xia19}, in order to prove the completeness of the metric space $(\Fil^s,d_1)$, it is enough to check that a (bounded) monotonic sequence has a limit. By Proposition \ref{prop:continuous along increasing sequences}, this is true for increasing sequences. So it is natural to ask the following question:

\begin{question}\label{q:continuity along decreasing nets}
    Let $\{\fa_{\bullet,k}\}_{k\in\bZ_{>0}}$ be a decreasing sequence in $\Fil^s_{R,\fm}$ such that $\fa_\bullet\coloneqq\cap \fa_{\bullet,k}\in\Fil^s_{R,\fm}$,
    do we have 
    \[
	\lim_{k\to\infty} d_1(\fa_\bullet,\fa_{\bullet,k})=0?
    \]
\end{question}

We remark that convergence along decreasing sequences or nets has been studied in the complex analytic setting \cite[Proposition 6.11]{Dar17} and the global non-archimedean setting \cite[Theorem 12.5]{BJ22}. The answer to the above question is likely no, but we do not have a counterexample.

\subsubsection{More structural questions}

We know in Example \ref{eg:DVR} that $\Fil^s$ is Euclidean when $R$ is a DVR. This is never the case in higher dimensions. Moreover, geodesics between two incomparable filtrations $\fa_\bullet$ and $\fb_\bullet$ are never unique, since the geodesic between $\fa_\bullet$ and $\fb_\bullet$ is distinct from the one obtained from combining the piecewise geodesics connecting $\fa_\bullet$, $\fa_\bullet\cap\fb_\bullet$ and $\fb_\bullet$. It might be interesting to ask the relationship between the topological or metric properties of $\Fil^s_{R,\fm}$ and the algebraic properties of the ring $R$. For example,

\begin{question}
    Is $\Fil^s$ always contractible? If not, is the contractibility related to the algebraic properties of $R$?
\end{question}

\noindent\textbf{Statement for conflict of interest.} The author states that there is no conflict of interest.

\end{document}